\documentclass[3p,times]{elsarticle}

\usepackage{graphicx}
\usepackage{psfrag}
\usepackage[dvipsnames]{xcolor}

\usepackage{amsmath}
\usepackage{amsthm}
\usepackage{amsfonts}
\usepackage{dsfont}
\usepackage{rotating}
\usepackage{multirow}
\usepackage{epstopdf}
\usepackage{newtxmath}
\usepackage{bm}
\usepackage{bbm}
\usepackage{comment}
\usepackage{booktabs}
\usepackage{accents}
\usepackage{empheq}

\DeclareMathAlphabet\mathbfcal{OMS}{cmsy}{b}{n}  

\newcommand{\R}{\mathds{R}}

\renewcommand{\vv}{\mathbf{v}}
\newcommand{\xx}{\mathbf{x}}
\newcommand{\nn}{\mathbf{n}}
\newcommand{\q}{\mathbf{q}}

\newcommand{\dt}{{\Delta t}}

\newcommand{\Q}{\mathbf{Q}}

\newcommand{\I}{\mathbf{I}}

\newcommand{\f}{\mathbf{f}}
\newcommand{\g}{\mathbf{g}}
\newcommand{\B}{\mathbf{B}}

\newcommand{\Fr}{\textnormal{Fr}}

\newcommand{\CFL}{\textnormal{CFL}}

\newcommand{\bzero}{\mathbf{0}}
\newcommand{\w}{\mathbf{w}}

\newcommand{\Vop}{\vmathbb{V}} 
\newcommand{\Cop}{\vmathbb{C}} 

\newtheorem{theorem}{Theorem}

\newcommand{\atanpol}{\text{atan2}}
\newcommand{\erf}{\text{erf}}

\newcommand{\diff}{\text{d}}
\newcommand{\Rey}{\text{Re}}

\newcommand{\xv}{\bm{x}}

\newcommand{\xs}{x}
\newcommand{\ys}{y}
\newcommand{\vs}{\varv}

\newcommand{\hh}{h}

\renewcommand{\P} {{P}}            
\newcommand  {\E} {{e}}            
\newcommand  {\V} {{k}}            
\newcommand{\xvP}{\xv_{\P}}        

\newcommand{\hP}{\hh_{\P}}

\newcommand{\mP}{\ABS{\P}}

\newcommand{\Pset}{{E}}    
\newcommand{\Eset}{{E}}    
\newcommand{\Vset}{{E}}    

\newcommand{\NMB}{N}
\newcommand{\NP}{\NMB_{\Pset}}   
\newcommand{\NE}{\NMB_{\Eset}}   
\newcommand{\NV}{\NMB_{\Vset}}   

\newcommand{\abs}    [1]{|#1|}

\newcommand{\ABS}    [1]{\left|#1\right|}

\newcommand{\norE} {\bm{n}_{\E}}

\newcommand{\norPE}{\bm{n}_{\P,\E}}

\newcommand{\HONE}  {H^1}
\newcommand{\VS}[1] {V^{#1}}
\newcommand{\PS}[1]{\mathbbm{P}_{#1}}
\newcommand{\CS}[1] {C^{#1}}

\newcommand{\Vh}  {\VS{\hh}}

\newcommand{\vsh}{{\vs_{\hh}}}

\newcommand{\XV}{\xv_{\V}}

\newcommand{\REAL}{\mathds{R}}
\newcommand{\NDOF}{N^{\textrm{dof}}_{\P}}
\newcommand{\dof}{\textrm{dof}}
\newcommand{\proj}{\Pi^\nabla_{\P,k}}
\newcommand{\projL}{\Pi^0_{\P,k}}
\newcommand{\projLm}{\Pi^0_{\P,k-1}}

\begin{document}
	
	\begin{frontmatter}
		

		\title{A new family of semi-implicit Finite Volume / Virtual Element methods for incompressible flows on unstructured meshes}
		
		\author[dmi]{Walter Boscheri$^*$}
		\ead{walter.boscheri@unife.it}
		\cortext[cor1]{Corresponding authors}
        
        \author[deps]{Andrea Chiozzi$^*$}
        \ead{andrea.chiozzi@unife.it}
        
        \author[dmi]{Michele Giuliano Carlino}
        \ead{michelegiuliano.carlino@unife.it}
        
        \author[deps]{Giulia Bertaglia}
        \ead{giulia.bertaglia@unife.it}
        
		\address[dmi]{Department of Mathematics and Computer Science, University of Ferrara, Via Niccol\`o Machiavelli 30, 44121 Ferrara, Italy}
		\address[deps]{Department of Environmental and Prevention Sciences, University of	Ferrara, C.so Ercole I d'Este 32, 44121 Ferrara, Italy}
%
\begin{abstract}
We introduce a new family of high order accurate semi-implicit schemes for the solution of non-linear hyperbolic partial differential equations (PDE) on unstructured polygonal meshes. The time discretization is based on a splitting between explicit and implicit terms that may arise either from the multi-scale nature of the governing equations, which involve both slow and fast scales, or in the context of projection methods, where the numerical solution is projected onto the physically meaningful solution manifold. We propose to use a high order finite volume (FV) scheme for the explicit terms, hence ensuring conservation property and robustness across shock waves, while the virtual element method (VEM) is employed to deal with the discretization of the implicit terms, which typically requires an elliptic problem to be solved. The numerical solution is then transferred via suitable L$_2$ projection operators from the FV to the VEM solution space and vice-versa. High order time accuracy is then achieved using the semi-implicit IMEX Runge-Kutta schemes, and the novel schemes are proven to be asymptotic preserving (AP) and well-balanced (WB). As representative models, we choose the shallow water equations (SWE), thus handling multiple time scales characterized by a different Froude number, and the incompressible Navier-Stokes equations (INS), which are solved at the aid of a projection method to satisfy the solenoidal constraint of the velocity field. Furthermore, an implicit discretization for the viscous terms is also devised for the INS model, which is based on the VEM technique. Consequently, the CFL-type stability condition on the maximum admissible time step is based only on the fluid velocity and not on the celerity nor on the viscous eigenvalues. A large suite of test cases demonstrate the accuracy and the capabilities of the new family of schemes to solve relevant benchmarks in the field of incompressible fluids.
\end{abstract}
%
\begin{keyword}
semi-implicit schemes \sep
Finite Volume  \sep
Virtual Element Method \sep
high order in space and time \sep
Asymptotic Preserving \sep
Incompressible flows	
\end{keyword}
\end{frontmatter}


\section{Introduction} \label{sec.intro}
Incompressible flows are mathematically described by non-linear systems of hyperbolic conservation laws, and they cover a wide range of physical phenomena such as environmental, geophysical and meteorological flows, as well as the dynamics of mechanical processes like turbo-machinery. In this article, we consider two different models: i) the shallow water equations (SWE), which are commonly used to simulate storm surges, inundations, dam breaks and river floods; ii) the incompressible Navier-Stokes (INS) equations, which model atmospheric flows, flows around wings and even pressurized pipe flows. Due to the relevant physical applications, there is a great interest in the numerical solution of these models, which however gives rise to some complexities. In particular, the shallow water equations are concerned with slow and fast time scales, which are linked to the non-linear convective and pressure terms, respectively. Moreover, the energy equation in the incompressible Navier-Stokes model reduces to a constraint on the velocity field that must satisfy a divergence-free condition. The use of unstructured grids is very convenient for real-world applications that deal, for example, with ocean bathymetry, wing and turbine shapes, and river morphology, which should be approximated with high accuracy. Moreover, obtaining high order of accuracy also in time represents an important goal in order to achieve accurate results for unsteady problems.

From the numerical viewpoint, the multi-scale nature of the system poses severe restrictions on the maximum admissible time step in order to make the numerical scheme capable of capturing the fast waves. For explicit time discretization, this stability condition can lead to extremely small time steps that ultimately prevent the method to be effectively adopted for real applications. Furthermore, the amount of numerical dissipation introduced by the scheme dramatically augments, spoiling the accuracy of the solution. An effective strategy is based on flux splitting techniques, which separate the slow and fast scales. An explicit time discretization is employed for the slow scales that typically correspond to the non-linear convective terms and are, therefore, related to the fluid velocity, while an implicit time stepping is retained for the fast scales which involve the acoustic waves like the sound speed or the celerity. On the other hand, the divergence-free constraint is rather complicated to be respected on general unstructured meshes with high order of accuracy, therefore pressure-correction methods are typically used, which firstly compute an intermediate velocity field that is not solenoidal, and then a correction is performed in order to project the velocity onto the divergence-free manifold. Furthermore, the viscous terms in the INS model are endowed with a parabolic time step restriction, which is also responsible of very small time steps in viscous dominated flows, which might be conveniently treated implicitly, although this is not trivial on unstructured meshes. 

By all means, this class of numerical methods belongs to the category of implicit-explicit (IMEX) \cite{AscRuuSpi,BP2017,PR_IMEX,BosFil2016} or semi-implicit time schemes \cite{Casulli1990,Casulli1992,StagDG_Dumbser2013}, which imply the solution of an algebraic system for the unknown physical quantity that is discretized implicitly (for example the pressure or the velocity field). This system does normally involve an elliptic equation to be solved on the entire computational domain. The effectiveness of IMEX schemes have been demonstrated in several contexts, including low Mach compressible flows \cite{ParkMunz2005,BosPar2021,SICNS22}, magnetized plasma flows \cite{3splitMHD}, pipe flows \cite{Ioriatti2018,Ioriatti_SWE2019}, as well as free surface \cite{Casulli1999,TumoloBonaventuraRestelli} and atmospheric \cite{TumoloBonaventura,ORLANDO2023115124} applications. If Cartesian meshes are adopted for the spatial discretization \cite{Fambri2017,BosPar2021}, finite difference schemes are likely to be employed for the implicit terms, and the high order extension is straightforward by enlarging the stencil of the finite difference operators. Besides, finite differences can also be used on orthogonal unstructured meshes like Voronoi tessellations \cite{Voronoi,Voronoi-DivFree,ADERFSE} up to second order of accuracy. In this case, the solution of the implicit algebraic system can be carried out quite easily. However, when general unstructured meshes pave the computational domain, the solution of the algebraic system becomes more difficult, since the spatial discretization might be rather complicated. If finite volume methods are used, the inversion of the reconstruction operator should be embedded in the system solver, leading to additional difficulties. Discontinuous Galerkin methods \cite{TavelliSWE2014,TavelliIncNS,ORLANDO2022111653} represent a clever option, because they simultaneously provide high accuracy and compactness of the stencil. Nevertheless, finite volume methods can deal with very general control volumes and they exhibit, by construction, conservation and shock capturing properties, which makes such schemes suitable for the discretization of the explicit slow scale terms, like the non-linear convective terms. In addition, since the implicit system is typically concerned with an elliptic equation on the fast scale quantity, like the pressure, finite element schemes could also be used on both quadrilateral and triangular cells in two space dimensions (2D) or cubic and tetrahedral control volumes in three-dimension (3D). 
This is the reasoning behind the recent emergence of a new class of hybrid finite volume/finite element methods \cite{Hybrid0,Hybrid1,Hybrid2,Busto_SWE2022}, which combine the robustness of the explicit finite volume solver with the versatility of the finite element method. Due to the classical finite element solver, these numerical methods are restricted to simplex meshes in 2D/3D -- i.e. meshes that are topologically dual of triangulations. A first attempt of using a hybrid scheme on general polygonal meshes has been forwarded in \cite{SIFVDG}, where the implicit pressure system is solved at the aid of a discontinuous Galerkin approach that acts on a sub-triangulation of the polygonal tessellation.

The goal of this paper is to design a genuinely hybrid finite volume/finite element scheme on arbitrary shaped grids in 2D. The originality of the proposed scheme lies in the usage of the Virtual Element Method (VEM) to handle the solution of the implicit elliptic equation. The Virtual Element Method (VEM), originally developed in \cite{vem2,vem1,vem4,Beirao_Brezzi_Marini_2016}, is a stabilized Galerkin formulation, originating from mimetic finite difference schemes \cite{Lipnikov2014}, to solve partial differential equations on very general polygonal and polyhedral meshes that overcomes the many difficulties and challenges that are associated
with polygonal finite element formulations. The VEM represents a generalization of the finite element method in which the explicit knowledge of the basis functions
is not needed. Instead, the VEM provides projection operators
of the basis functions onto polynomial spaces of arbitrary degree that allow to discretize and to suitably approximate the bilinear form and the continuous linear
functional deriving from the variational formulation. The discretized bilinear form is conveniently decomposed as the sum of a consistent term, which ensures polynomial consistency, and a correction term, which guarantees stability. The VEM has been extensively proposed for the solution of elliptic problems in many fields of mechanics, such as two- and three-dimensional linear elasticity \cite{vem3, Gain2014, Artioli2020, Dassi2020},  nearly incompressible elasticity \cite{Chi2017}, linear elastodynamics \cite{Antonietti2021}, elastic problems with singularities and discontinuities \cite{BenvenutiChiozzi2019,BenvenutiChiozzi2022}, fracture mechanics \cite{Nguyen2018,Hussein2019}, flow problems in porous media \cite{Borio2021, Borio2022}. Recently, there has been an interest in addressing also steady state partial differential equations (PDE), as for instance the steady Navier-Stokes equations \cite{vem6,Beirao-Lovadina2018,Chernov2021,WANG202163,Beirao_Stokes22,Antonietti22}. Noticeably, in this application, the divergence-free constraint is directly imposed in the definition of the virtual basis, avoiding any projection (or correction) operator. Nevertheless, unsteady problems have still to be faced recurring to the VEM. 

In this paper, we introduce an innovative hybrid finite volume/virtual element solver for hyperbolic PDE on general unstructured polygonal meshes in the context of incompressible flows, solving both the SWE and the INS model. The FV solver is used to treat the explicit terms, while the VEM strategy enters the scheme to solve the implicit algebraic system for the pressure. Furthermore, in the incompressible Navier-Stokes equations, the VEM approach is actually used twice: firstly to implicitly discretize the viscous terms \cite{boscheri2021efficient,SIINS22}, and secondly to perform the projection step ensuring a divergence-free velocity field. Remarkably, the entire scheme works on genuinely polygonal tessellations. Finally, high order in time is ensured by applying a semi-implicit IMEX time stepping technique, which has already been employed also in the context of compressible fluids \cite{SICNS22}.  

The rest of the paper is organized as follows. In Section \ref{sec.pde} the governing equations are presented, studying the multi-scale character of the mathematical models which leads to the flux splitting for SWE or the projection strategy for INS. Section \ref{sec.numscheme} is devoted to the presentation of the numerical scheme, with details on the explicit FV solver and the implicit VEM. The asymptotic and the well-balance property of the scheme are also studied in this section. A large suite of test problems is proposed in Section \ref{sec.numtest}, showing the accuracy and the robustness of the new family of schemes while solving academic benchmarks for SWE and INS systems. Finally, we provide some concluding remarks and an outlook on possible future works in Section \ref{sec.concl}.

\section{Mathematical model} \label{sec.pde}
Incompressible flows are mathematically described by hyperbolic systems of PDE, which can be cast into the following formulation:
\begin{equation}
	\label{eqn.PDE}
	\frac{\partial \Q}{\partial t} + \nabla \cdot \mathbb{F}(\Q) + \mathbb{B} \cdot \nabla \Q = \mathbf{0}\,, \qquad 
	\xx \in \Omega \subset \mathds{R}^d\,, \quad 
	t \in \mathds{R}_0^+\,, \quad 
	\Q \in \Omega_{\Q} \subset \mathds{R}^{\gamma}\,,
\end{equation}
with $d=2$ representing the number of space dimensions. The vector $\xx=(x,y)$ identifies the spatial coordinates in the physical domain $\Omega$ and $t$ is the time. The vector of evolutionary variables $\Q=(q_1,q_2,\ldots,q_{\gamma})$ is defined in the space of the admissible states $\Omega_{\Q}\subset \mathds{R}^{\gamma}$, and $\mathbb{F}(\Q)=\left( \f(\Q),\g(\Q) \right)$ denotes the conservative non-linear flux tensor ($\f$ in $x-$direction, $\g$ in $y-$direction), while $\mathbb{B}=\left(\B_x,\B_y\right)$ contains the purely non-conservative part of the system written in block-matrix notation.  

A widespread model is given by the Incompressible Navier--Stokes (INS) equations, which can be cast into form \eqref{eqn.PDE} with absence of non-conservative terms by defining
\begin{equation}
	\label{eqn.INS}
	\Q = \left( 0, \vv \right), \qquad 
	\mathbb{F}(\Q) = \left( \vv, \, \vv \otimes \vv - \nu \nabla \vv + p\I \right), \qquad \mathbb{B} \cdot \nabla \Q = \bzero,
\end{equation}
introducing the identity matrix $\I$. The velocity field is $\vv(\xx,t) = (u,v)$, with the components along $x-$ and $y-$direction, and  $p(\xx,t)=P(\xx,t)/\rho$ denotes the normalized fluid pressure, where $P(\xx,t)$ is the physical pressure and $\rho$ is the constant fluid density. The kinematic viscosity coefficient is given by $\nu=\mu/\rho$ with $\mu$ being the dynamic viscosity of the fluid. The first equation of the INS model \eqref{eqn.INS} is the divergence-free constraint on the velocity field, which corresponds to the energy equation in the low Mach limit of compressible flows \cite{BosPar2021}.

For situations in which the vertical velocity field is negligible with respect to the horizontal one, incompressible flows are governed by the Shallow Water Equations (SWE). By neglecting friction forces, the SWE fits the general formalism \eqref{eqn.PDE} by setting
\begin{equation}
	\label{eqn.SWE}
	\Q = \left( \eta, \, \q , \, b\right), \qquad 
	\mathbb{F}(\Q) = \left( \q, \, \vv \otimes \q, \, 0 \right), \qquad 
	\mathbb{B}\cdot \nabla \Q = \left( 0,\, gH \nabla \eta, \, 0 \right).
\end{equation}
Here, $\eta(\xx,t) \geq 0$ is the free surface elevation of the fluid and $b(\xx)$ defines the bottom bathymetry, measured with respect to a reference plane, thus the total water depth is given by $H(\xx,t)=\eta(\xx,t)-b(\xx)\geq 0$. The flow discharge is $\q(\xx,t)=H(\xx,t) \vv(\xx,t)$, while $g$ is the gravity acceleration. The SWE model \eqref{eqn.SWE} is written in terms of conservative variables and it can reduce to a conservation law ($\mathbb{B} \cdot \nabla \Q = \bzero$) in the case of flat bottom topography, hence $\eta=H$ and $\mathbb{F}(\Q) = \left( \q, \, \vv \otimes \q + gH^2/2, \, 0 \right)$.

These mathematical models can describe phenomena characterized by different time scales (i.e., multi-scale regimes), leading to the resolution of problems that are notoriously complicated to approximate numerically due to the associated stiffness. For example, low Froude number $\Fr=|\vv|/\sqrt{gH}$ flows in the SWE are characterized by a rather low convective speed $|\vv|$ compared to the acoustic-gravity celerity $\sqrt{gH}$. For the INS equations, the viscous dominated flows induce a stiffness, which is then responsible of a severe CFL-type stability restriction on the time step for explicit time discretizations. To identify the terms responsible for the stiffness, it is convenient to write the governing equations in dimensionless form \cite{BosPar2021,SIFVDG}. Let the variables of 
\eqref{eqn.INS} and \eqref{eqn.SWE} be rescaled as
\begin{equation}
	\label{eqn.Qadim}
	\tilde{\xx}=\frac{\xx}{L_0}, \quad 
	\tilde{t}=\frac{t}{T_0}, \quad 
	\tilde{\vv}=\frac{\vv}{U_0}, \quad 
	\tilde{\eta}=\frac{\eta}{H_0}, \quad 
	\tilde{H}=\frac{H}{H_0}, \quad 
	\tilde{b}=\frac{b}{H_0}, \quad 
	\tilde{\q}=\frac{\q}{U_0 H_0}, \quad
	\tilde{p}=\frac{p L_0}{\nu U_0},
\end{equation}
where $L_0$, $T_0$, $U_0=L_0/T_0$, $H_0$ are the characteristic length, time, velocity and depth, respectively, of the problem under consideration. By omitting the tilde symbol for ease of reading, the dimensionless form of the INS equations \eqref{eqn.INS} writes
\begin{equation}
	\label{eqn.INS_adim1}
	\Q = \left( \mathbf{0}, \, \vv\right), \qquad 
	\mathbb{F}(\Q) = \left(\vv,\, \vv \otimes \vv - \frac1{\Rey} \left(\nabla \vv - p\I\right)\right),
\end{equation}
where the Reynolds number $\Rey=U_0L_0/\nu$ appears in front of the viscosity and pressure terms of the momentum equation. 

Analogously, we can derive the dimensionless form of the SWE system \eqref{eqn.SWE} that is given by
\begin{equation}
	\label{eqn.SWE_adim}
	\Q = \left( \eta, \, \q , \, b\right), \qquad 
	\mathbb{F}(\Q) = \left( \q,\, \vv \otimes \q, \, \mathbf{0} \right), \qquad 
	\mathbb{B}\cdot \nabla \Q = \left( 0,\, \frac{H}{\Fr^2} \nabla \eta, \, 0 \right),
\end{equation}
where it can be observed the presence of the Froude number $\Fr = U_0/\sqrt{gH_0}$ in the pressure gradient term. 

In both cases, the dimensionless systems can be characterized by a stiffness parameter $\varepsilon$, which is the reference Reynolds number $\varepsilon=\Rey$ in \eqref{eqn.INS_adim1} or the reference Froude number $\varepsilon=\Fr$ in \eqref{eqn.SWE_adim}. The asymptotic limit of the models can consequently be studied using the asymptotic parameter $\varepsilon$, as it will be carried out in Section \ref{ssec.AP}.

%

\subsection{Flux splitting}
\label{sec.splitting}
In order to develop an efficient and accurate numerical method to solve multi-scale problems as those discussed in the previous section, it is convenient to treat the governing equations with a flux splitting approach, which is based on the separation of the different time scales, fast and slow, in terms of the flux terms (either conservative or non-conservative). In the literature, this strategy has already been followed to treat incompressible flow models \cite{Casulli1990,SIFVDG,SIINS22,ToroIJNMF22,Vater2018} but also compressible fluids \cite{Toro_Vazquez12,BosPar2021,Dumbser_Casulli16}. 
The idea behind this technique consists in dividing the system of governing equations into two sub-systems, where one contains the fluxes dependent on the scaling parameter $\varepsilon$ and the other one does not. This procedure basically allows to highlight the terms that will be discretized explicitly in time with respect to those that, due to the $\varepsilon-$dependence, will encounter an implicit discretization. 

The flux splitting applied to a generic system of the type \eqref{eqn.PDE} reads
\begin{equation}
\label{eqn.PDE_split}
\frac{\partial \Q}{\partial t} + \nabla \cdot \mathbb{F}_E(\Q) + \nabla \cdot \mathbb{F}_I(\Q) + \mathbb{B}_E \cdot \nabla \Q + \mathbb{B}_I \cdot \nabla \Q = \mathbf{0},
\end{equation}
being $\mathbb{F} = \mathbb{F}_E + \mathbb{F}_I$, $\mathbb{B} = \mathbb{B}_E + \mathbb{B}_I$. The subscripts ``$E$'' and ``$I$'' distinguish the explicit terms related to convective-type phenomena, from the implicit contribution concerned with the fluid pressure and viscosity. As a consequence, we obtain a partitioned system \cite{Hofer} consisting in the following sub-systems.
\begin{itemize}
\item Convective sub-system:
\begin{equation}
\label{eqn.PDE_conv}
\frac{\partial \Q}{\partial t} + \nabla \cdot \mathbb{F}_E(\Q) + \mathbb{B}_E \cdot \nabla \Q = \mathbf{0};
\end{equation}
\item Pressure-viscosity sub-system:
\begin{equation}
\label{eqn.PDE_visco}
\frac{\partial \Q}{\partial t} + \nabla \cdot \mathbb{F}_I(\Q) + \mathbb{B}_I \cdot \nabla \Q = \mathbf{0}.
\end{equation}
\end{itemize}

In the case of the INS equations \eqref{eqn.INS_adim1}, the scaling parameter is represented by the Reynolds number, $\varepsilon = \Rey$, implying that all terms related to the Reynolds number should be treated implicitly. Therefore, we may opt for the following splitting:
\begin{equation}
\label{eqn.INS_split}
\mathbb{F}_E(\Q) = \left(\mathbf{0},\, \vv \otimes \vv \right), \qquad  
\mathbb{F}_I(\Q) = \left(\vv, \, -\frac1{\varepsilon}\left(\nu \nabla \vv - p\I\right)\right),
\end{equation}
which simultaneously allows the time step stability condition to be independent from the parabolic viscous terms and also to respect the divergence-free constraint on the velocity field. The eigenvalues of the explicit sub-system in normal direction $\nn = (n_x,n_y)$ are simply given by
\begin{equation}
	\label{eqn.INS_eig}
	\lambda_1 = 0, \qquad \lambda_{2,3,4} = \vv \cdot \nn,
\end{equation}
thus the time step is limited by the fluid velocity and not by the viscosity coefficient with parabolic cell size.

For the specific case of the SWE \eqref{eqn.SWE_adim}, we adopt the classical flux splitting originally proposed in \cite{Casulli1990}, hence obtaining
\begin{equation}
	\label{eqn.SWE_split}
	\mathbb{F}_E(\Q) = \left( \mathbf{0},\, \vv \otimes \q, \, \mathbf{0} \right), \qquad 
	\mathbb{F}_I(\Q) = \left( \q,\, \mathbf{0}, \, \mathbf{0} \right), \qquad  
	\mathbb{B}_I \cdot \nabla \Q  = \left( 0,\, \frac{H}{\varepsilon^2} \nabla \eta, \, 0 \right).
\end{equation}
Even in this case, the eigenvalues of the explicit sub-system only involve the convective fluid velocity:
\begin{equation}
	\label{eqn.SWE_eig}
	\lambda_{1,2} = 0, \qquad
	\lambda_{3} = \vv \cdot \nn, \qquad
	\lambda_{4} = 2\vv \cdot \nn.
\end{equation}
Indeed, we recall that the scaling parameter coincides with the Froude number, hence $\varepsilon = \Fr$, and all the related terms are taken implicitly.

For both the models, on a computational mesh with characteristic mesh size $h$, the time step is limited by a classical CFL stability condition that is only based on the maximum convective eigenvalue, reading
\begin{equation}
	\label{eqn.timestep} 
	\dt \leq \CFL \min \limits_{\Omega} \frac{h}{\max|\lambda|}.
\end{equation}
The above condition yields a milder stability restriction, especially in the asymptotic regime when $\varepsilon \to 0$. In particular, the time step becomes independent of the stiffness parameter, which means that simulations of flows defined by any Reynolds and Froude number can be run with the same computational efficiency. Moreover, since the terms related to the stiffness are treated implicitly, the numerical dissipation is solely proportional to the fluid speed, yielding a numerical scheme that is particularly suited for applications in the asymptotic (zero-relaxation) limit, i.e., when $\varepsilon \to 0$. We remark here that this type of problems could not be simulated adopting a purely explicit scheme in time \cite{GH,Guillard,Dellacherie1}.

%
\section{Numerical scheme} \label{sec.numscheme}
We start by presenting the discretization of the computational domain in space and time. Next, we define the polynomial spaces adopted for the discretization of the numerical solution, and we present the first order semi-discrete scheme in time, analysing the asymptotic and well-balance properties. Finally, the high order extension in space and time of the new schemes are introduced. 

\subsection{Discretization of the time computational domain}
The time coordinate is defined in the time interval $t \in [0, t_f ]$, where $t_f \in \mathds{R}^+_0$ represents the
final time. A sequence of discrete points $t^n$ approximate the temporal computational domain such that $t \in [t^n; t^{n+1}]$, hence
\begin{equation}
	t^{n+1}=t^n + \dt,
\end{equation}
where the time step $\dt=t^{n+1}-t^n$ is determined at each time iteration according to the stability condition \eqref{eqn.timestep}.

\subsection{Discretization of the space computational domain}
The space computational domain $\Omega \subseteq \REAL^2$ is discretized by a set of non-overlapping polygonal control volumes $P_i$ with boundary $\partial P_i$, surface $|P_i|$ and characteristic mesh size $h_i=\sqrt{|P_i|}$. The total number of cells is $N_P$, thus $i=1,\ldots,N_P$, and the union of all elements is called the tessellation
$\mathcal{T}_{\Omega}$ of the domain 
\begin{equation}
	\mathcal{T}_{\Omega} = \bigcup \limits_{i=1}^{N_P}{P_i}. 
	\label{trian}
\end{equation}	
To make the notation easier, we drop the subscript $i$, bearing in mind that all the geometric quantities are specific to each polygon $P_i$. The barycentre ${\xvP}:=(\xs_{\P},\ys_{\P})$ of each cell is computed as
\begin{equation}
	\xvP = \frac{1}{\mP} \int_{P} \xx \, d\xx.
\end{equation}	
The number of vertices of element $\P$ is denoted by $\NP$.
The boundary of $\P$ is formed by straight edges, and the polygon $\P$ has $\NE$ edges.
The vertices of the polygonal element $\P$ are oriented in counter-clockwise order and their coordinates are denoted by $\XV := (\xs_{\V},\ys_{\V})$, $\V = 1,2,\ldots,\NV$. For convenience, we also use $\XV$ as a label for the $k$-th vertex.
We denote the unit normal vector to edge $\E \in \partial \P$ by $\norPE$. Each vector $\norPE$ points out of $\P$.
We assume that the orientation of all edges is fixed once and for all, so that we can unambiguously introduce
$\norE$, the unit normal vector to edge $\E$, which is independent of the elements $\P$ that share the common edge $\E$.
Moreover, the following assumptions on the regularity of the mesh are adopted.

All the integrals appearing in the numerical scheme are evaluated relying on quadrature formulae of suitable accuracy. For boundary integrals we use Gaussian quadrature rules \cite{stroud}, while for volume integrals we adopt the efficient numerical integration proposed in \cite{SOMMARIVA2009886,SOMMARIVA2020}. 

\paragraph{Mesh regularity assumptions}
There exists a positive constant $\varrho$ such that for every polygonal element $\P\in\mathcal{T}_h$ it holds that
\begin{description}
	\item[$(i)$] $\P$ is star-shaped with respect to a disk with radius greater than
	$\varrho \max \limits_{\mathcal{T}_h} \hP$; 
	\item[$(ii)$] for every edge $\E\in\partial\P$ it holds that
	$|e| \geq\varrho\hP$, with $|e|$ being the edge length.
\end{description}

	The restriction of $\P$ being star-shaped in $(i)$ implies that all the elements
	are simply connected subsets of $\REAL^{2}$ and that they have a finite number of vertices	and edges.
	The scaling assumption in $(ii)$ implies that the number of edges on the boundary of any element is uniformly bounded over the whole mesh 	$\mathcal{T}_h$, thus avoiding collapsing vertexes that would give rise to degenerate edges.

\subsection{Conforming virtual element space}
\label{Sec_Vh}
On every polygonal element $\P$ with boundary $\partial\P$, we define
the standard local virtual element space of order $k$, $\Vh_k(\P)$, as
\begin{align}\label{eq:Vh:def}
	\Vh_k(\P) = \Big\{\,
	\vsh\in\HONE(\P)\,:\, \Delta\vsh \in\PS{k-2}(\P) , \
	\varv_{h}|_{\partial\P}\in\CS{0}(\partial\P), \
	\varv_{h}|_{\E}\in\PS{k}(\E)\,\,\forall\E\in\partial\P \, \Big\},
\end{align}
where $\PS{k}(\P)$ denotes the space of polynomials of degree less than or equal to $k$ on $\P$. 
A straightforward consequence of definition~\eqref{eq:Vh:def} is that
$\PS{k}(\P)$ is a subspace of $\Vh(\P)$. We define
\begin{align}\label{eq:nk} 
	n_k := \textrm{dim}  \Vh_k(\P) = \frac{(k+1)(k+2)}{2}.
\end{align}
Let us introduce the multi-index $\bm{\kappa}=(\kappa_1,\kappa_2)$, with the usual notation $\abs{\bm{\kappa}}=\kappa_1+\kappa_2$. Moreover, if $\bm{x}=(x_1,x_2)$, then $\bm{x}^{\bm{\kappa}}=(x_1^{\kappa_1},x_2^{\kappa_2})$. We define the scaled monomials $m_{\bm{\kappa}}$ of degree  equal to $\abs{\bm{\kappa}}$ as
\begin{align}\label{eq:m_alpha} 
	m_{\bm{\kappa}} = \Big(\frac{\xv-\xvP}{\hP}\Big)^{\bm{\kappa}}.
\end{align}
Hence, we define the set $\mathcal{M}_k(\P)$ of scaled monomials of degree less than or equal to $k$,
\begin{align}\label{eq:M_k} 
	\mathcal{M}_k(\P) := \{ m_{\bm{\kappa}} : 0 \leq \abs{\bm{\kappa}} \leq k \},
\end{align}
and recognize that $\mathcal{M}_k(\P)$ is a basis for $\PS{k}(\P)$. We also use the $\alpha$ subscript to denote the $\alpha$-th scaled monomial $m_{\alpha}$ of $\mathcal{M}_k(\P)$. 

As shown in \cite{vem2}, each function $\vsh$ in $\Vh_k(\P)$ is uniquely determined by the following degrees of freedom:
\begin{itemize}
	\item the $\NP$ values of $\vsh$ at the vertices of $\P$;
	\item the $\NP(k-1)$ values of $\vsh$ on the $(k-1)$ Gauss-Lobatto quadrature points on each edge $\E$;  
	\item the moments up to degree $k-2$ of $\vsh$ in $\P$, defined as
	\begin{equation}
		\label{eqn.dof_mom}
		\frac{1}{\abs{\P}} \int_{\P} \vsh m_{\alpha} \, d\xx, \qquad \alpha = 1,...,n_{k-2},
	\end{equation}
	where $n_{k-2}$ is computed from \eqref{eq:nk} as the dimension of $\PS{k-2}(\P)$.
\end{itemize}
It follows that the dimension of $\Vh_k(\P)$ is
\begin{align}\label{eq:dim_Vh} 
	\textrm{dim}\Vh_k(\P) = \NP k + \frac{(k-1)k}{2},
\end{align}
thus we denote by $\NDOF = \textrm{dim}\Vh_k(\P)$ the number of degrees of freedoms of each  $\vsh \in \Vh_k(\P)$. We also denote by $\dof_i(\vsh)$ the value of the $i$-th degree of freedom of $\vsh$. Since $\{\varphi\}_{i=1}^{\NDOF}$ is the canonical basis for $\Vh_k(\P)$, we can represent each $\vsh \in \Vh_k(\P)$ in terms of its degrees of freedom by means of a Lagrange interpolation:
\begin{align}\label{eq:vh_repr} 
	\vsh = \sum_{i=1}^{\NDOF} \dof_i(\vsh)\varphi_i,
\end{align}
where the usual interpolation property holds true:
\begin{equation}
	\dof_i(\varphi_j) = \delta_{ij}, \qquad i,j=1, \ldots, \NDOF.
\end{equation}
Given the mesh $\mathcal{T}_{\Omega}$, the global conforming virtual element space $\Vh_k \in H^1(\Omega)$ is obtained by gluing together all the elemental spaces $\Vh_k(\P)$, i.e.:
\begin{align}\label{eq:Vh_global:def}
	\Vh_k = \Big\{\,
	\vsh\in\HONE(\P)\,:\, \varv_{h}|_{\P}\in\Vh_k(\P)\,\,\forall\P\in\mathcal{T}_{\Omega} \, \Big\}.
\end{align}

\subsection{Elliptic projection operator}
Since functions in the virtual element space $\Vh_k(\P)$ are not explicitly known, we need to resort to a projection operator $\proj: \Vh_k(\P) \rightarrow \PS{k}(\P)$ that maps functions from the virtual element space to the polynomial space of degree $k$, which is indeed known. Such projector will play a fundamental role in what follows and it is defined up to a constant by the orthogonality condition
\begin{align}\label{eq:vem_proj} 
	\int_{\P} \nabla p_k \cdot \nabla (\proj\vsh - \vsh) \, d\xx = 0, \quad \forall p_k \in \PS{k}(\P).
\end{align}
In order to fix the projection on constants as well, we need to separately define the projector operator $P_0: \Vh_k(\P) \rightarrow \PS{0}(\P)$ requiring that
\begin{align}\label{eq:vem_proj_const} 
	P_0(\proj\vsh - \vsh) = 0.
\end{align}
Among the many possible options for $P_0$, as forwarded in \cite{vem2} we choose to define
\begin{align}\label{eq:vem_proj_0} 
	P_0 \vsh &:= \frac{1}{\NV} \sum_{i=1}^{\NV} \vsh(\xv_i) \quad \textrm{for} \enskip k = 1, \\
	P_0 \vsh &:= \frac{1}{\abs{\P}}\int_{\P} \vsh \, d\xx \quad \textrm{for} \enskip k \geq 2.
\end{align}
It turns out that the projection $\proj \vsh$ is computable using only the degrees of freedom of $\vsh$. Indeed, since $\mathcal{M}_k(\P)$ is a basis for $\PS{k}(\P)$, from \eqref{eq:vem_proj} we can write
\begin{align}\label{eq:vem_proj2} 
	\int_{\P} \nabla m_\alpha \cdot \nabla (\proj\vsh - \vsh) \, d\xx = 0, \quad \alpha = 1,...,n_k, 
\end{align}
and since $\proj\vsh$ is an element of $\PS{k}(\P)$, it can be represented in the basis $\mathcal{M}_k(\P)$ as well:
\begin{align}\label{eq:vem_proj2b} 
	\proj\vsh = \sum_{\beta=1}^{n_k} s^\beta m_\beta.
\end{align}
Consequently, equation \eqref{eq:vem_proj2} becomes:
\begin{align}\label{eq:vem_proj3} 
	\sum_{\beta=1}^{n_k} s^\beta \int_{\P} \nabla m_\alpha \cdot \nabla m_\beta \, d\xx = \int_{\P} \nabla m_\alpha \cdot \nabla \vsh \, d\xx, \quad \alpha = 1,...,n_k,
\end{align}
which constitutes a linear system of $n_k$ equations in the $n_k$ unknowns $s^\beta$. The system is still indeterminate though, and such indeterminacy is overcome by adding conditions \eqref{eq:vem_proj_0}:
\begin{align}\label{eq:vem_proj_aux} 
	\sum_{\beta=1}^{n_k} s^\beta P_0 m_\beta = P_0 \vsh.
\end{align}
Now, the linear system arising from \eqref{eq:vem_proj3} and \eqref{eq:vem_proj_aux} is solvable. In fact, the left-hand side of \eqref{eq:vem_proj3} entails the integration of known polynomials over $\P$. As for the left-hand side, applying integration by parts we get
\begin{align}\label{eq:vem_proj4} 
	\int_{\P} \nabla m_\alpha \cdot \nabla \vsh \, d\xx = -\int_{\P}\Delta m_\alpha \vsh \, d\xx + \int_{\partial \P} \frac{\partial m_\alpha}{\partial n} \vsh \, dS.
\end{align}
The first term is explicitly computed from the internal degrees of freedom of $\vsh$, while the second term contains a polynomial integrand which can be exactly integrated by evaluating it at the Gauss-Lobatto quadrature points along each edge $\E$, where the edge degrees of freedom are located. Indeed, we recall that the basis functions $\{\varphi\}_{i=1}^{\NDOF}$, which span $\vsh$, are of Lagrangian type according to \eqref{eq:vh_repr}.

A projection $\proj\vsh$ can be easily obtained by computing the projections $\proj\varphi_i$ of each of the virtual basis functions of $\Vh_k(\P)$. In fact, we have
\begin{align*} 
	\proj\varphi_i = \sum_{\alpha=1}^{n_k} s^\alpha_i m_\alpha, \quad i = 1,...,\NDOF,
\end{align*}
where $s^\alpha_i$ are the components of the solution vector of the system \eqref{eq:vem_proj3} with the additional condition \eqref{eq:vem_proj_aux}, in which we choose $\vsh = \varphi_i$. In matrix form, this system writes
\begin{align}\label{eq:proj_system} 
	\mathbf{G} \accentset{\star}{\mathbf{\Pi}}^{\nabla}_k = \mathbf{B},
\end{align}
where $\mathbf{G}$ and $\mathbf{B}$ are respectively $n_k \times n_k$ and   $n_k \times \NDOF$ matrices such that:
\begin{align}\label{gb_matrices} 
	(\mathbf{G})_{\alpha \beta} &= P_0 m_\beta, \quad & \textrm{for} \enskip \alpha = 1, & \\
	(\mathbf{G})_{\alpha \beta} &= \int_{\P}\nabla m_\alpha \cdot \nabla m_\beta \, d\xx, \quad & \textrm{for} \enskip  \alpha \geq 2, & \\
	(\mathbf{B})_{\alpha  i} &= P_0 \varphi_\beta, \quad & \textrm{for} \enskip \alpha = 1, & \enskip i = 1,...,\NDOF \\
	(\mathbf{B})_{\alpha  i} &= \int_{\P} \nabla m_\alpha \cdot \nabla \varphi_i \, d\xx, \quad & \textrm{for} \enskip \alpha \geq 2, & \enskip i = 1,...,\NDOF,
\end{align}
and $\accentset{\star}{\mathbf{\Pi}}^{\nabla}_k$ is the $n_k \times \NDOF$ matrix representation of operator $\proj$ in the basis $\mathcal{M}_k(\P)$, so that $(\accentset{\star}{\mathbf{\Pi}}^{\nabla}_k)_{\alpha i} = s_i^\alpha$. 

It is also possible to provide a matrix representation $\mathbf{\Pi}^\nabla_k$ of operator $\proj$ in the canonical basis of $\Vh_k(\P)$, by introducing the $\NDOF \times n_k$ matrix $\mathbf{D}$ that performs a change of basis:
\begin{align}\label{D_matrix} 
	(\mathbf{D})_{i \alpha} = \dof_i (m_\alpha), \quad i = 1,...,\NDOF, \quad \alpha = 1,...,n_k.
\end{align}
Thus, $\mathbf{\Pi}^\nabla_k$ is the $\NDOF \times \NDOF$ matrix computed as
\begin{align}\label{Pi_nabla} 
	\mathbf{\Pi}^\nabla_k = \mathbf{D}  \accentset{\star}{\mathbf{\Pi}}^{\nabla}_k = \mathbf{D}\mathbf{G}^{-1}\mathbf{B}.
\end{align}

\subsection{$L_2$ projection operator}
\label{L2_proj_subs}
We also need to resort to the $L_2$ projection operator $\projL: \Vh_k(\P) \rightarrow \PS{k}(\P)$ that maps functions from the virtual element space to the polynomial space of degree $k$. Such projector is defined by the orthogonality condition
\begin{align}\label{eq:vem_projL1} 
	\int_{\P} p_k (\projL\vsh - \vsh) \, d\xx = 0, \quad \forall p_k \in \PS{k}(\P).
\end{align}
The projection $\projL \vsh$ can be evaluated using only the degrees of freedom of $\vsh$. Indeed, since $\projL\vsh$ is an element of $\PS{k}(\P)$, it can be represented in the basis $\mathcal{M}_k(\P)$:
\begin{align}\label{eq:vem_projL2} 
	\projL\vsh = \sum_{\beta=1}^{n_k} r^\beta m_\beta.
\end{align}
From this, it follows that condition \eqref{eq:vem_projL1} becomes
\begin{align}\label{eq:vem_projL3} 
	\sum_{\beta=1}^{n_k} t^\beta \int_{\P} m_\alpha m_\beta \, d\xx = \int_{\P} m_\alpha \vsh \, d\xx, \quad \alpha = 1,...,n_k, 
\end{align}
which is a linear system of $n_k$ equations in the $n_k$ unknowns $r^\beta$. While the left-hand side of \eqref{eq:vem_projL3} is readily computable, involving the integration of known polynomials over $\P$, the right-hand side clearly is not. In fact, for a general $k$ we know moments as degrees of freedom only for $m_\alpha \in \PS{k-2}(\P)$. The idea here is to replace $\vsh$ in the right-hand side with its elliptic projection $\proj \vsh$ only for monomials of degree $k$ and $k-1$. Such system can be written in matrix form to give the $n_k \times \NDOF$ matrix representation $\accentset{\star}{\mathbf{\Pi}}^0_k$ of $\projL$ as follows:
\begin{align}\label{eq:proj_system2} 
	\mathbf{H} \accentset{\star}{\mathbf{\Pi}}^0_k = \mathbf{C}.
\end{align}
The matrices $\mathbf{H}$ and $\mathbf{C}$ are of dimensions $n_k \times n_k$ and $n_k \times \NDOF$, respectively, and they are explicitly defined as
\begin{subequations}
\begin{align} 
	(\mathbf{H})_{\alpha \beta} &= \int_{\P}m_\alpha m_\beta \, d\xx, \qquad \quad \alpha,\beta = 1,...,n_k, \label{h_matrix} \\
	(\mathbf{C})_{\alpha  i} &= 
	\left\{\begin{array}{@{}l@{}l@{}l@{}}
		\int_{\P}m_\alpha \varphi_i \, d\xx, & \quad \alpha= 1,...,n_{k-2}, &\enskip i = 1,...,\NDOF, \\ [4pt]
		\int_{\P}m_\alpha \proj \varphi_i \, d\xx, & \quad n_{k-2}+1 \leq \alpha \leq n_k, &\enskip i = 1,...,\NDOF. \label{c_matrix}
	\end{array}\right. 
\end{align}
\end{subequations}
It is also possible to provide the matrix representation $\mathbf{\Pi}^0_k$ of operator $\projL$ in the canonical basis of $\Vh_k(\P)$, by using the change of basis matrix $\mathbf{D}$ \eqref{D_matrix}, which is
\begin{align}\label{Pi_0} 
	\mathbf{C} = \mathbf{D} \accentset{\star}{\mathbf{\Pi}}^0_k = \mathbf{H}\mathbf{G}^{-1}\mathbf{C}.
\end{align}

In a similar fashion, we can also compute the matrix representations $\mathbf{\Pi}^0_{k-1}$ and $\accentset{\star}{\mathbf{\Pi}}^0_{k-1}$ of the $L_2$ projection onto $\PS{k-1}(\P)$. To this end, we solve the linear system
\begin{align}\label{eq:proj_system3} 
	\mathbf{H}' \accentset{\star}{\mathbf{\Pi}}^0_{k-1} = \mathbf{C}',
\end{align}
where the $n_{k-1} \times n_{k-1}$ matrix $\mathbf{H}'$ is obtained by taking the first $n_{k-1}$ rows and columns of matrix $\mathbf{H}$, defined in \eqref{h_matrix}, and the $n_{k-1} \times \NDOF$ matrix $\mathbf{C}'$ is obtained by taking the first $n_{k-1}$ rows of matrix $\mathbf{C}$, defined in \eqref{c_matrix}. 

Finally, we obtain
\begin{align}\label{eq:proj_system4} 
	\mathbf{\Pi}^0_{k-1} = \mathbf{D}'\accentset{\star}{\mathbf{\Pi}}^0_{k-1},
\end{align}
where the $\NDOF \times n_{k-1}$ matrix $\mathbf{D}'$ is computed by taking the first $n_{k-1}$ columns of matrix $\mathbf{D}$, defined in \eqref{D_matrix}.
\subsection{Finite volume space}
In this work, the leading numerical scheme relies on the finite volume paradigm, which is used to discretize the explicit terms in \eqref{eqn.PDE_split} and to store the numerical solution within each control volume $P_i$ for every time level $t^n$. Finite volume schemes represent the vector of conserved variables $\Q$ as cell averages refereed to each polygonal element $P$:
\begin{equation}
	\label{eqn.cellAv}
	\Q_i^n:=\frac{1}{|P_i|} \int \limits_{P_i} \Q(\xx,t^n) \, d\xx,
\end{equation}
which obviously belong to the space of polynomials of degree zero ($k=0$), since they are constant states within each cell. To make notation easier, we omit the subscript $i$ referring to the cell $P_i$ and the superscript $n$ pointing to the current time level.

To achieve higher order of accuracy, a non-linear reconstruction of the finite volume solution \eqref{eqn.cellAv} is carried out, hence obtaining an approximation of the solution of degree $k>0$ which is addressed with high order polynomials $\w(\xx)$ while avoiding spurious oscillations in the proximity of eventual discontinuities (further details on the high order reconstruction in the FV space will be given in Section \ref{sec.FVM_ex}). Consequently, the reconstructed solution lies in the space of polynomials $\PS{k}$ and it is written in terms of an expansion of a set of modal basis functions $\beta_{\ell}$ with corresponding $n_k$ degrees of freedom $\hat{\w}_{\ell}$:
\begin{equation}
	\label{eqn.recPoly}
	\w(\xx) = \sum \limits_{\ell=1}^{n_k} \beta_{\ell}(\xx) \, \hat{\w}_{\ell}.
\end{equation} 
The basis functions $\beta_l$ are given by a set of conservative Taylor functions that can be directly retrieved from the monomials \eqref{eq:m_alpha}:
\begin{equation}
	\label{eqn.Voronoi_modal}
	\beta_{\ell}(\xx)|_{P} = m_{\bm{\kappa}} - \frac{c}{|P|}\int_{P} m_{\bm{\kappa}} \, d\xx.
\end{equation}
The coefficient $c$ ensures the conservation property of the basis functions, thus we have $c = 0$ for $|\bm{\kappa}|=0$, and $c=1$ for $|\bm{\kappa}|>0$. We remark that the conservation property of the basis functions \eqref{eqn.Voronoi_modal} means that
\begin{equation}
	\label{eqn.modal_cons}
	\frac{1}{|P|}\int_{P} \sum_{l=1}^{n_k}\beta_{\ell}(\xx) \, d\xx = 1,
\end{equation}
thus, the first degree of freedom of each element $P$ (i.e. the one identified by $\ell=1$) represents the cell average value, in the finite volume sense. 

\subsection{Semi-discrete semi-implicit IMEX scheme in time}
Concerning the temporal discretization of the equations, let us start by presenting the first order in time semi-discrete scheme according to the splitting \eqref{eqn.PDE_split}, which reads:
\begin{equation}
\label{eqn.FOtime}
\frac{\Q^{n+1} - \Q^{n}}{\dt} + \nabla \cdot \mathbb{F}_E(\Q^{n}) + \nabla \cdot \mathbb{F}_I(\Q^{n+1}) + \mathbb{B}_E(\Q^{n}) \cdot \nabla \Q^{n} + \mathbb{B}_I(\Q^{n}) \cdot \nabla \Q^{n+1} = \mathbf{0}.
\end{equation}

\paragraph{Incompressible Navier-Stokes model} For the INS equations \eqref{eqn.INS}, the semi-discrete scheme expressed in \eqref{eqn.FOtime} leads to
\begin{subequations}
	\label{eqn.INS_FOtime}
	\begin{align}
		\nabla \cdot \vv^{n+1} &= 0,  \label{eqn.INS_FOtime_cont}\\
		\frac{\vv^{n+1}-\vv^n}{\dt}  + \nabla \cdot (\vv^n \otimes \vv^n) - \nu \Delta \vv^{n+1} + \nabla p^{n+1} &= \boldsymbol{0}, \label{eqn.INS_FOtime_mom}
	\end{align}
\end{subequations}
which is solved as follows. From the momentum equation we obtain a provisional velocity field $\vv^*$, which already contains the computation of both explicit convection and implicit viscous contribution, that is given by
\begin{equation}
	\label{eqn.vvisc}
	\vv^*- \dt \, \nu \Delta \vv^{*} = F_{\vv}^n - \dt \, \nabla p^{n}, \qquad F_{\vv}^n = \vv^n - \dt \, \nabla \cdot \left( \vv^n \otimes \vv^n \right).
\end{equation}
Notice that the discretization of the convection terms is simply referred to with $F_{\vv}^n$ and that we use the pressure gradient at time $t^n$ to supplement the provisional velocity with physical information about the pressure field, as proposed in \cite{Tavelli2014}. Once the velocity field $\vv^*$ is determined upon the solution of \eqref{eqn.vvisc}, the momentum equation \eqref{eqn.INS_FOtime_mom} is given by
\begin{equation}
	\label{eqn.vnewINS}
	\frac{\vv^{n+1}-\vv^*}{\dt} + \nabla p^{n+1} - \nabla p^{n} = \bzero.
\end{equation}
The above equation can be inserted into the continuity equation \eqref{eqn.INS_FOtime_cont} yielding an elliptic equation for the unknown pressure $p^{n+1}$:
\begin{equation}
	\label{eqn.pnew_INS}
     \dt \, \Delta p^{n+1} = \nabla \cdot \vv^{*} + \dt \, \Delta p^n.
\end{equation}
The pressure system \eqref{eqn.pnew_INS} is solved at the aid of the GMRES method \cite{GMRES}, where we prescribe a tolerance $\delta_0=10^{-12}$ to stop the iterative procedure. Once the new pressure is known, the velocity field is updated according to \eqref{eqn.vnewINS}.

\begin{theorem}\label{th_ins_1}(Divergence-free constraint). 
	Assuming periodic boundary conditions on $\partial \Omega \in \R$, the semi-discrete scheme \eqref{eqn.INS_FOtime}-\eqref{eqn.vvisc} satisfies the divergence-free constraint on the velocity field:
	\begin{equation}
		\label{eqn.th_divfree}
		\nabla \cdot \vv^{n+1} = 0.
	\end{equation}
\end{theorem}

\begin{proof}
The semi-discrete scheme \eqref{eqn.INS_FOtime} supplemented with \eqref{eqn.vvisc} is equivalent to a splitting method in which we consider the following set of equations:
\begin{subequations}
	\label{eqn.divv1}
	\begin{align}
		\nabla \cdot \vv^{n+1} &= 0,  \label{eqn.divv1_cont}\\
		\frac{\vv^{*}-\vv^n}{\dt}  + \nabla \cdot (\vv^n \otimes \vv^n) - \nu \Delta \vv^{*} + \nabla p^{n} &= \boldsymbol{0}. \label{eqn.divv1_visc} \\
		\frac{\vv^{n+1}-\vv^{*}}{\dt} - \nabla p^{n} + \nabla p^{n+1} &= \boldsymbol{0}. \label{eqn.divv1_mom}
	\end{align}
\end{subequations}
We can solve the system by substitution, hence inserting \eqref{eqn.divv1_mom} into \eqref{eqn.divv1_visc} leads to
\begin{equation}
	\label{eqn.divv2}
	\vv^{n+1} = \vv^n - \dt \, \nabla \cdot (\vv^n \otimes \vv^n)  + \dt \, \nu \Delta \vv^{*} - \dt \, \nabla p^{n} + \dt \, \nabla p^{n} - \dt \, \nabla p^{n+1}.
\end{equation}	
Taking the divergence of the above equation, which is equivalent to substitute the above expression into the continuity equation \eqref{eqn.divv1_cont}, yields
\begin{eqnarray}
	\label{eqn.divv3}
	\nabla \cdot \vv^{n+1} &=& \nabla \cdot \left( \vv^n - \dt \, \nabla \cdot (\vv^n \otimes \vv^n)  + \dt \, \nu \Delta \vv^{*} - \dt \, \nabla p^{n} \right) + \dt \, \Delta p^{n} - \dt \, \Delta p^{n+1}, \nonumber \\
	&=& \nabla \cdot \vv^* + \dt \, \Delta p^{n} - \dt \, \Delta p^{n+1}.
\end{eqnarray}
The right hand side of \eqref{eqn.divv3} is indeed the elliptic equation \eqref{eqn.pnew_INS}, thus it vanishes, and the sought divergence-free constraint is preserved \eqref{eqn.th_divfree}.
\end{proof}

\paragraph{Shallow water equations} For the SWE model \eqref{eqn.SWE}, the first order semi-discrete scheme in time \eqref{eqn.FOtime} reads:
\begin{subequations}
	\label{eqn.SWE_FOtime}
	\begin{align}
		\frac{\eta^{n+1}-\eta^n}{\dt} + \nabla \cdot \q^{n+1} &= 0\,,  \label{eqn.SWE_FOtime_cont}\\
		\frac{\q^{n+1}-\q^n}{\dt}  + \nabla \cdot \left( \vv^n \otimes \q^n \right) + g H^n \nabla \eta^{n+1} &= \boldsymbol{0}\,. \label{eqn.SWE_FOtime_mom}
	\end{align}
\end{subequations}
Notice that the last equation for the bottom bathymetry is neglected since $b(\xx)$ is not time dependent. Following \cite{Casulli1990}, the above system is solved by substitution, inserting \eqref{eqn.SWE_FOtime_mom} into \eqref{eqn.SWE_FOtime_cont}, which leads to a wave equation for the unknown $\eta^{n+1}$:
\begin{equation}
	\label{eqn.Poisson}
	\eta^{n+1} + \dt^2 \, g \, \nabla \cdot \left( H^n \, \nabla \eta^{n+1}\right) = \eta^n - \dt \nabla \cdot F_{\q}^{n}, \qquad F_{\q}^{n} = \q^n - \dt \, \nabla \cdot \left( \vv^n \otimes \q^n \right),
\end{equation} 
where $F_{\q}^{n}$ denotes again the contribution of the non-linear explicit terms. As for the INS equations, the above linear system is solved with the GMRES method up to the prescribed tolerance $\delta_0$. Once $\eta^{n+1}$ is computed, it can be used to straightforwardly obtain $\q^{n+1}$ through \eqref{eqn.SWE_FOtime_mom}.

\begin{theorem}\label{th_swe_1}(Well-balance property). 
	Assuming periodic boundary conditions on $\partial \Omega \in \R$ and assuming the following initial condition
	\begin{equation}
		\label{eqn.QWB}
		\eta(\xx,0) = \eta_0, \qquad \vv(\xx,0)=\bzero, \qquad b(\xx) \neq 0,
	\end{equation}
	the semi-discrete scheme \eqref{eqn.SWE_FOtime} is well-balanced in the sense of \cite{WBLeVeque}.
\end{theorem}

\begin{proof}
The initial condition implies that $\eta^n=\eta_0$ and $\vv^n=0$, hence $\q^n=\bzero$. Consequently, in the absence of discontinuities in the numerical solution, we get a vanishing non-linear convective contribution, i.e. $\q^*=\bzero$. More precisely, the numerical dissipation associated to the numerical flux of the convective term $\nabla \cdot \left(\vv^n \otimes \q^n \right)$ is exactly zero for any constant state, included $\q^n=\bzero$. Thus, the wave equation \eqref{eqn.Poisson} reduces to
\begin{equation}
	\label{eqn.Poisson2}
	\eta^{n+1} - \dt^2 \, g \, \nabla \cdot \left( H^n \, \nabla \eta^{n+1}\right) = \eta^n, \qquad \q^* = \bzero,
\end{equation}
for which $\eta^{n+1}=\eta^n=\eta_0$ is an admissible solution. From this it follows that $\nabla \eta^{n+1}=0$. The discharge equation is then updated according to \eqref{eqn.SWE_FOtime_mom}, meaning that
\begin{equation}
	\q^{n+1} = \bzero - \dt g H^n \cdot 0 = \bzero.
\end{equation}
It is therefore evident that the semi-discrete scheme \eqref{eqn.SWE_FOtime} can preserve stationary solutions of the shallow water system of the form given by \eqref{eqn.QWB} with arbitrary bathymetry.
\end{proof}

\paragraph{High order semi-implicit IMEX schemes} The class of semi-implicit Implicit-Explicit (IMEX) Runge-Kutta methods is used to attain higher order accuracy in time \cite{BosFil2016}.  An IMEX Runge-Kutta scheme is a multi-step method characterized by two $s \times s$ triangular matrices, the explicit one, $\tilde A = (\tilde a_{kj})$, with $\tilde a_{kj} = 0$ for $ j\geq k$, and the implicit one, $A = (a_{kj})$, with $a_{kj} = 0$ for $j > k$, and by the weights vectors $\tilde b = (\tilde b_1, ...,\tilde b_s)^T$, $b = (b_1, ...,b_s)^T$, where $s$ identifies the number of implicit Runge-Kutta stages. They are usually defined by the following explicit (on the left) and implicit (on the right) Butcher tableau:
\begin{center}
\begin{tabular}{c | c}
 $\tilde{c}$ & $\tilde{A}$ \\ \hline
 & $\tilde{b}$
\end{tabular}
\hspace{1.0cm}
\begin{tabular}{c | c}
$c$ & $A$ \\ \hline
 & $b$
\end{tabular}
\end{center}
whose characterization provides the time accuracy of the scheme. According to \cite{BosFil2016}, the governing equations \eqref{eqn.PDE} are written under the form of an autonomous system, that is
\begin{equation}
	\label{eqn.PDE_part}
	\frac{\partial \Q}{\partial t} = \mathcal{H}\left(\Q_E(t), \Q_I(t) \right), \qquad \forall t > t_0, \qquad \textnormal{with} \qquad \Q(t_0)= \Q_0,
\end{equation}
where $\Q_0$ defines the initial condition at time $t_0$. The semi-implicit IMEX technique is a method-of-lines (MOL) strategy, hence the function $\mathcal{H}$ represents the spatial approximation of the conservative fluxes as well as the non-conservative products in \eqref{eqn.PDE}. The arguments of $\mathcal{H}$ undergo an explicit or implicit discretization in accordance to the flux splitting \eqref{eqn.PDE_split}. Consequently, a partitioned system with $\Q=(\Q_E,\Q_I)$ is obtained:
\begin{equation}
	\left\{\begin{aligned}
		\frac{\partial \Q_E}{\partial t} &=  \mathcal{H}\left(\Q_E, \Q_I \right) \\[0.7pt]
		\frac{\partial \Q_I}{\partial t} &=  \mathcal{H}\left(\Q_E, \Q_I \right) 
	\end{aligned} \right. ,
	\label{eqn.Hauto}
\end{equation}
where the number of unknowns has been doubled. However, for specific choices of the IMEX scheme and for autonomous systems this duplication is only apparent \cite{BosFil2016}. We can easily assume that the governing equations \eqref{eqn.INS} and \eqref{eqn.SWE} are autonomous, since no explicit time dependency is present in the function $\mathcal{H}$. Therefore, only one set of stage fluxes needs to be computed, which, at each stage $i = 1, \ldots, s$ can be evaluated as
\begin{equation}
	k_i = \mathcal{H} \left( \, \, \Q_E^n + \dt \sum \limits_{i=1}^s \tilde{a}_{ij} \, k_j, \quad \Q_I^n + \dt \sum \limits_{i=1}^s a_{ij} \, k_j \, \, \right), \qquad 1 \leq i \leq s.
	\label{eqn.ki}
\end{equation}
A semi-implicit IMEX Runge-Kutta method is then obtained as follows. Let us first set $\Q_E^n=\Q_I^n=\Q^n$, then the stage fluxes for $i = 1, \ldots, s$ are calculated as
\begin{subequations}
	\begin{align}
		\Q_E^i &= \Q_E^n + \dt \sum \limits_{j=1}^{i-1} \tilde{a}_{ij} k_j, & 2 \leq i \leq s, \label{eq.QE} \\[0.5pt]
		\tilde{\Q}_I^i &= \Q_E^n + \dt \sum \limits_{j=1}^{i-1} a_{ij} k_j, & 2 \leq i \leq s, \label{eq.QI}  \\[0.5pt]
		k_i &= \mathcal{H} \left( \Q_E^i, \tilde{\Q}_I^i + \dt \, a_{ii} \, k_i \right), & 1 \leq i \leq s. \label{eq.k} 
	\end{align}
\end{subequations}
Finally, the numerical solution is updated with
\begin{equation}
	\Q^{n+1} = \Q^n + \dt \sum \limits_{i=1}^s b_i k_i.
	\label{eqn.QRKfinal}
\end{equation} 

We remark that equation \eqref{eq.k} gives rise to a formally implicit step with the solution of a system for $k_i$, that corresponds to the elliptic equations \eqref{eqn.pnew_INS} and \eqref{eqn.Poisson}. The final update of the solution \eqref{eqn.QRKfinal} is done using the implicit weights $b^\top$ that are assumed to be equal to the explicit ones $\tilde{b}^\top$. Furthermore, the stage fluxes $k_i$ in \eqref{eqn.ki} are the same for both explicit and implicit conserved vectors $\Q_E$ and $\Q_I$,  therefore the system is actually not doubled, since there is indeed only one set of numerical solution. The Butcher tableaux for the IMEX schemes chosen to reach first, second and third order of accuracy in time are given in \ref{app.IMEX}.

\subsection{Asymptotic-Preserving property} \label{ssec.AP}
Numerical schemes satisfy the Asympotic-Preserving (AP) property if they provide a consistent discretization of the asymptotic limit model of the governing equations. To prove the fulfilment of this property, we first derive the asymptotic limit of the PDE systems \eqref{eqn.INS} and \eqref{eqn.SWE}, and then we study the AP property at the semi-discrete level for the schemes \eqref{eqn.INS_FOtime} and \eqref{eqn.SWE_FOtime}. The high order version of the schemes will also be compliant with the AP property, thanks to the semi-implicit IMEX strategy, as fully detailed in \cite{BosFil2016}. For this reason, in this framework we limit ourselves to consider only the first order semi-discrete schemes presented in the previous section.

Let us assume the computational domain $\Omega(\xx)$ to be assigned with periodic boundary conditions on $\partial \Omega$ and let us introduce the $k$-th order Chapman-Enskog expansion of a generic variable $\phi(\xx,t)$ in powers of the non-dimensional stiffness parameter $\varepsilon$, that reads
\begin{equation}
	\phi(\xx,t) = \phi_{(0)}(\xx,t) + \varepsilon \phi_{(1)}(\xx,t) + \varepsilon^2 \phi_{(2)}(\xx,t) + \ldots + \mathcal{O}(\varepsilon^k).
	\label{eqn.exp}
\end{equation}
As previously introduced in Section \ref{sec.pde}, the Reynolds number $\Rey$ is the stiffness parameter for the incompressible Navier-Stokes equations \eqref{eqn.INS_adim1}, while the Froude number $\Fr$ is used to characterize the scaling of the shallow water equations \eqref{eqn.SWE_adim}. 

\paragraph{Incompressible Navier-Stokes equations} Application of the expansion \eqref{eqn.exp} to the rescaled governing PDE \eqref{eqn.INS_adim1} and collection of the like powers of $\varepsilon$ yields the following $k$-th leading order equations for $k\in\{0,-1\}$.
\begin{itemize}
	\item $\mathcal{O}(\varepsilon^{0})$:
	\begin{subequations}
		\label{eqn.INS_eps0}		
		\begin{align}
			\nabla \cdot \vv_{(0)} &= 0, \\ 
			\partial_t \vv_{(0)} + \nabla \cdot \left( \vv_{(0)} \otimes \vv_{(0)} \right) + \nabla p_{(1)} - \nu \Delta \vv_{(1)}&= \bzero.
		\end{align}
	\end{subequations}
	\item $\mathcal{O}(\varepsilon^{-1})$:
	\begin{equation}
		\label{eqn.INS_eps1}		
		\nabla p_{(0)} - \nu \Delta \vv_{(0)}= \bzero.
	\end{equation}
\end{itemize}	
Notice that at the first order asymptotic expansion we retrieve the diffusive Stokes equation \cite{Schlichting}. 

\begin{theorem}\label{th_ins_2}(Asymptotic-Preserving property). Assuming periodic boundary conditions on $\partial \Omega \in \R$, the semi-discrete scheme \eqref{eqn.INS_FOtime}-\eqref{eqn.vvisc} is a consistent approximation of the Stokes system \eqref{eqn.INS_eps1} at the first order asymptotic expansion in the asymptotic limit ($\varepsilon \to 0$).	
\end{theorem}

\begin{proof}
We assume that the following expansions hold true for the discrete variables at any generic time $t^n$:
\begin{equation}
	\label{eqn.exp2}
	\vv^n(\xx) = \vv_{(0)}^n (\xx) + \varepsilon \vv_{(1)}^n(\xx), \qquad p^n(\xx) = p_{(0)}^n (\xx) + \varepsilon p_{(1)}^n(\xx).
\end{equation}
We recall that the semi-discrete scheme \eqref{eqn.INS_FOtime}-\eqref{eqn.vvisc} can be explicitly written as a splitting method for the viscous terms according to \eqref{eqn.divv1}, and its rescaled version with the dimensionless variables \eqref{eqn.Qadim} writes
\begin{subequations}
	\label{eqn.divv1-adim}
	\begin{align}
		\nabla \cdot \vv^{n+1} &= 0,  \label{eqn.divv1_adim_cont}\\
		\frac{\vv^{*}-\vv^n}{\dt}  + \nabla \cdot (\vv^n \otimes \vv^n) - \frac{\nu}{\varepsilon} \Delta \vv^{*} + \frac{1}{\varepsilon}\nabla p^{n} &= \boldsymbol{0}. \label{eqn.divv1_adim_visc} \\
		\frac{\vv^{n+1}-\vv^{*}}{\dt} - \frac{1}{\varepsilon} \nabla p^{n} + \frac{1}{\varepsilon} \nabla p^{n+1} &= \boldsymbol{0}. \label{eqn.divv1_adim_mom}
	\end{align}
\end{subequations}
Inserting \eqref{eqn.exp2} into the semi-discrete scheme \eqref{eqn.divv1-adim} and retaining only first order terms of the expansions leads to
\begin{subequations}
	\label{eqn.divv2-adim}
	\begin{align}
		- \nu \Delta \vv_{(0)}^{*} + \nabla p_{(0)}^{n} &= \boldsymbol{0}. \label{eqn.divv2_adim_visc} \\
		- \nabla p_{(0)}^{n} + \nabla p_{(0)}^{n+1} &= \boldsymbol{0}. \label{eqn.divv2_adim_mom}
	\end{align}
\end{subequations}
This obviously yields the discrete Stokes model
\begin{equation}
	\nabla p_{(0)}^{n+1} = \nu \Delta \vv_{(0)}^{*},
\end{equation}
which is a consistent discretization of the limit model \eqref{eqn.INS_eps1}. Indeed, at the first order asymptotic expansion the divergence-free constraint disappears, and we have that $\vv^*=\vv^{n+1}$. 
\end{proof}

\paragraph{Shallow water equations} Following the same reasoning, we can identify the limit model of the shallow water system by using the expansion \eqref{eqn.exp2} in the rescaled model \eqref{eqn.SWE_adim} and by collecting like powers of $\varepsilon$ for $k\in\{0,-1,-2\}$.
\begin{itemize}
	\item $\mathcal{O}(\varepsilon^{0})$:
	\begin{subequations}
		\label{eqn.SWE_eps0}		
		\begin{align}
			\partial_t \eta_{(0)} + \nabla \cdot \left( \left(\eta_0-b\right) \, \vv_0 \right)&= 0, \\ 
			\partial_t \q_{(0)} + \nabla \cdot \left( \vv_{(0)} \otimes \q_{(0)} \right) + \eta_{(2)} \nabla \eta_{(0)} + \eta_{(1)} \nabla \eta_{(1)} + \left(\eta_{(0)}-b \right) \nabla \eta_{(2)} &= \bzero.
		\end{align}
	\end{subequations}
	
	\item $\mathcal{O}(\varepsilon^{-1})$:
	\begin{equation}
		\label{eqn.SWE_eps1}		
		\eta_{(1)} \nabla \eta_{(0)} + \left(\eta_{(0)}-b \right) \nabla \eta_{(1)} = \bzero.
	\end{equation}
	
	\item $\mathcal{O}(\varepsilon^{-2})$:
	\begin{equation}
		\label{eqn.SWE_eps2}		
		\left(\eta_{(0)}-b \right) \nabla \eta_{(0)} = \bzero.
	\end{equation}
\end{itemize} 
From \eqref{eqn.SWE_eps2} it directly follows that $\eta_{(0)} = \eta_{(0)}(t)$, which can be used in \eqref{eqn.SWE_eps1} to obtain $\eta_{(1)} = \eta_{(1)}(t)$ under the assumption of $H(\xx,t)>0$. Since we have assumed periodic boundaries, integration of the mass equation over the computational domain and application of Gauss theorem leads to vanishing fluxes, hence determining that the free surface elevation $\eta_{(0)}$ is constant in space as well. This means that the total water depth also becomes constant in both space and time, that is $H(\xx,t)=const$. Then, the low Froude shallow water system writes
\begin{subequations}
	\label{eqn.LowFr}		
	\begin{align}
		\nabla \cdot \left( \left(\eta_0-b\right) \, \vv_0 \right) &=0, \label{eqn.LowFr_eta}\\ 
		\partial_t \q_{(0)} + \nabla \cdot \left( \vv_{(0)} \otimes \q_{(0)} \right) + \left(\eta_{(0)}-b \right) \nabla \eta_{(2)} &= \bzero. \label{eqn.LowFr_q}
	\end{align}
\end{subequations}
The reader is referred to \cite{SIFVDG} for further details on the derivation of the limit model for the SWE system.

\begin{theorem}\label{th_swe_2}(Asymptotic Preserving property). Assuming periodic boundary conditions on $\partial \Omega \in \R$, the semi-discrete scheme \eqref{eqn.SWE_FOtime} is a consistent approximation of the low Froude shallow water system \eqref{eqn.LowFr} at the leading order asymptotic expansion in the asymptotic limit ($\varepsilon \to 0$).	
\end{theorem}

\begin{proof}
We assume that the following expansions hold true for the discrete variables at any generic time $t^n$:
\begin{equation}
	\label{eqn.exp3}
	\eta^n(\xx) = \eta_{(0)}^n (\xx) + \varepsilon^2 \eta_{(2)}^n(\xx), \qquad \vv^n(\xx) = \vv_{(0)}^n (\xx) + \varepsilon \vv_{(1)}^n(\xx).
\end{equation}
We remark that $\eta_{(0)}^n(\xx)=\eta_{(0)}$ is constant in space and time and that $\varepsilon^2 \eta_{(2)}^n(\xx)$ is a perturbation of the free surface level, thus the constant total water depth at zeroth order is given by $H_{(0)}^n(\xx)=\eta_{(0)}(\xx)-b(\xx)$. Inserting \eqref{eqn.exp3} into the semi-discrete scheme \eqref{eqn.SWE_FOtime} and retaining only zeroth order terms of the expansions yields
\begin{subequations}
	\label{eqn.FOtime_scaled2}
	\begin{align}
		\nabla \cdot \left( H_{(0)} \vv_{(0)} \right)^{n+1} &= 0, \label{eqn.FOtime_scaled2_eta} \\
		\frac{\left( H_{(0)} \vv_{(0)}\right)^{n+1}-\left( H_{(0)} \vv_{(0)} \right)^n}{\dt}  + \nabla \cdot \left( {\vv_{(0)}^n} \otimes (H_{(0)} \vv_{(0)})^n \right) + H_{(0)}^n \nabla \eta_{(2)}^{n+1} &= \bzero, \label{eqn.FOtime_scaled2_q}
	\end{align}
\end{subequations} 
that is a consistent discretization of the low Froude shallow water system \eqref{eqn.LowFr}. Furthermore, if we formally substitute the discharge equation \eqref{eqn.FOtime_scaled2_q} into the mass equation \eqref{eqn.FOtime_scaled2_eta}, we obtain the corresponding rescaled version of the elliptic equation \eqref{eqn.Poisson} in the asymptotic limit, namely
\begin{equation}
	\label{eqn.Poisson_scaled}
		\dt \, H_{(0)} \, \nabla \cdot \nabla \eta_{(2)}^{n+1} = \nabla \cdot \left( H_{(0)} \vv_{(0)} \right)^n - \dt \, \nabla \cdot \left( \nabla \cdot \left( {\vv_{(0)}^n} \otimes (H_{(0)} \vv_{(0)})^n \right) \right).
\end{equation}
It is interesting to notice that the above equation \eqref{eqn.Poisson_scaled}, divided by the constant water depth $H_{(0)}$, is the classical Poisson equation for the pressure correction of the projection stage for the incompressible Navier-Stokes model, which we solve in \eqref{eqn.pnew_INS}. The only difference is that in \eqref{eqn.pnew_INS} we use a splitting for the viscous terms, hence yielding the additional term related to the pressure gradient at the current time $t^n$. We can therefore conclude that the semi-discrete scheme \eqref{eqn.SWE_FOtime} is a consistent discretization of the limit model \eqref{eqn.LowFr}.  
\end{proof}

\subsection{Spatial finite volume discretization of the explicit terms}
\label{sec.FVM_ex}
The explicit terms are only concerned with the conservative fluxes $\mathbb{F}_E(\Q)$ in \eqref{eqn.PDE_split}, and they are discretized using a classical finite volume scheme tailored for general polygonal meshes. They are constructed upon integration of the governing PDE over the control volume $P_i$ and application of Gauss theorem, hence leading to
\begin{equation}
	\label{eqn.PDE_Gauss}
	\partial_t \int \limits_{P_i} \Q_i \, d\xx = - \int \limits_{\partial P_i} \mathbb{F}_E(\Q) \cdot \nn \, dS,
\end{equation}
which indeed gives the explicit fluxes $\mathcal{H}(\Q_E)$ in \eqref{eqn.PDE_part}. The contribution of the explicit fluxes is then collected in the term $F_{\Q,i}^n$, which can be either $F_{\vv}^n$ in \eqref{eqn.vvisc} or $F_{\q}^n$ in \eqref{eqn.Poisson}: 
\begin{equation}
	\label{eqn.fvscheme}
	F_{\Q,i}^n = \Q_i^n - \frac{\dt}{|P_i|} \sum \limits_{e=1}^{\NE} \int \limits_{\partial {P,e}} \mathcal{F}(\w_e^{-,n},\w_e^{+,n},\norPE) \, dS.
\end{equation}
The Rusanov--type numerical flux $\mathcal{F}$ is a function of three arguments, namely the left and the right state $\w_e^{-,n},\w_e^{+,n}$, with respect to the edge $e$, and the associated outward pointing unit normal vector $\norPE$. The numerical flux is thus defined by
\begin{equation}
	\label{eqn.rusanov}
	\mathcal{F}(\w_e^{-,n},\w_e^{+,n},\norPE) = \frac{1}{2} \left( \mathbb{F}_E(\w_e^{-,n}) + \mathbb{F}_E(\w_e^{+,n}) \right) \cdot \norPE - \frac{1}{2} |s_{\max}| \left( \w_e^{+,n} - \w_e^{-,n} \right).
\end{equation}
Notice that the numerical dissipation $s_{\max}$ is given by the maximum eigenvalue of the convective sub-systems, namely \eqref{eqn.INS_eig} or \eqref{eqn.SWE_eig}, related to the right and left states. Consequently, the numerical diffusion of the scheme is proportional to the flow velocity only, which is a remarkable property in the context of low Froude number flows for the SWE model, since the acoustic-gravity wave speed does not play any role and the numerical dissipation automatically tends to zero for $\varepsilon \to 0$. The left and right states can be simply given by the finite volume solution stored as cell averages \eqref{eqn.cellAv} within each cell, being
\begin{equation}
	\label{eqn.states}
	\w_e^{-,n} = \Q_i^n, \qquad \w_e^{+,n}= \Q_{j}^n,
\end{equation}
where $j$ is the index of the polygon which shares edge $e$ with the cell $P_i$. Of course, the above choice yields only a first order in space accurate scheme. To improve the spatial order of accuracy, we perform a reconstruction technique that can effectively be used to fed the left and right state in the computation of the numerical flux function \eqref{eqn.rusanov}. 

\paragraph{CWENO reconstruction on polygonal meshes} The aim of the reconstruction operator is to determine the expansion coefficient of the reconstruction polynomial \eqref{eqn.recPoly}, that will be subsequently employed for the computation of the high order numerical fluxes in the finite volume scheme \eqref{eqn.rusanov}. We choose to use the class of Central WENO (CWENO) schemes, which has been originally proposed in \cite{LPR:99,LPR:2001}, and subsequently extended to deal with general polygonal meshes in \cite{CWENOGBK,FVBoltz,SIFVDG}. The interested reader is referred to the aforementioned references for all the details about the CWENO procedure, in particular we point out Appendix B in \cite{SIFVDG}. Eventually, we obtain a reconstruction polynomial within each cell of arbitrary degree $k$, that is expressed according to \eqref{eqn.recPoly}. This polynomial is then used to evaluate the high order extrapolated states $\w_e^{-,n}$ and $\w_e^{+,n}$ in the finite volume scheme \eqref{eqn.fvscheme}, hence improving the spatial accuracy of the scheme up to order $k+1$.

\subsection{Spatial virtual element discretization of the implicit terms}
The implicit discretization is based on higher-order conforming virtual element approximations on the polygonal tessellation. Hereafter, we detail the virtual element formulation for both the INS and the SWE models. 

\paragraph{Incompressible Navier-Stokes model}
We consider the implicit contributions in the semi-discrete scheme \eqref{eqn.INS_FOtime} on domain $\Omega$ and derive the weak form for both equations \eqref{eqn.vvisc} and \eqref{eqn.pnew_INS}, where pure Dirichlet boundary conditions ${\bm{\vv}^{\ast} = [g_x(\xv), g_y(\xv)]^T}$ and  ${p^{n+1} = h(\xv)}$ are applied on $\partial \Omega$. We explicitate the components of $\bm{\vv}^{\ast}$ on $\Omega$ by writing ${\bm{\vv}^{\ast} = [v^{\ast}_x, v^{\ast}_y]^T}$. Let us denote by $H^1_0(\Omega)$ the affine subspace of functions in the Sobolev space $H^1(\Omega)$ whose trace is zero on $\partial \Omega$. Moreover, we denote by $H^1_{g_x}(\Omega)$, $H^1_{g_y}(\Omega)$ and $H^1_h(\Omega)$ respectively the affine subspaces of functions in the Sobolev space $H^1(\Omega)$ whose trace is equal to $g_x(\xv)$, $g_y(\xv)$ and $h(\xv)$ on $\partial \Omega$. Multiplying component-wise the  momentum equation \eqref{eqn.vvisc} and the continuity equation \eqref{eqn.pnew_INS} by a test function $\varw \in H^1_0(\Omega)$, integrating over the domain $\Omega$ and applying the divergence theorem we get the following weak formulation for the INS model:
\begin{subequations}
	\label{ins_weak_form1} 
	\begin{equation}
		\textrm{find} \enskip v^{\ast}_x \in H^1_{g_x}(\Omega), \, v^{\ast}_y \in H^1_{g_y}(\Omega), \, p^{n+1} \in H^1_h(\Omega) \enskip \textrm{s.t.} \nonumber \\
	\end{equation}
	\begin{empheq}[left=\empheqlbrace]{align}
		\int_{\Omega} v^{\ast}_x \varw d\Omega + \Delta t \int_{\Omega} \nabla v^{\ast}_x \cdot \nabla{\varw}  d\Omega 
		= \int_{\Omega} (F_{\vv,x}^n - \Delta t \frac{\partial p^n}{\partial x} ) \varw d\Omega \quad \forall \varw \in H^1_0(\Omega), \label{ins_weak_form1a}  \\
		\int_{\Omega} v^{\ast}_y \varw d\Omega + \Delta t \int_{\Omega} \nabla v^{\ast}_y \cdot \nabla{\varw}  d\Omega 
		= \int_{\Omega} (F_{\vv,y}^n - \Delta t \frac{\partial p^n}{\partial y} ) \varw d\Omega \quad \forall \varw \in H^1_0(\Omega), \label{ins_weak_form1b}  \\
		\Delta t \int_{\Omega} \nabla p^{n+1} \cdot \nabla{\varw}  d\Omega 
		= -\int_{\Omega} (\nabla \cdot \vv^{\ast} + \Delta t \Delta p^n) \varw d\Omega \quad \forall \varw \in H^1_0(\Omega). \label{ins_weak_form1c} 	
	\end{empheq}
\end{subequations}
Given the polygonal tessellation of the space domain $\mathcal{T}_{\Omega}$ \eqref{trian}, the local virtual element approximations ${v^{\ast}_{x,h}, v^{\ast}_{y,h} \in V^h_k(\P)}$ of the velocity field components $v^{\ast}_x, v^{\ast}_y$ and the local virtual element approximation $p_h^{n+1}  \in V^h_k(\P)$ of the pressure field $p^{n+1}$ can be respectively represented, on each control volume $\P \in \mathcal{T}_{\Omega}$, in terms of the canonical basis $\{\varphi_i\}_{i=1}^{\NDOF}$ of the local virtual element space ${V}^h_k(\P)$, introduced in Section \ref{Sec_Vh}. Hence, we write
\begin{subequations}\label{ins_vem1} 
\begin{align}
	v^{\ast}_{x,h}(\xv)\rvert_{\xv\in\P} &= \sum_{i=1}^{\NDOF} \dof_i(v^{\ast}_{x,h})\varphi_i(\xv), \label{ins_vem1a}\\
	v^{\ast}_{y,h}(\xv)\rvert_{\xv\in\P} &= \sum_{i=1}^{\NDOF} \dof_i(v^{\ast}_{y,h})\varphi_i(\xv), \label{ins_vem1b}\\
	p_h^{n+1}(\xv)\rvert_{\xv\in\P} &= \sum_{i=1}^{\NDOF} \dof_i(p_h^{n+1})\varphi_i(\xv). \label{ins_vem1c}
\end{align}
\end{subequations}
Choosing $\varw = \varphi_j$ in equations \eqref{ins_weak_form1} and plugging \eqref{ins_vem1} into \eqref{ins_weak_form1}, we obtain the discrete version of the weak form on the cell $\P$:
\begin{subequations}
	\label{ins_discr_weak_form1} 
	\begin{align}
		\mathbf{M}_{\P} \hat{\vv}^{\ast}_{x,\P} + \Delta t \mathbf{K}_{\P} \hat{\vv}^{\ast}_{x,\P} = \mathbf{F}^{n}_{m_x,\P}, \\
		\mathbf{M}_{\P} \hat{\vv}^{\ast}_{y,\P} + \Delta t \mathbf{K}_{\P} \hat{\vv}^{\ast}_{y,\P} = \mathbf{F}^{n}_{m_y,\P}, \\
		\Delta t \mathbf{K}_{\P} \hat{\mathbf{p}}_{\P}^{n+1} = \mathbf{F}^{n}_{c,\P},
	\end{align}
\end{subequations}
where vectors $\hat{\vv}^{\ast}_{x,\P}$ and $\hat{\vv}^{\ast}_{y,\P}$ are the degrees of freedom of the velocity field components $v_{x,h}^{\ast}$ and $v_{y,h}^{\ast}$ on $\P$ and vector $\hat{\mathbf{p}}_{\P}^{n+1}$ collects the degrees of freedom of the pressure field $p_h^{n+1}$ on $\P$. In \eqref{ins_discr_weak_form1}, we also introduce the definition of the local mass and stiffness matrices $\mathbf{M}_{\P}$ and $\mathbf{K}_{\P}$, respectively:
\begin{align} 
	(\mathbf{M}_{\P})_{i,j} &= \int_{\P} \varphi_i \varphi_j  \, d\xx,  \label{ins_mass_mat}\\\
	(\mathbf{K}_{\P})_{i,j} &= \int_{\P} \nabla \varphi_i \cdot \nabla \varphi_j \, d\xx. \label{ins_stiffness_mat_c} 
\end{align}
Furthermore, the local load terms $\mathbf{F}^{n}_{m_x,\P},\mathbf{F}^{n}_{m_y,\P}$ and $\mathbf{F}^{n}_{c,\P}$ on the right-hand sides, which are known from the $n-$th time step, are given by
\begin{align} \label{ins_load_terms} 
	(\mathbf{F}^{n}_{m_x,\P})_i &= \phantom{-} \int_{\P} (F_{\vv,x}^n - \Delta t \frac{\partial p^n}{\partial x} ) \varphi_i \, d\xx,  \\
	(\mathbf{F}^{n}_{m_y,\P})_i &= \phantom{-} \int_{\P} (F_{\vv,y}^n - \Delta t \frac{\partial p^n}{\partial y} ) \varphi_i \, d\xx,  \\
	(\mathbf{F}^{n}_{c,\P})_i &= -\int_{\P} (\nabla \cdot \vv^{\ast} + \Delta t \Delta p^n) \varphi_i \, d\xx. 
\end{align}

Matrices $\mathbf{M}_{\P}$ and $\mathbf{K}_{\P}$, as well as the load terms $\mathbf{F}^{n}_{m_x,\P},\mathbf{F}^{n}_{m_y,\P}$ and $\mathbf{F}^{n}_{c,\P}$, are not readily computable since they involve functions that are unknown within the element interior. To approximate $\mathbf{M}_{\P}$, we observe that for each virtual basis function $\varphi_i(\xv)$ on $\P$ we can write the expansion
\begin{align*}
	\varphi_i = \projL \varphi_i + (I-\projL) \varphi_i,
\end{align*}
where $I$ is the identity operator, and plugging it into \eqref{ins_mass_mat} we obtain
\begin{align}\label{ins_proj_M} 
	(\mathbf{M}_{\P})_{i,j} = \int_{\P} \varphi_i  \varphi_j  \, d\xx = \int_{\P} \projL\varphi_i  \projL \varphi_j  \, d\xx + \int_{\P} (I-\projL)\varphi_i  (I-\projL) \varphi_j \, d\xx  + \nonumber \\ 
	+ \int_{\P} \projL\varphi_i  (I-\projL) \varphi_j \, d\xx + \int_{\P} (I-\projL)\varphi_i  \projL\varphi_j \, d\xx.
\end{align}
We note that the last two terms in \eqref{ins_proj_M} vanish due to the orthogonality condition \eqref{eq:vem_projL1}, yielding
\begin{align}\label{ins_proj_M2} 
	(\mathbf{M}_{\P})_{i,j} = \int_{\P} \projL\varphi_i \projL \varphi_j  \, d\xx + \int_{\P} (I-\projL)\varphi_i  (I-\projL) \varphi_j \, d\xx.
\end{align}
The first term in \eqref{ins_proj_M2} guarantees consistency and can be computed exactly, whereas the second one ensures stability and although it cannot be exactly computed, it can be approximated as follows:
\begin{align}\label{ins_proj_M3} 
	\int_{\P} (I-\projL)\varphi_i (I-\projL) \varphi_j \, d\xx \simeq 
	\abs{\P}\sum_{r=1}^{\NDOF} \dof_r((I-\projL)\varphi_i) \dof_r((I-\projL) \varphi_j).
\end{align}
Here, we adopted the so-called $\textrm{dof}_i-\textrm{dof}_i$ stabilization \cite{vem2}, although other choices are possible \cite{vem5,Mascotto2018}. From simple algebraic manipulations, recalling \eqref{Pi_0}, the sought approximation $\mathbf{M}_{\P}^h$ of matrix $\mathbf{M}_{\P}$ can be evaluated as the sum of a consistency term and a stability term, reading
\begin{align}\label{ins_proj_M4} 
	\mathbf{M}_{\P}^h = \mathbf{C}^{T}\mathbf{H}^{-1}\mathbf{C} + \abs{\P}(\mathbf{I}-\mathbf{\Pi}^0)^T(\mathbf{I}-\mathbf{\Pi}^0).
\end{align}
Analogously, the integral in the definition of the stiffness matrix in \eqref{ins_stiffness_mat_c} needs to be approximated, since its integrand is unknown in the element interior. Plugging the expansion
\begin{align*}
	\varphi_i = \proj \varphi_i + (I-\proj) \varphi_i,
\end{align*}
into \eqref{ins_stiffness_mat_c}, we obtain
\begin{align}\label{ins_proj_K} 
	(\mathbf{K}_{\P})_{i,j} = \int_{\P} \nabla \varphi_i \cdot \nabla \varphi_j \, d\xx = \int_{\P} \nabla \proj \varphi_i \cdot \nabla \proj \varphi_j  \, d\xx + \int_{\P} \nabla (I-\proj) \varphi_i \cdot \nabla (I-\proj)\varphi_j  \, d\xx  + \nonumber \\ 
	+ \int_{\P} \nabla(I-\proj) \varphi_i \cdot \nabla \proj \varphi_i \, d\xx + \int_{\P} \nabla \proj \varphi_j \cdot \nabla (I-\proj)\varphi_j \, d\xx.
\end{align}
Again, the last two terms in \eqref{ins_proj_K} vanish due to orthogonality conditions \eqref{eq:vem_proj}, leading to
\begin{align}\label{ins_proj_K2} 
	(\mathbf{K}_{\P})_{i,j} = \int_{\P} \nabla \proj \varphi_i \cdot \nabla \proj \varphi_j \, d\xx + \int_{\P} \nabla (I-\proj) \varphi_i \cdot \nabla (I-\proj)\varphi_j  \, d\xx.
\end{align}
Here, the first term ensures polynomial consistency and the second term is responsible for stability. In particular, the consistency term is exactly computable since it entails the integration of known polynomials, whilst the stability term must be approximated. For such approximation we adopt an expression analogue to \eqref{ins_proj_M3}, thus
\begin{align}\label{ins_proj_K3} 
	\int_{\P} \nabla (I-\proj) \varphi_i \cdot \nabla (I-\proj)\varphi_j \, d\xx \simeq 
	\abs{\P}\sum_{r=1}^{\NDOF} \dof_r((I-\proj) \varphi_i)\dof_r((I-\proj)\varphi_j) .
\end{align}
After some algebraic manipulations, recalling \eqref{Pi_nabla}, the approximation $\mathbf{K}_{\P}^h$ of matrix $\mathbf{K}_{\P}$ can be evaluated as the sum of a consistency term and a stability term as follows:
\begin{align}\label{ins_proj_K4} 
	\mathbf{K}_{\P}^h = (\mathbf{\Pi}^{\nabla}_{\ast})^{T}\mathbf{G}\mathbf{\Pi}^{\nabla}_{\ast} + \abs{\P}(\mathbf{I}-\mathbf{\Pi}^{\nabla})^T(\mathbf{I}-\mathbf{\Pi}^{\nabla}).
\end{align}
Finally, to compute the approximations $\mathbf{F}^{n,h}_{m_x,\P},\mathbf{F}^{n,h}_{m_y,\P}$ and $\mathbf{F}^{n,h}_{c,\P}$ of the load terms $\mathbf{F}^{n}_{m_x,\P},\mathbf{F}^{n}_{m_y,\P}$ and $\mathbf{F}^{n}_{c,\P}$, respectively, we resort to the $L_2$ projector $\projLm$ introduced in Section \ref{L2_proj_subs} and we write
\begin{align}\label{ins_proj_K5} 
	(\mathbf{F}^{n,h}_{m_x,\P})_i &=  \phantom{-} \int_{\P} (F_{\vv,x}^n - \Delta t \frac{\partial p^n}{\partial x} ) \projLm \varphi_i \, d\xx,  \\
	(\mathbf{F}^{n,h}_{m_y,\P})_i &=  \phantom{-} \int_{\P} (F_{\vv,y}^n - \Delta t \frac{\partial p^n}{\partial y} ) \projLm \varphi_i \, d\xx,  \\
	(\mathbf{F}^{n,h}_{c,\P})_i &= -\int_{\P} (\nabla \cdot \vv^{\ast} + \Delta t \Delta p^n) \projLm \varphi_i \, d\xx.
\end{align}
The implicit terms are then computed by solving the following virtual element linear system:
\begin{subequations}
	\label{ins_vem_system} 
	\begin{align}
		\mathbf{M}^h_{\P} \hat{\vv^{\ast}}_{x,\P} + \Delta t \mathbf{K}^h_{\P} \hat{\vv^{\ast}}_{x,\P} = \mathbf{F}^{n,h}_{m_x,\P}, \\
		\mathbf{M}^h_{\P} \hat{\vv^{\ast}}_{y,\P} + \Delta t \mathbf{K}^h_{\P} \hat{\vv^{\ast}}_{y,\P} = \mathbf{F}^{n,h}_{m_y,\P}, \\
		\Delta t \mathbf{K}^h_{\P} \hat{\mathbf{p}}_{\P}^{n+1} = \mathbf{F}^{n,h}_{c,\P}.
	\end{align}
\end{subequations}

\paragraph{Shallow water equations}
For the SWE model, we consider the implicit contributions in the semi-discrete scheme \eqref{eqn.SWE_FOtime}. The weak form for equation \eqref{eqn.Poisson}, where pure Dirichlet boundary conditions $\eta^{n+1} = \bar{\eta}(\xv)$ are applied on $\partial \Omega$, is now derived. We denote by $H^1_{\bar{\eta}}(\Omega)$ the affine subspace of functions in $H^1(\Omega)$ whose trace is equal to $\bar{\eta}(\xv)$ on $\partial \Omega$. Multiplying \eqref{eqn.Poisson} by a test function $\varw \in H^1_0(\Omega)$, integrating over the domain $\Omega$ and applying the divergence theorem, the weak formulation reads:
\begin{equation}
	\textrm{find} \enskip \eta^{n+1} \in H^1_{\bar{\eta}}(\Omega) \enskip \textrm{s.t.} \nonumber \\
\end{equation}
\begin{align}\label{swe_weak_form} 
	\int_{\Omega} \eta^{n+1} \varw  d\Omega  
	-\Delta t^2 g \int_{\Omega} H^n \nabla \eta^{n+1} \cdot \nabla \varw  d\Omega  = 
	\int_{\Omega} (\eta^{n}-\Delta t \nabla \cdot F^n_{\mathbf{q}}) \varw  d\Omega \quad \forall \varw \in H^1_0(\Omega).
\end{align}
Given the polygonal tessellation of the space domain $\mathcal{T}_{\Omega}$, the virtual element approximations of the free surface elevation $\eta_h$ and the total water depth $H_h$ can be represented, on each control volume $\P \in \mathcal{T}_{\Omega}$, in terms of the canonical basis functions $\{\varphi_i\}_{i=1}^{\NDOF}$ of the virtual element space $\Vh_k(\P)$ introduced in Section \ref{Sec_Vh}:
\begin{align}\label{swe_vem1} 
	\eta_h(\xv)\rvert_{\xv\in\P} &= \sum_{i=1}^{\NDOF} \dof_i(\eta_h)\varphi_i(\xv), \\
	H_h(\xv)\rvert_{\xv\in\P} &= \sum_{i=1}^{\NDOF} \dof_i(H_h)\varphi_i(\xv).
\end{align}
Choosing $\varw = \varphi_j$, and plugging \eqref{swe_vem1} into \eqref{swe_weak_form} we obtain the discrete version of the weak form on the cell $\P$, that is
\begin{align}\label{discr_weak_form} 
	\mathbf{M}_{\P} \hat{\bm{\eta}}^{n+1}_{\P} - \Delta t \mathbf{K}^n_{\P} \hat{\bm{\eta}}^{n+1}_{\P} = \mathbf{F}^n_{s,\P},
\end{align}
where $\hat{\bm{\eta}}^{n+1}_{\P}$ is the vector of degrees of freedom of $\eta_h^{n+1}$ on $\P$, while matrices $\mathbf{M}_{\P}$ and $\mathbf{K}_{\P}$ are defined as
\begin{align}
	(\mathbf{M}_{\P})_{i,j} &= \int_{\P} \varphi_i \varphi_j  \, d\xx, \label{matrix_M} \\
	(\mathbf{K}^n_{\P})_{i,j} &= \int_{\P} H^n \nabla \varphi_i \cdot \nabla \varphi_j \, d\xx, \label{matrix_K}
\end{align}
and the load term $\mathbf{F}^n_{s,\P}$ reads
\begin{align}
	(\mathbf{F}^n_{s,\P})_i = \int_{\P} (\eta^{n}-\Delta t \nabla \cdot F^n_{\mathbf{q}}) \varphi_i  \, d\xx.
\end{align}

Matrices $\mathbf{M}_{\P}$ and $\mathbf{K}_{\P}$, as well as the load term $\mathbf{F}^n_{s,\P}$, are not readily computable since the respective integrands are unknown within the element interior. 
For the mass matrix $\mathbf{M}_{\P}$ we provide the approximation $\mathbf{M}^h_{\P}$ already defined in \eqref{ins_proj_K4} for the INS equations. 
On the other hand, matrix $\mathbf{K}^n_{\P}$ in \eqref{matrix_K} differs from matrix $\mathbf{K}_{\P}$ derived in \eqref{ins_stiffness_mat_c} for the INS model, because it contains the non-constant diffusion coefficient $H^n$, which is obviously space dependent. As shown in \cite{Beirao_Brezzi_Marini_2016}, in such a case, using the elliptic projector $\proj$ to approximate $\nabla \varv_h$, $\varv_h \in \Vh_k(\P)$, would entail the loss of optimal convergence rates. Instead, as suggested in \cite{Beirao_Brezzi_Marini_2016}, we employ the $L_2$ projector operator $\projLm$ on polynomials of degree $k-1$ to approximate $\nabla \varv_h$ within the consistency term of the corresponding approximated bilinear form.
The computation of $\projLm \nabla \varv_h$ in terms of the degrees of freedom is carried out by constructing the matrix representation of the operator $\projLm \nabla$. Such representation can be expressed in terms of the two matrices $\accentset{\star}{\mathbf{\Pi}}^{0,x}_{k-1}$ and $\accentset{\star}{\mathbf{\Pi}}^{0,y}_{k-1}$ given by
\begin{align*}
	\accentset{\star}{\mathbf{\Pi}}^{0,x}_{k-1} &= \hat{\mathbf{H}}^{-1}\mathbf{E}^x, \\
	\accentset{\star}{\mathbf{\Pi}}^{0,y}_{k-1} &= \hat{\mathbf{H}}^{-1}\mathbf{E}^y, 
\end{align*}
where matrix $\hat{\mathbf{H}}$ is obtained by taking the first $n_{k-1}$ rows and columns of matrix $\mathbf{H}$ defined in \eqref{h_matrix}. Matrices $\mathbf{E}^x$ and $\mathbf{E}^y$ are defined as follows:
\begin{align*}
	(\mathbf{E}^x)_{i\alpha} = \int_{\P} \varphi_{i,x}m_\alpha \, d\xx, \quad\quad	(\mathbf{E}^y)_{i\alpha} = \int_{\P} \varphi_{i,y}m_\alpha \, d\xx, \quad \quad \alpha = 1,...,n_{k-1}.
\end{align*}
Hence, matrix $\mathbf{K}^n_{\P}$ can be approximated with the matrix $\mathbf{K}^{n,h}_{\P}$ which is computed as
\begin{align}
	\label{projL_K} 
	(\mathbf{K}^{n,h}_{\P})_{i,j} &= (\mathbf{K}^{n,h}_{c,\P})_{i,j} + (\mathbf{K}^{n,h}_{s,\P})_{i,j} = \nonumber \\
	&= \int_{\P} H^n  \projLm \nabla\varphi_i \cdot \projLm \nabla\varphi_j  \, d\xx \,+\, \bar{H}^n_h\mathcal{S}_{\P}\left((I-\proj)\varphi_i,  (I-\proj)\varphi_j\right),
\end{align}
where the first integral represents the consistency term $\mathbf{K}^{n,h}_{c,\P}$ and the second term is the stability term $\mathbf{K}^{n,h}_{s,\P}$. Moreover, $\bar{H}^n_h$ is the mean value of $H^n_h$ over $\P$ and $\mathcal{S}_{\P}$ is a suitable stabilizing bilinear form. After a few algebraic passages, it can be shown that in \eqref{projL_K} the consistency term can be computed from the degrees of freedom in the following manner:
\begin{align}
	\label{K_consist} 
	\mathbf{K}^{n,h}_{c,\P} = (\accentset{\star}{\mathbf{\Pi}}^{0,x}_{k-1})^T \mathbf{H}^H \accentset{\star}{\mathbf{\Pi}}^{0,x}_{k-1} + (\accentset{\star}{\mathbf{\Pi}}^{0,y}_{k-1})^T \mathbf{H}^H \accentset{\star}{\mathbf{\Pi}}^{0,y}_{k-1},
\end{align}
where matrix $\mathbf{H}^H$ is defined as:
\begin{align}\label{H-h} 
	(\mathbf{H}^H)_{\alpha,\beta} := \int_{\P} H^n m_\alpha m_\beta d\xx, \quad 1\leq \alpha,\beta \leq n_{k-1}.
\end{align}
In the stability term $\mathbf{K}^{n,h}_{s,\P}$, we choose to adopt the $\dof_i-\dof_i$ stabilization for the bilinear form $\mathcal{S}_{\P}$. Therefore, $\mathbf{K}^{n,h}_{s,\P}$ matrix is eventually determined as
\begin{align}
	\label{K_stab} 
	\mathbf{K}^{n,h}_{s,\P} = \bar{H} (\mathbf{I} - \mathbf{\Pi}^{\nabla}_{k})^T(\mathbf{I} - \mathbf{\Pi}^{\nabla}_{k}).
\end{align}
The implicit terms are then computed by solving the following virtual element linear system:
\begin{align}\label{swe_vem_system} 
	\mathbf{M}^h_{\P} \hat{\bm{\eta}}^{n+1}_{\P} - \Delta t \mathbf{K}^{n,h}_{\P} \hat{\bm{\eta}}^{n+1}_{\P} = \mathbf{F}^{n,h}_{s,\P}.
\end{align}

\subsection{Projection operators between finite volume and virtual element space} \label{ssec.operators}
Let us define the $L_2-$projection operators used to transfer data from one approximation space to the other and viceversa. In particular, at the beginning of the implicit step of the numerical scheme, data are available in the finite volume space as cell averages \eqref{eqn.cellAv}, thus they need to be projected onto the virtual element space to solve the elliptic equation, either \eqref{eqn.pnew_INS} or \eqref{eqn.Poisson}. To that aim, we first perform a CWENO reconstruction step \cite{SIFVDG,FVBoltz} to obtain a high order representation of the numerical solution given by \eqref{eqn.recPoly}, that is in terms of the Taylor modal basis functions \eqref{eqn.Voronoi_modal}. This is then mapped onto the VEM space by the $\NDOF \times n_k$ operator $\Vop_{\P}$, which is defined for each element as
\begin{equation}
	\Vop_{\P} = \left( \mathbf{M}_{\P}^h \right)^{-1} \, \int \limits_{\P}  \, \Pi^0_{\ast}\varphi_{i} \, m_{\beta} \, d\xx, \qquad i=1, \ldots,\NDOF, \quad \alpha=1,\ldots,n_k,
\end{equation}
where the VEM mass matrix is computed according to \eqref{ins_proj_M4}. Analogously, the inverse $n_k \times \NDOF$ operator $\Cop_{\P}$, which transfers data from the virtual element to the finite volume space at the end of the implicit step, writes
\begin{equation}
	\Cop_{\P} = \left( \int \limits_{P} m_{\alpha} m_{\beta} \, d\xx \right)^{-1} \, \int \limits_{\P}  m_{\alpha} \, \Pi^0_{\ast}\varphi_{i} \, d\xx, \qquad i=1, \ldots,\NDOF, \quad \alpha=1,\ldots,n_k.
\end{equation}
It is clear that these operators verify the consistency relations
\begin{equation}
	\Cop_{\P} \, \Vop_{\P} = \mathbf{I}_{[n_k \times n_k]}, \qquad \Vop_{\P} \, \Cop_{\P} = \mathbf{I}_{[\NDOF \times \NDOF]}.
\end{equation}

\section{Numerical results} \label{sec.numtest}
The novel numerical methods are applied to several test cases in order to thoroughly assess their convergence, stability and accuracy properties. Whenever possible, a comparison against the reference solution is proposed. The quantitative analysis is carried out by measuring the errors between the numerical and the exact solution, denoted by $q_h$ and $q_{ex}$, respectively, for a generic variable $q(\xx): \Omega \rightarrow \R^2$. The $L_2$ and $L_\infty$ norms are defined as
\[
L_2(q) = \sqrt{\int_\Omega (q_h-q_{ex})^2 \, \diff x } \qquad \text{and} \qquad L_\infty(q) = \max_\Omega | q_h-q_{ex} |.
\]
The novel hybrid semi-implicit finite volume/virtual element methods are labelled with SI-FVVEM. The CFL number is assumed to be $\text{CFL}=0.9$ in \eqref{eqn.timestep}. We remark that the time step is independent from the fast scales of the problem under consideration, which are discretized implicitly. We also highlight that if the velocity field is initialized with zero, the first time step is determined according to the CFL condition of fully explicit schemes, hence including the eigenvalues of the full system of governing equations. The first six applications refer to the SWE, while the remaining seven test problems deal with the INS equations. Different Froude numbers and viscosity coefficients are considered for shallow water and viscous flows, respectively, to numerically show that they do not affect neither the accuracy nor the efficiency of the proposed SI-FVVEM schemes. If not stated otherwise, the default configuration is the third order SI-FVVEM scheme in space and time.

\subsection{Convergence rates study (SWE)}

This subsection examines the numerical convergence of the SI-FVVEM schemes on the steady shallow water vortex initially proposed in \cite{Voronoi}. The computational domain is the square $\Omega = [-5; 5]^2$ with a flat bottom ($b = 0$), considering periodic boundary conditions. For this test case, the exact solution in polar coordinates is defined by the following free surface elevation and velocity field
\begin{equation}
	\label{eqn.SWEVortex-ini}
	\eta(\xx) = H_0 - \frac{1}{2 g} e^{-(r^2-1)}, \qquad \vv(\xx)= \begin{bmatrix}
		u(\xx)\\
		v(\xx)
	\end{bmatrix} = \begin{bmatrix}
		-u_{\alpha} \sin(\alpha) \\  
		\phantom{-}u_{\alpha} \cos(\alpha) 
	\end{bmatrix},
\end{equation} 
where the radius $r$ and the polar angle $\alpha$ are given by the inverse polar map
\[
r = \sqrt{x^2 + y^2} \quad \text{and} \quad \alpha = \atanpol \left( y, x \right).
\]
In \eqref{eqn.SWEVortex-ini}, the water depth for infinite radius is identified by $H_0$ and the angular velocity $u_\alpha=r e^{ -(r^2 - 1)/2 }$ allows the centrifugal forces ($u_\alpha / g r$) to balance the pressure forces ($\partial \eta / \partial r$). In this simulation, the initial conditions are given by the exact solution in \eqref{eqn.SWEVortex-ini}. The Froude number can be varied by adjusting the value of the constant $H_0$. In particular, a gravitational constant $g = 10$ is used in order to ease this computation. For polynomial space dimension $M=2,3$ the convergence is performed on four refined computational meshes, each with four different Froude numbers $\Fr \in \{0.32, 10^{-2}, 10^{-4}, 10^{-6}\}$ (corresponding to an asymptotic free surface $H_0 \in \{10^0, 10^{3}, 10^{7}, 10^{11}\}$, respectively). Each mesh is identified by the characteristic length $h(\Omega)$, taken as the maximum element size in the computational grid. The analysis takes into account the $L_2$-norm for the free surface $\eta$ and the velocity $u$ in $x-$direction at final time $t_f = 0.1$, with quadruple precision finite arithmetic \cite{Busto_SWE2022}. Table \ref{tab.conv_rate} sums up the results and reports the convergence rates. We can see that the formal order of accuracy is correctly achieved, and that the scheme is both asymptotic preserving and accurate, with the order of accuracy being independent of the Froude number.

\begin{table}[!htp]
	\begin{center}
		\caption{Numerical convergence results of the SI-FVVEM scheme with second and third order of accuracy in space and time using the steady shallow water vortex problem on Voronoi meshes. The errors are measured in $L_2$ norm and refer to the free surface elevation $\eta$ and velocity component $u$ at time $t_f=0.1$. The asymptotic preserving (AP) property of the schemes is studied by considering different Froude numbers $\Fr=\{ 10^{-1}, 10^{-2}, 10^{-4}, 10^{-6} \}$.}
		\small{
			\renewcommand{\arraystretch}{1.1}	
			\begin{tabular}{cccccccccc}
				& \multicolumn{4}{c}{SI-FVVEM $\mathcal{O}(2)$} & & \multicolumn{4}{c}{SI-FVVEM $\mathcal{O}(3)$} \\
				\cline{2-5} \cline{7-10}
				$h(\Omega)$ & ${L_2}(\eta)$ & $\mathcal{O}(\eta)$ & ${L_2}(u)$ & $\mathcal{O}(u)$ & & ${L_2}(\eta)$ & $\mathcal{O}(\eta)$ & ${L_2}(u)$ & $\mathcal{O}(u)$ \\
				\hline
				& \multicolumn{9}{c}{$\Fr=10^{-1}$ (double precision)} \\
				\hline
				5.0184E-01 & 6.5520E-04 & -    & 4.3967E-02 & -    & & 5.5807E-04 & -    & 4.1527E-02 & -     \\
				2.4761E-01 & 1.2273E-04 & 2.37 & 9.5032E-03 & 2.17 & & 8.0381E-05 & 2.74 & 5.9083E-03 & 2.76  \\
				1.6167E-01 & 5.3189E-05 & 1.96 & 4.3203E-03 & 1.85 & & 2.7156E-05 & 2.55 & 1.9278E-03 & 2.63  \\
				1.2464E-01 & 2.9092E-05 & 2.32 & 2.4833E-03 & 2.13 & & 1.0751E-05 & 3.56 & 8.4473E-04 & 3.17  \\
				\hline
				& \multicolumn{9}{c}{$\Fr=10^{-2}$ (quadruple precision)} \\
				\hline
				5.0184E-01 & 6.7318E-06 & -    & 4.4214E-02 & -    & & 5.7630E-06 & -    & 4.1836E-02 & -     \\
				2.4761E-01 & 1.6369E-06 & 2.00 & 9.6380E-03 & 2.16 & & 8.0877E-07 & 2.78 & 5.9892E-03 & 2.77  \\
				1.6167E-01 & 7.4398E-07 & 1.85 & 4.5656E-03 & 1.75 & & 2.4626E-07 & 2.79 & 1.8936E-03 & 2.70  \\
				1.2464E-01 & 4.0295E-07 & 2.36 & 2.5566E-03 & 2.23 & & 9.7144E-08 & 3.58 & 8.5621E-04 & 3.05  \\
				\hline
				& \multicolumn{9}{c}{$\Fr=10^{-4}$ (quadruple precision)} \\
				\hline
				5.0184E-01 & 6.7348E-10 & -    & 4.4216E-02 & -    & & 5.7672E-10 & -    & 4.1839E-02 & -     \\
				2.4761E-01 & 1.7694E-10 & 1.89 & 9.6401E-03 & 2.16 & & 8.1477E-11 & 2.77 & 5.9903E-03 & 2.75  \\
				1.6167E-01 & 7.7779E-11 & 1.93 & 4.5693E-03 & 1.75 & & 2.4780E-11 & 2.79 & 1.8942E-03 & 2.70  \\
				1.2464E-01 & 4.0266E-11 & 2.53 & 2.5578E-03 & 2.23 & & 9.7662E-12 & 3.58 & 8.5635E-04 & 3.05  \\
				\hline
				& \multicolumn{9}{c}{$\Fr=10^{-6}$ (quadruple precision)} \\
				\hline
				5.0184E-01 & 6.7386E-14 & -    & 4.4281E-02 & -    & & 5.7694E-14 & -    & 4.1908E-02 & -     \\
				2.4761E-01 & 1.7780E-14 & 1.89 & 9.7891E-03 & 2.14 & & 8.5738E-15 & 2.78 & 7.0563E-03 & 2.52  \\
				1.6167E-01 & 7.7780E-15 & 1.94 & 4.5707E-03 & 1.79 & & 2.5189E-15 & 2.79 & 2.3166E-03 & 2.61  \\
				1.2464E-01 & 4.0491E-15 & 2.51 & 2.5657E-03 & 2.22 & & 9.9750E-16 & 3.58 & 1.0207E-04 & 3.15  \\
			\end{tabular}
		}
		\label{tab.conv_rate}
	\end{center}
\end{table}
%
%
\subsection{Well-balance test (SWE)}

To test the well-balance property of the scheme (also known as the C-property) proven in Theorem \ref{th_swe_1}, we use a benchmark proposed in \cite{WBLeVeque}. This test case assesses whether a numerical scheme can maintain stationary equilibrium solutions of the governing equations up to machine precision. The particular water-at-rest equilibrium solution of the SWE is characterized by a constant free surface elevation $\eta(\xx,t)=0$ and zero fluid velocity, i.e., $\vv(\xx,t)=\mathbf{0}$, with an arbitrary bottom topography different from the trivial profile $b(\xx)= 0$. Following \cite{WBLeVeque}, we use a computational domain $\Omega=[-2;1]\times[-0.5;0.5]$ with Dirichlet boundary conditions in the $x-$direction and periodic boundaries along the $y-$direction. The domain is paved with $N_P=8633$ Voronoi cells of mesh size $h=1/50$. The bathymetry and the initial free surface elevation are defined as follows:
\begin{equation}
	\label{eqn.WBini}
	b(\xx) = 0.5 \cdot e^{-5 \, (x+0.1)^2 - 50 y^2}, \qquad \eta(\xx,0) = \left\{ \begin{array}{lc}
		1 + \delta & \textnormal{ if } -0.95 \leq x \leq -0.85 \\
		1 & \textnormal{elsewhere} 
	\end{array}\right. .
\end{equation} 
Here, the fluid is initially at rest, and we first set the perturbation amplitude $\delta=0$. The simulation is run until the final time $t_f=0.1$, using double and quadruple finite arithmetic. The errors with respect to the initial condition are reported in Table \ref{tab.Cproperty}. The results show that the scheme is well-balanced up to machine accuracy.
\begin{table}[!htp]
	\begin{center}
		\caption{Well-balance test with double and quadruple finite arithmetic precision. Errors measured in $L_2$ and $L_{\infty}$ norms for the free surface elevation $\eta$ and momentum component $Hu$ at the final time $t_f=0.1$.}
		\begin{tabular}{lcccc}
			Precision & $L_2(\eta)$ & $L_{\infty}(\eta)$ & $L_2(Hu)$ & $L_{\infty}(Hu)$ \\
			\hline
			Double    & 1.5221E-15 & 3.5527E-15  & 2.010E-18 & 9.7600E-17  \\
			Quadruple & 9.3593E-34 & 4.5259E-33  & 4.833E-22 & 1.3078E-20
		\end{tabular}
		\label{tab.Cproperty}
	\end{center}
\end{table}

In accordance with \cite{WBLeVeque}, a slight perturbation is introduced in the free surface elevation by setting $\delta=10^{-2}$ in \eqref{eqn.WBini}. To properly track the wave propagation, a constant time step of $\dt=0.01$ is employed. Figure \ref{fig.Ctest2} displays the results at various output times, demonstrating that the presence of the bottom bump does not generate any spurious oscillations. Moreover, the flow structure agrees well with the outcomes found in the literature \cite{TavelliSWE2014,Canestrelli2010,Busto_SWE2022}.
\begin{figure}[!htbp]
	\begin{center}
		\begin{tabular}{cc}
			\includegraphics[width=0.49\textwidth]{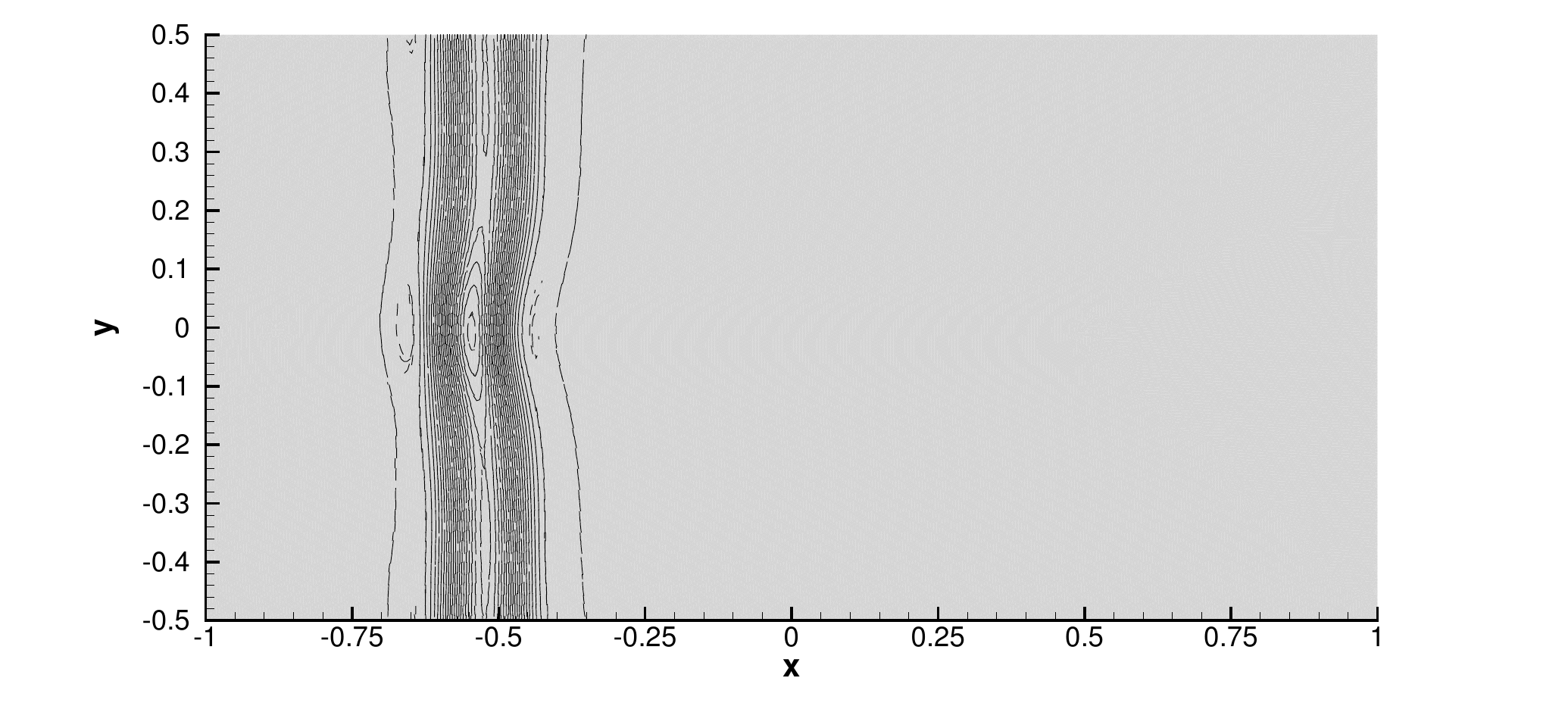}  &          
			\includegraphics[width=0.49\textwidth]{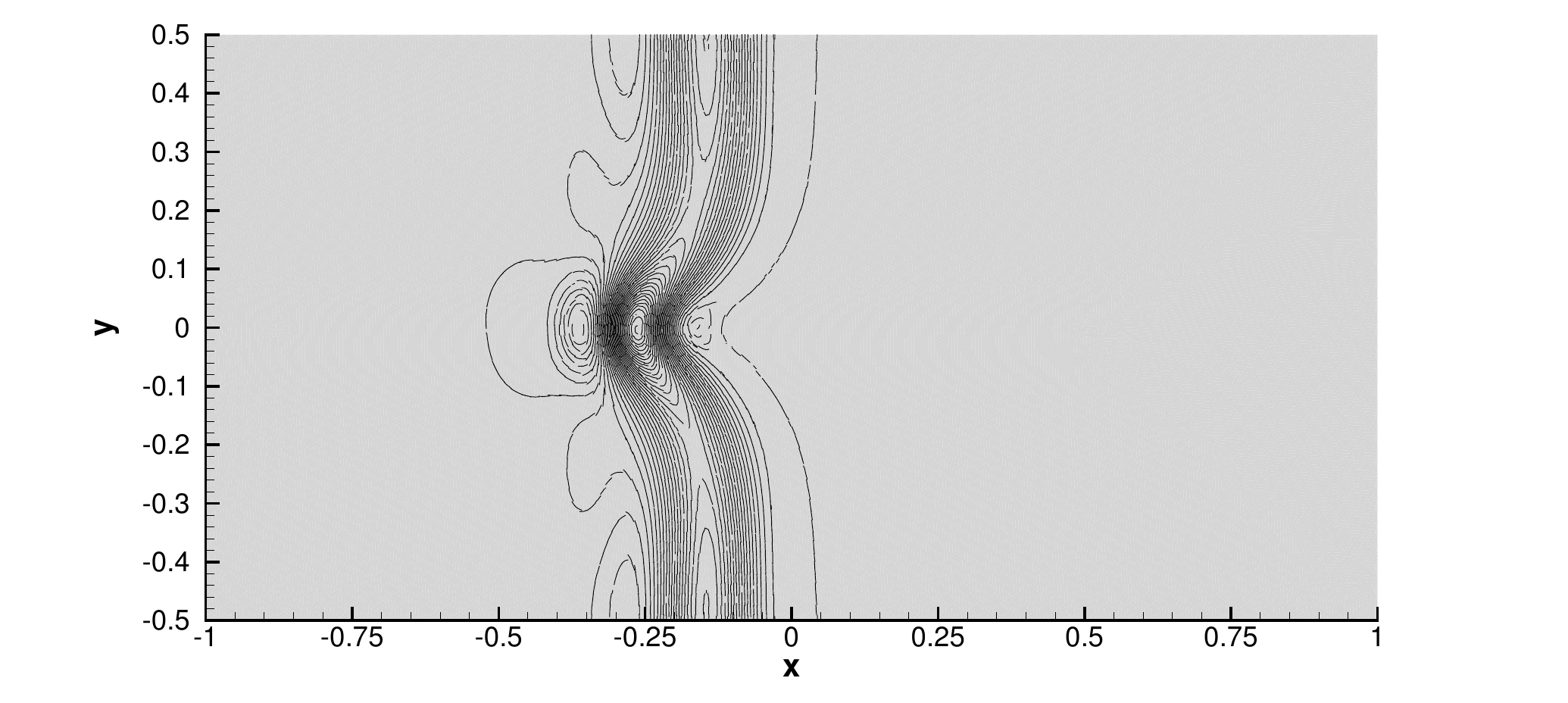}    \\
			\includegraphics[width=0.49\textwidth]{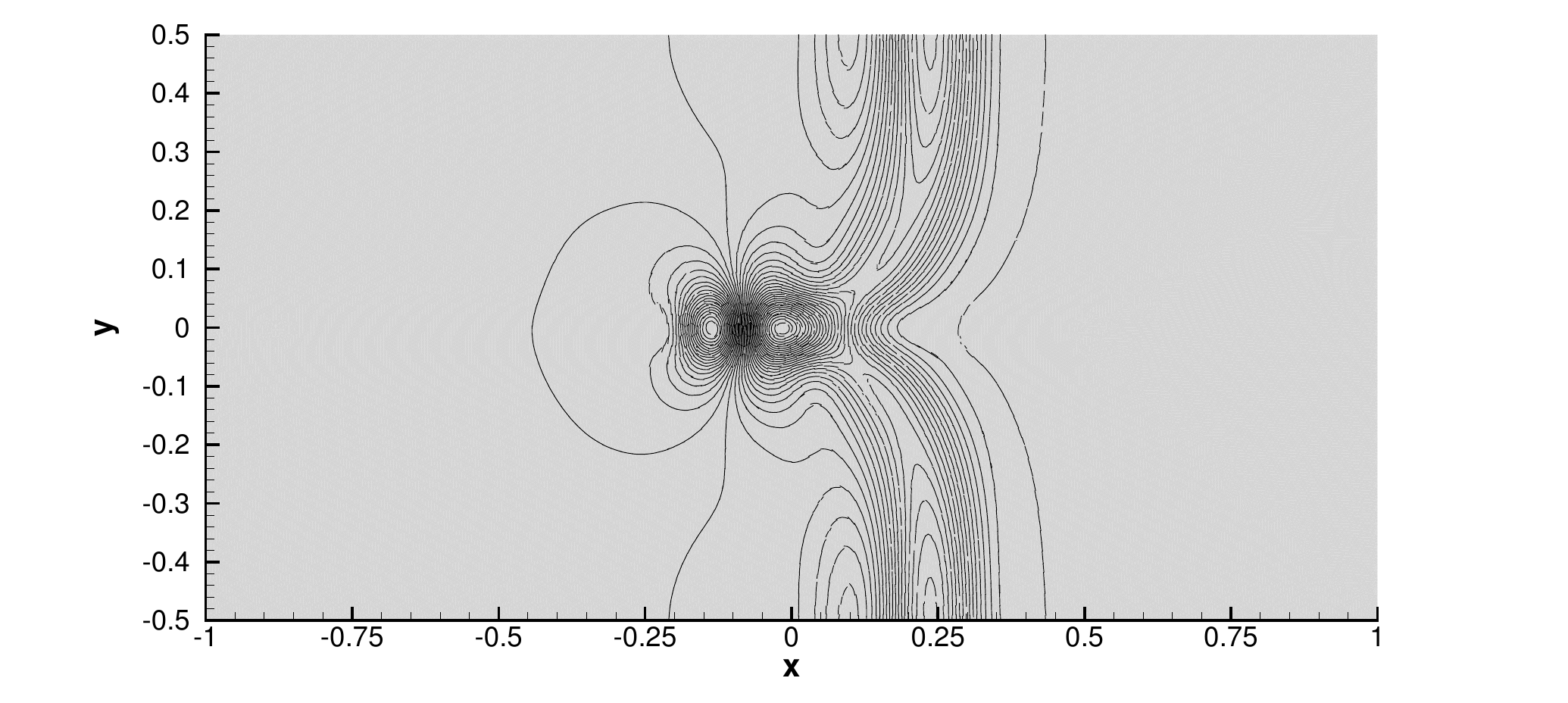}  &          
			\includegraphics[width=0.49\textwidth]{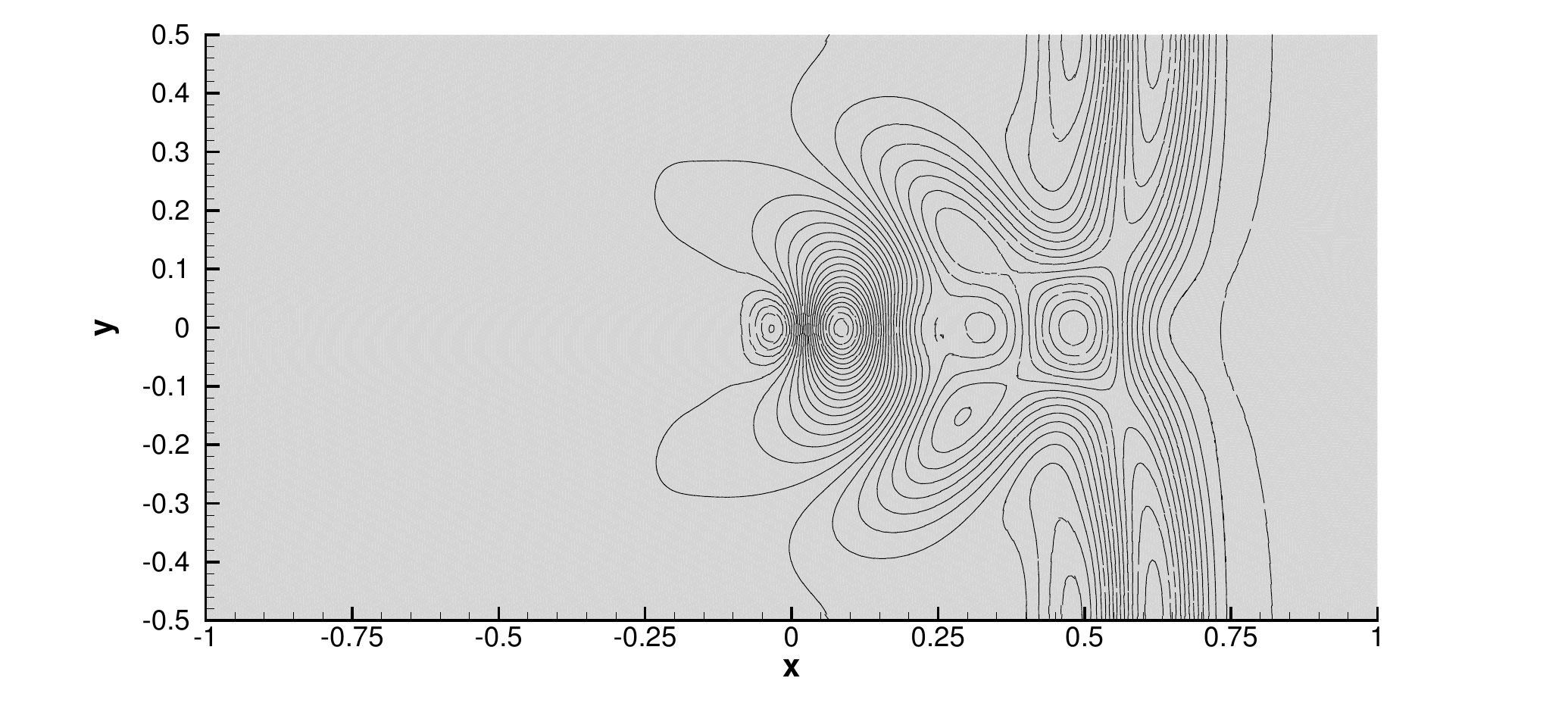}    \\
		\end{tabular}
		\caption{Well-balance test with small perturbation of the free surface ($\delta=10^{-2}$). 50 equidistant contour lines in the interval $\eta=[0.995;1.008]$ are shown at output times $t=0.12$, $t=0.24$, $t=0.36$ and $t=0.48$ (from top left to bottom right panel).}
		\label{fig.Ctest2}
	\end{center}
\end{figure}

\subsection{Circular dambreak (SWE)}

To simulate shock waves, we will use the circular dam break problem over a bottom step presented in \cite{TavelliSWE2014,StagDG_Dumbser2013}. The computational domain is a circle with radius $r=|\xx|\leq 2$, denoted by $\Omega=\{\xx \in \mathds{R}^2 : \, r=|\xx|\leq 2\}$, with Dirichlet boundary conditions everywhere. The computational grid has a total number of $N_P=34477$ cells, with a characteristic mesh size $h=1/50$. The initial condition for the problem reads
\begin{equation}
	\eta(\xx,0) = \left\{ \begin{array}{cc}
		1.0 & \textnormal{ if } r \leq 1 \\
		0.5 & \textnormal{ if } r > 1
	\end{array}\right., \qquad b(\xx) = \left\{ \begin{array}{cc}
		0.2 & \textnormal{ if } r \leq 1 \\
		0.0 & \textnormal{ if } r > 1
	\end{array}\right., \qquad \vv(\xx,0) = \mathbf{0}.
\end{equation}
The simulation is carried out until the final time $t_f=0.2$, at which the solution exhibits a contact wave travelling towards the centre of the domain, as well as a shock wave heading towards the outer boundary. Additionally, due to the presence of the bottom step, there is a discontinuity in the flow at $r=1$. Figure \ref{fig.dambreak} depicts the results of the simulation along with a comparison against a reference solution. The reference solution has been computed by solving the one-dimensional (1D) SWE in the radial direction with geometric reaction source terms, using a classical shock capturing MUSCL-TVD finite volume scheme with 10000 cells \cite{ToroBook}. The figure shows that there is a very good agreement between the numerical and reference solutions, with no spurious oscillations occurring in the plateau between the different discontinuities. Note that numerical dissipation is only present in the finite volume solver for the convective terms and not in the pressure solver \cite{SIFVDG}, as required in \cite{Busto_SWE2022}.

\begin{figure}[!htbp]
	\begin{center}
		\begin{tabular}{cc}
			\includegraphics[width=0.47\textwidth]{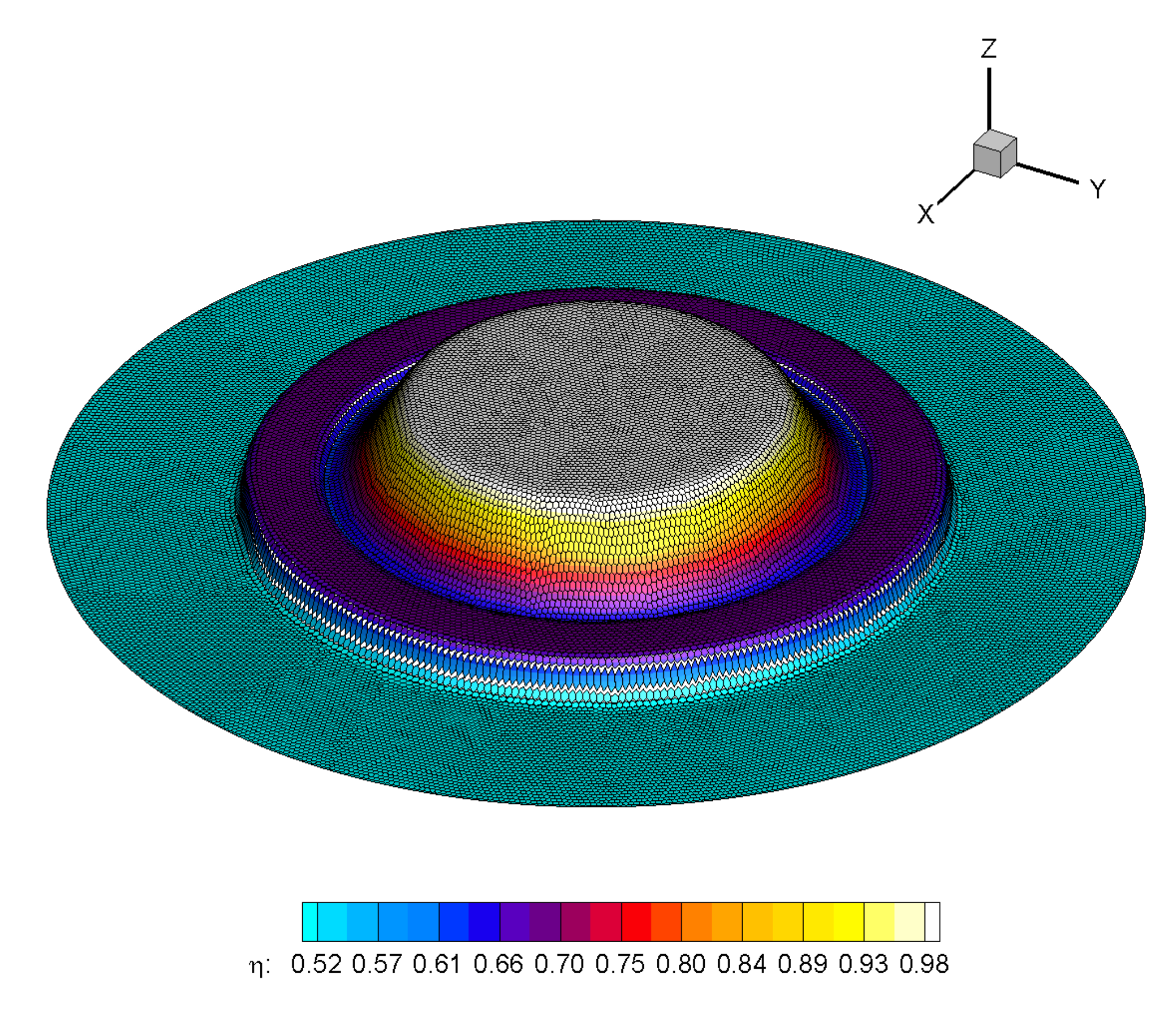} &
			\includegraphics[width=0.47\textwidth]{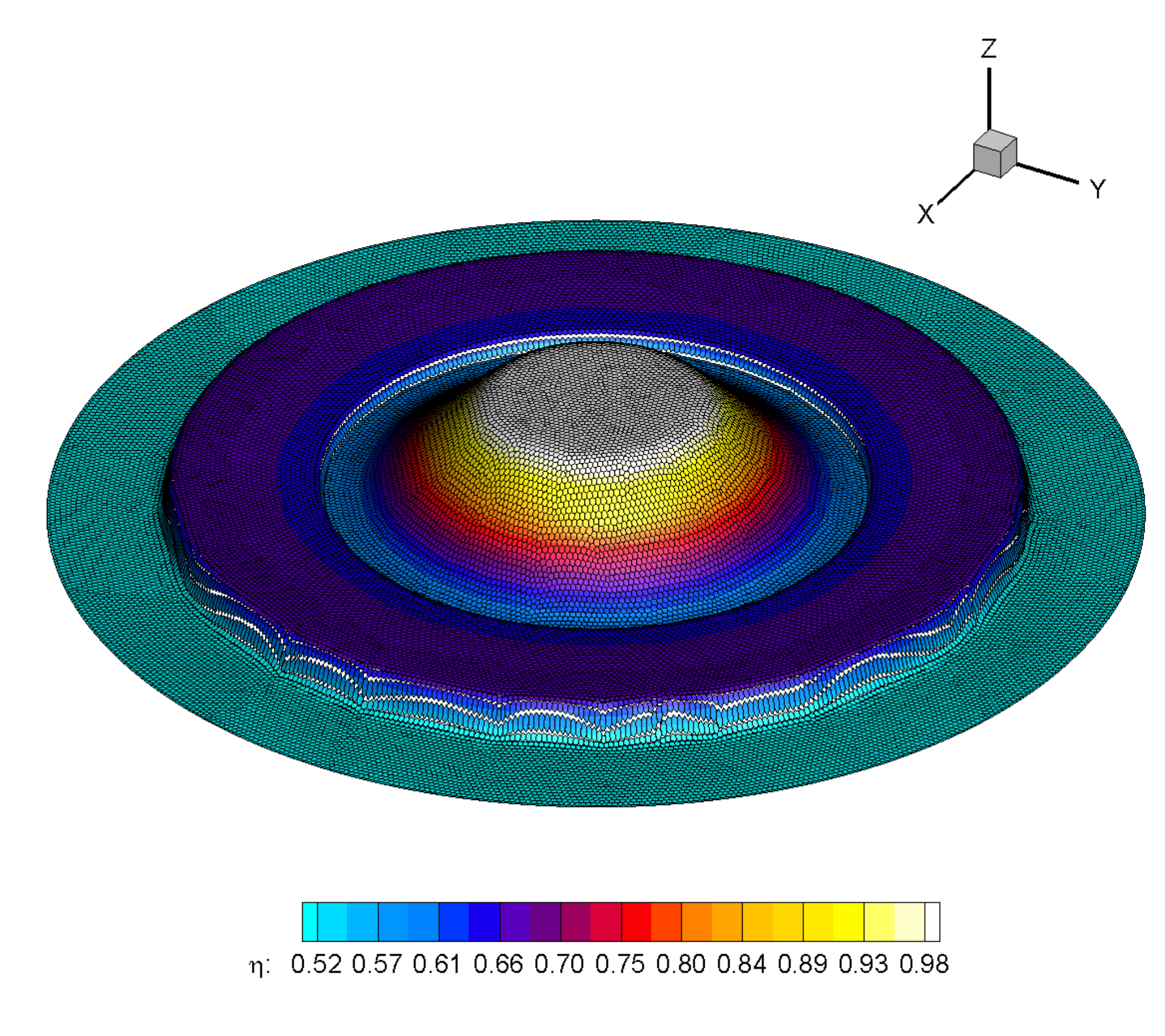} \\   	   
			\includegraphics[width=0.47\textwidth]{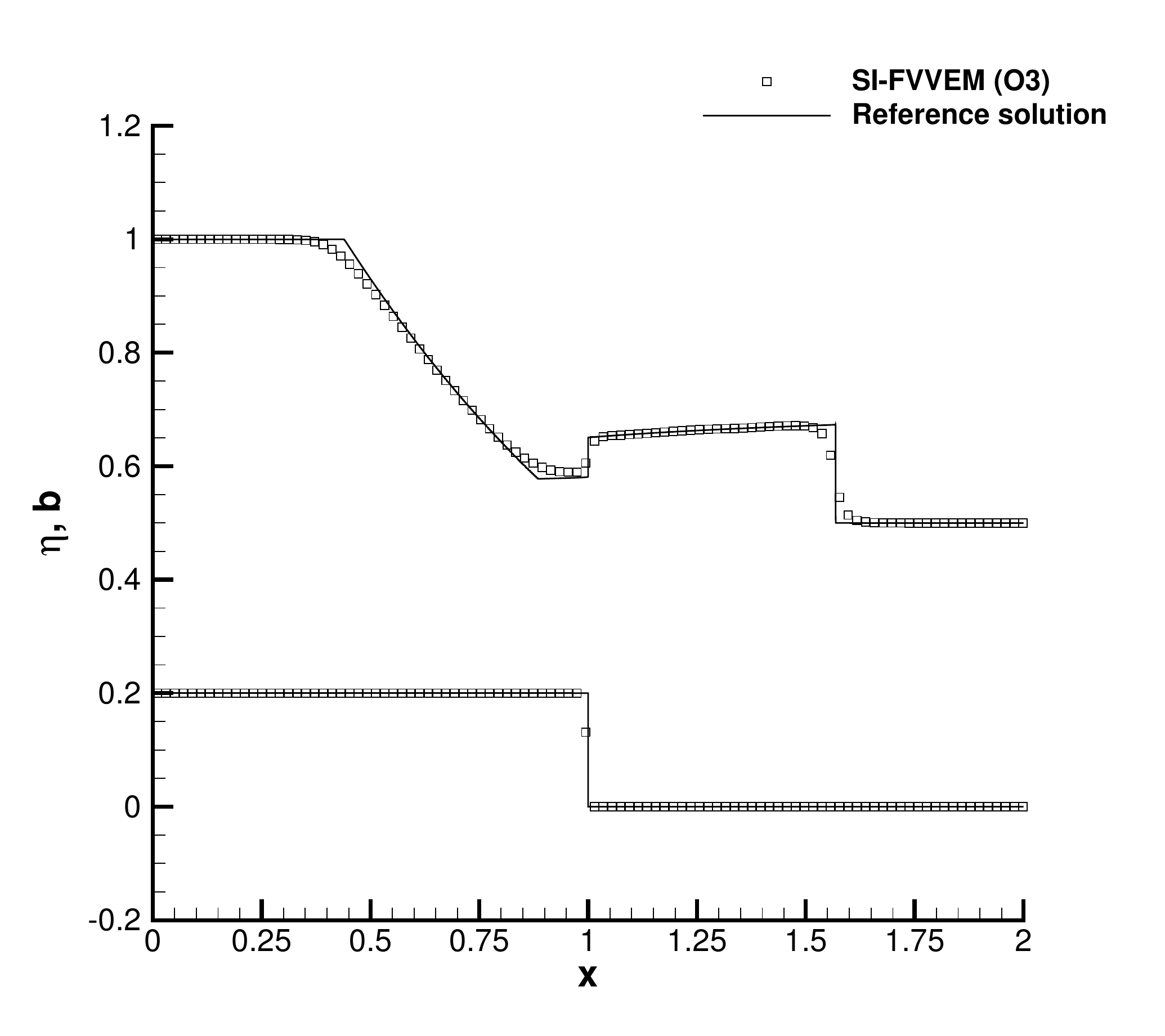} &  			\includegraphics[width=0.47\textwidth]{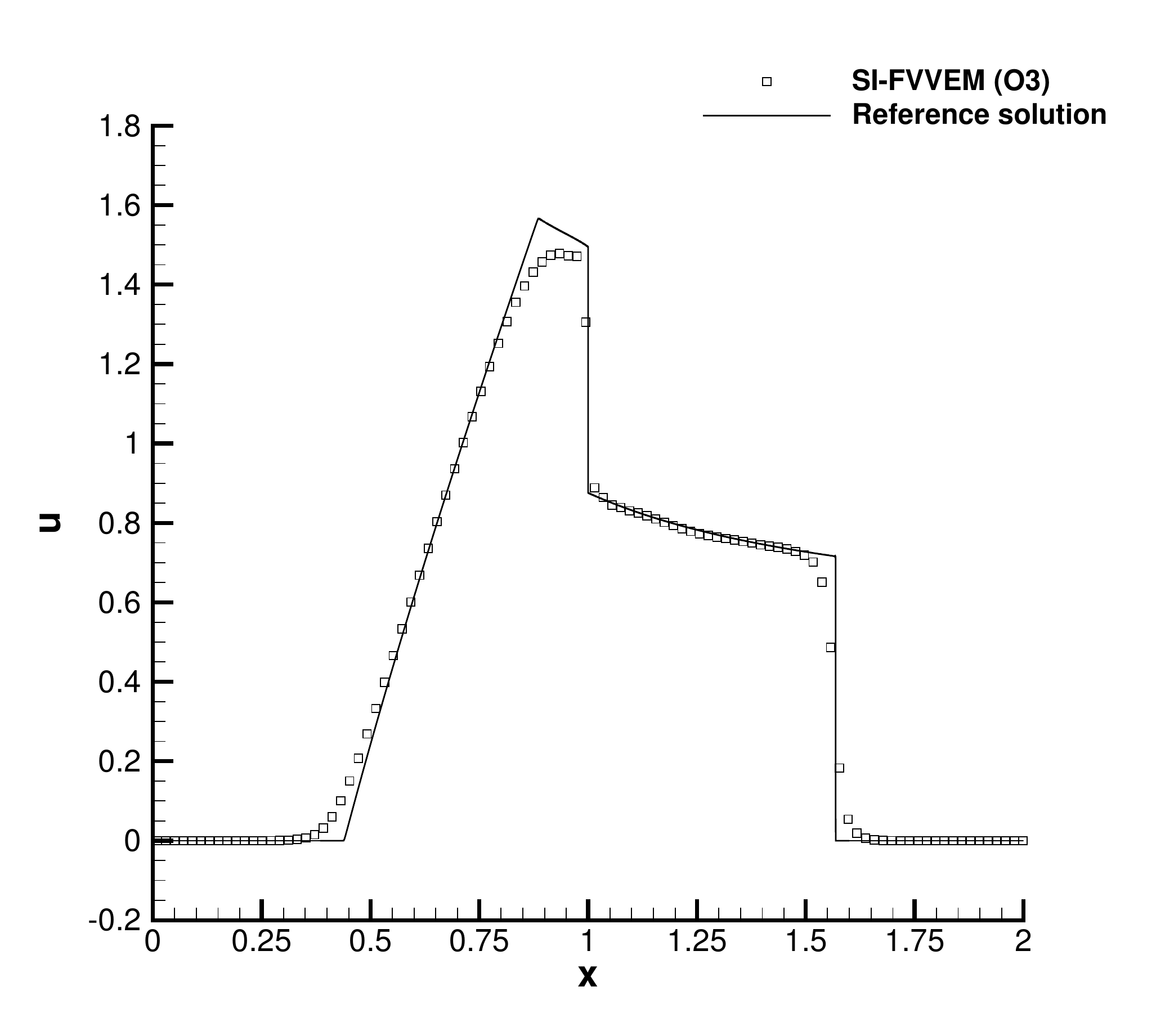} \\
		\end{tabular}
		\caption{Circular dam break problem. Top: 3D view of the free surface elevation at time $t=0.1$ (left) and $t=0.2$ (right). Bottom: 1D cut along the line $y=0$ of the numerical solution compared with the reference solution at time $t_f=0.2$ for the free surface and bottom profile (left) as well as for the horizontal velocity component (right).}
		\label{fig.dambreak}
	\end{center}
\end{figure}

\subsection{Riemann problems (SWE)}

To verify shock-capturing and conservation properties, a set of Riemann problems is considered, which include both flat and variable bottom topography. The exact solution is obtained using the Riemann solver proposed in \cite{ToroBookSWE} and \cite{ToroSWERP} for flat and variable bottom, respectively. The initial condition is given in terms of two states $\Q_L=(\eta_L,u_L,b_L)$ and $\Q_R=(\eta_R,u_R,b_R)$ separated by a discontinuity located at position $x=x_d$:
\begin{equation}
	\Q(\xx,0) = \left\{ \begin{array}{ccc} \Q_L & \textnormal{if} & x \leq x_d \\
		\Q_R & \textnormal{if} & x > x_d
	\end{array}\right. .
\end{equation}
Data of the computational domain and the initial condition for the chosen four Riemann problems are summarized in Table \ref{tab:initRP}. The computational domain in these test cases is $\Omega=[x_L, x_R] \times [x_L/10, x_R/10]$, and it is discretized using an unstructured Voronoi mesh of size $h$. This makes the computation multidimensional, despite the 1D setup of the test problems. 
%
\begin{table}[!htbp]  
	\caption{Initialization of Riemann problems. Initial left (L) and right (R) states are reported as well as the final time of the simulation $t_f$, the computational domain $[x_L;x_R]$, the position of the initial discontinuity $x_d$ and the characteristic mesh size $h$.}  
	\begin{center} 
		\begin{small}
			\renewcommand{\arraystretch}{1.0}
			\begin{tabular}{llllllllllll} 
				\hline
				Test & $\eta_L$ & $u_L$ & $b_L$ & $\eta_R$ & $u_R$ & $b_R$ & $x_L$ & $x_R$ & $x_d$ & $h$ & $t_{f}$\\
				\hline
				RP1 \cite{ToroBookSWE}  & 1 & 0 & 0 & 2 & 0 & 0 & -0.5 & 0.5 & 0 & 1/200 & 0.075\\
				RP2                 & $10^3$ & 0 & 0 & 1 & 0 & 0 & -15 & 15 & 0 & 1/400 & 0.09\\ 
				RP3 \cite{ToroSWERP} & 1 & 0 & 0.2 & 0.5 & 0 & 0 & -5 & 5 & 0 & 1/200 & 1\\
				RP4 \cite{ToroSWERP}  & 1.46184 & 0 & 0 & 0.30873 & 0 & 0.2 & -5.0 & 5.0 & 0 & 1/200 & 1\\ 
				\hline
			\end{tabular}
		\end{small}
	\end{center}
	\label{tab:initRP}
\end{table}
%

Figure \ref{fig.RP} shows the comparison between the numerical and the reference solution through a 1D cut of 200 equidistant points along the $x-$axis of the computational domain at $y=0$. The first two Riemann problems (RP1 and RP2) assume a constant flat bathymetry and involve shock and rarefaction waves. The remaining Riemann problems (RP3 and RP4) are concerned with a jump in the bottom elevation of height $\Delta b=0.2$, which generates contact discontinuities. The SI-FVVEM scheme exhibits an excellent agreement and can handle supercritical flows with Froude numbers greater than one. For RP2, the maximum Froude number is $\Fr=5.73$. The FV strategy used to discretize the non-linear convective terms makes the scheme robust. The implicit treatment of the free surface elevation still ensures a stable scheme for all four Riemann problems, without any need of additional numerical dissipation. We also remark that the proposed approach is conservative by construction, correctly capturing moving shocks and the values of the plateau between two discontinuities. The 1D symmetry of the problem is preserved even in the context of arbitrary shaped polygonal cells, as depicted in the 3D views of the free surface elevation in Figure \ref{fig.RP}.
\begin{figure}[!htbp]
	\begin{center}
		\begin{tabular}{ccc} 
			\includegraphics[width=0.33\textwidth]{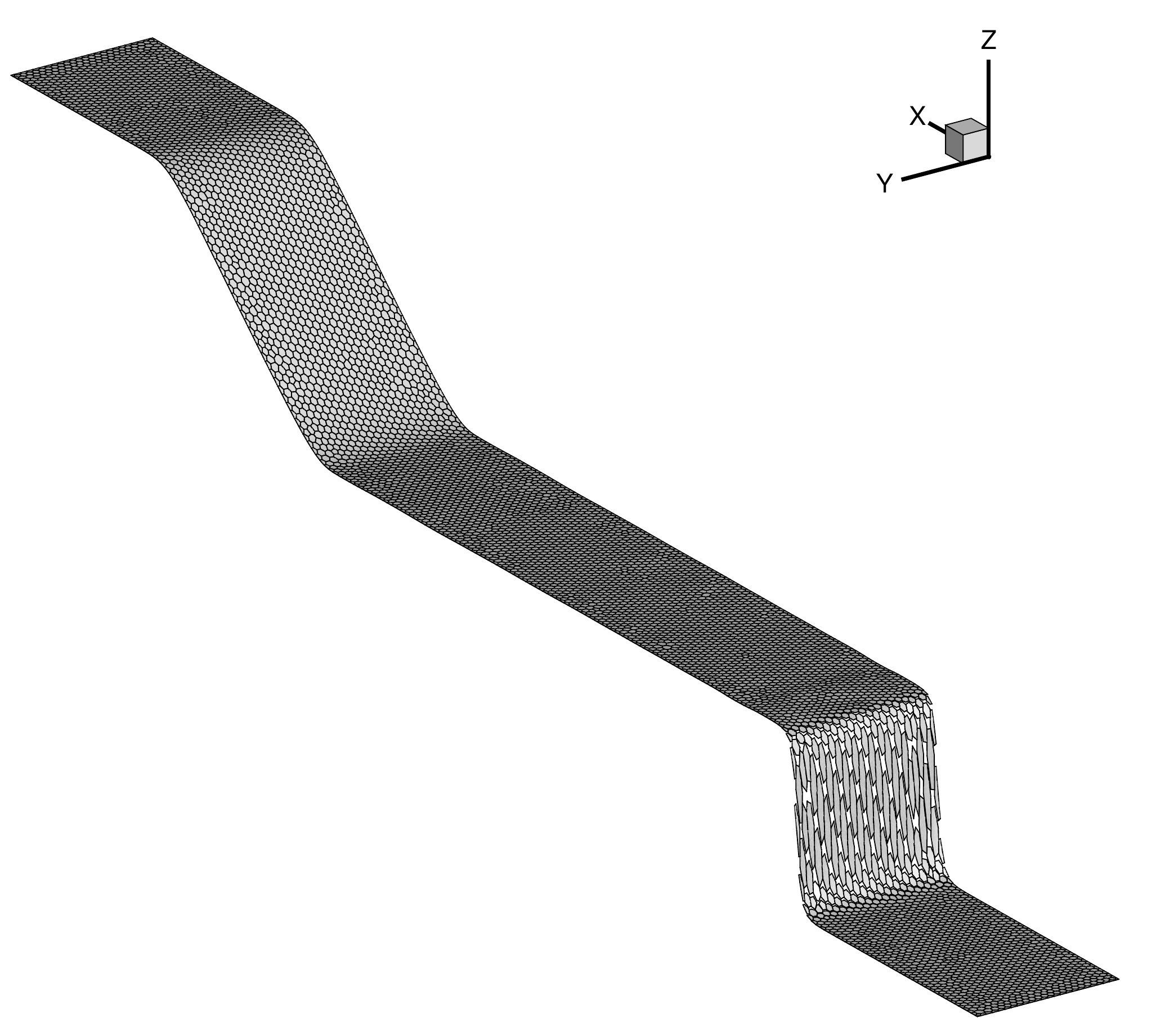} &  
			\includegraphics[width=0.33\textwidth]{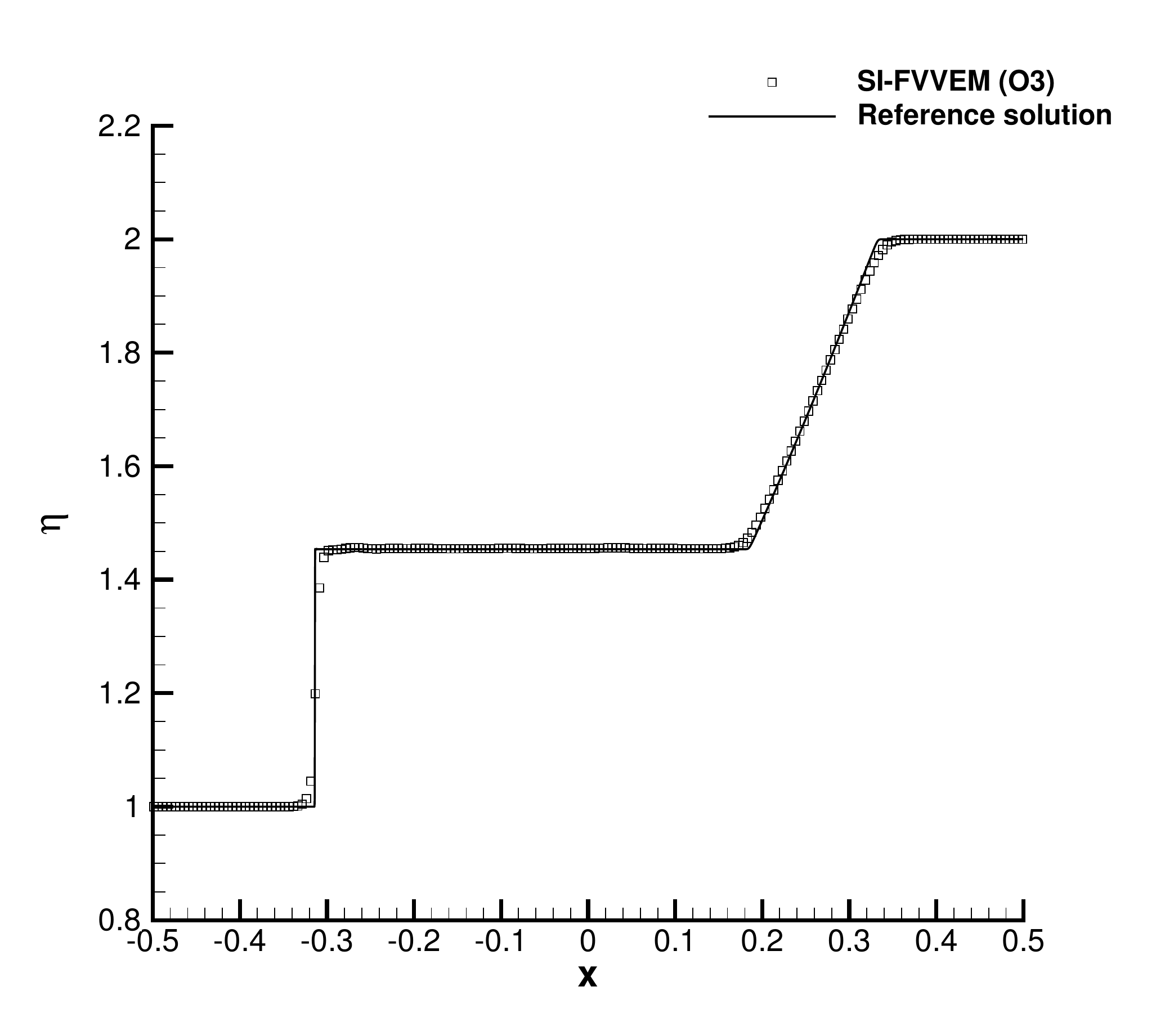} &  			\includegraphics[width=0.33\textwidth]{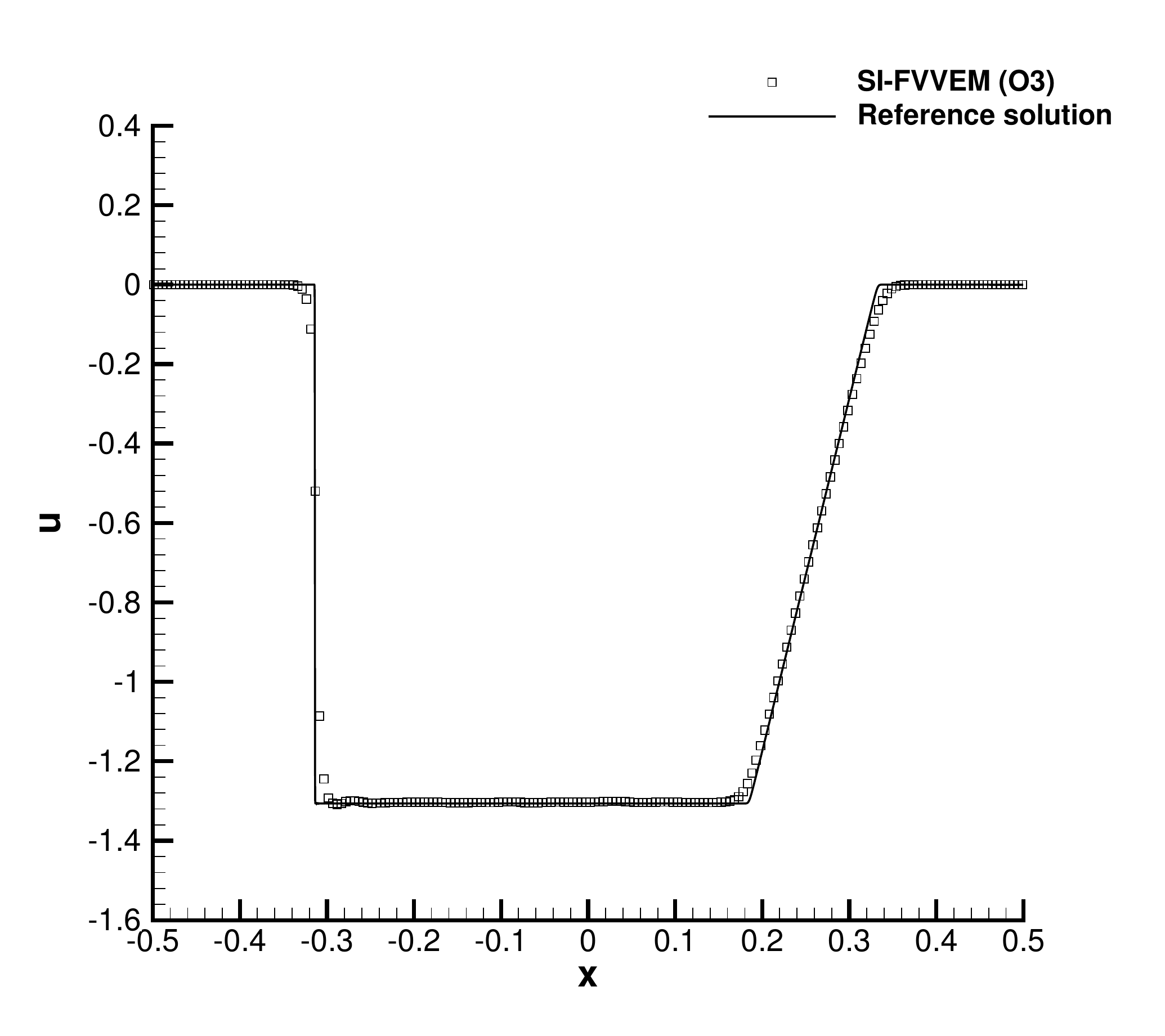} \\
			\includegraphics[width=0.33\textwidth]{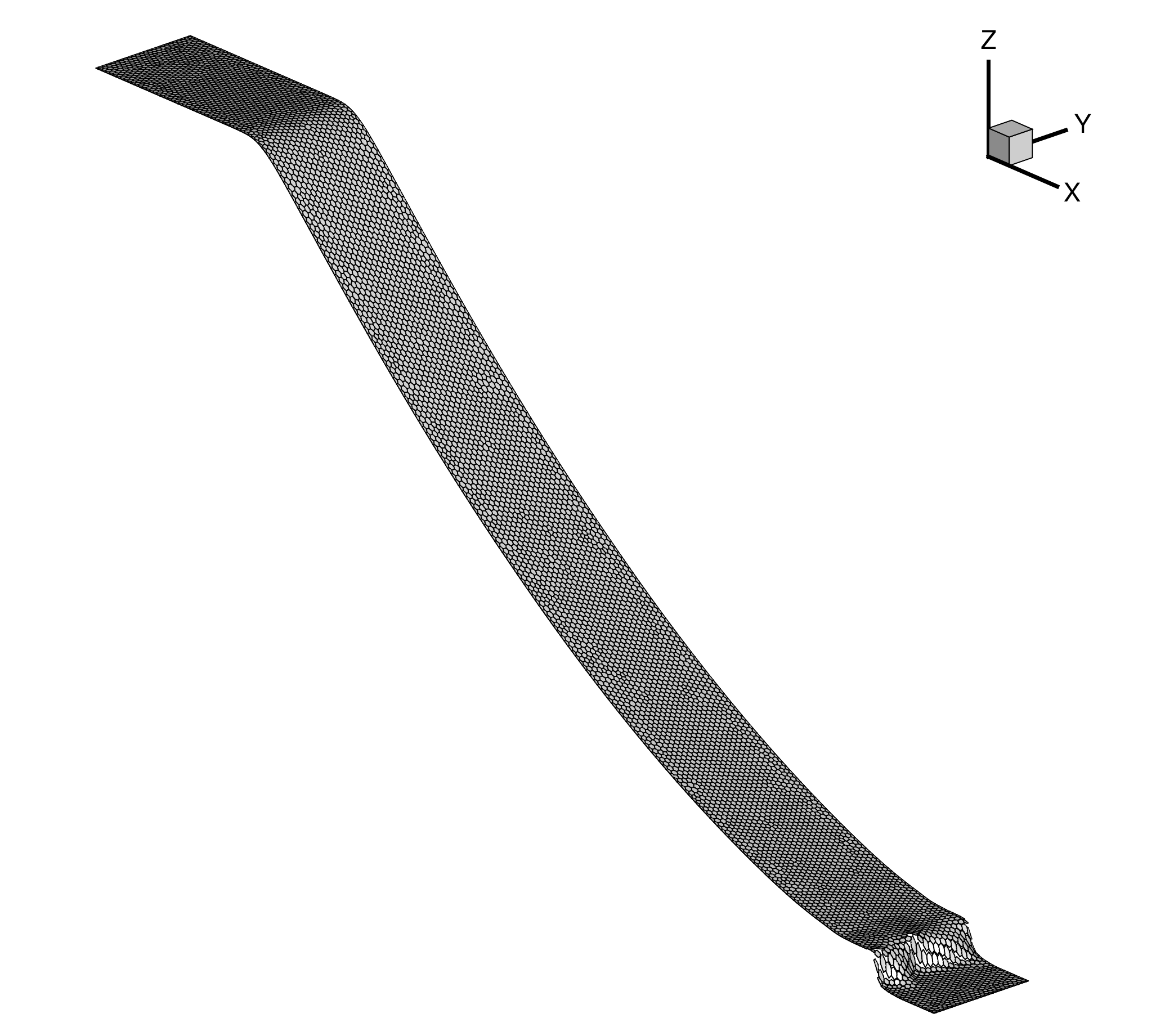} &  
			\includegraphics[width=0.33\textwidth]{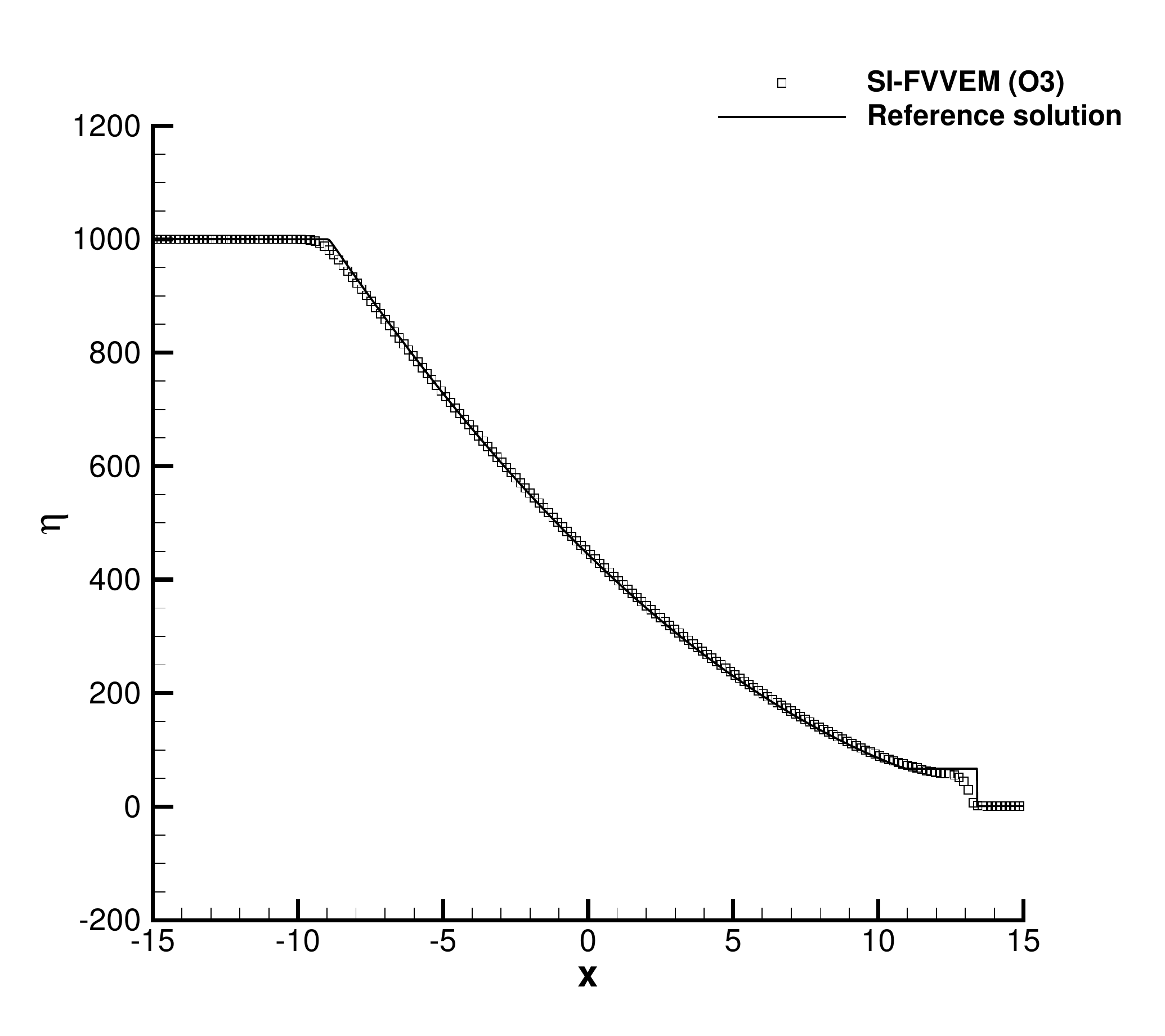} &  			\includegraphics[width=0.33\textwidth]{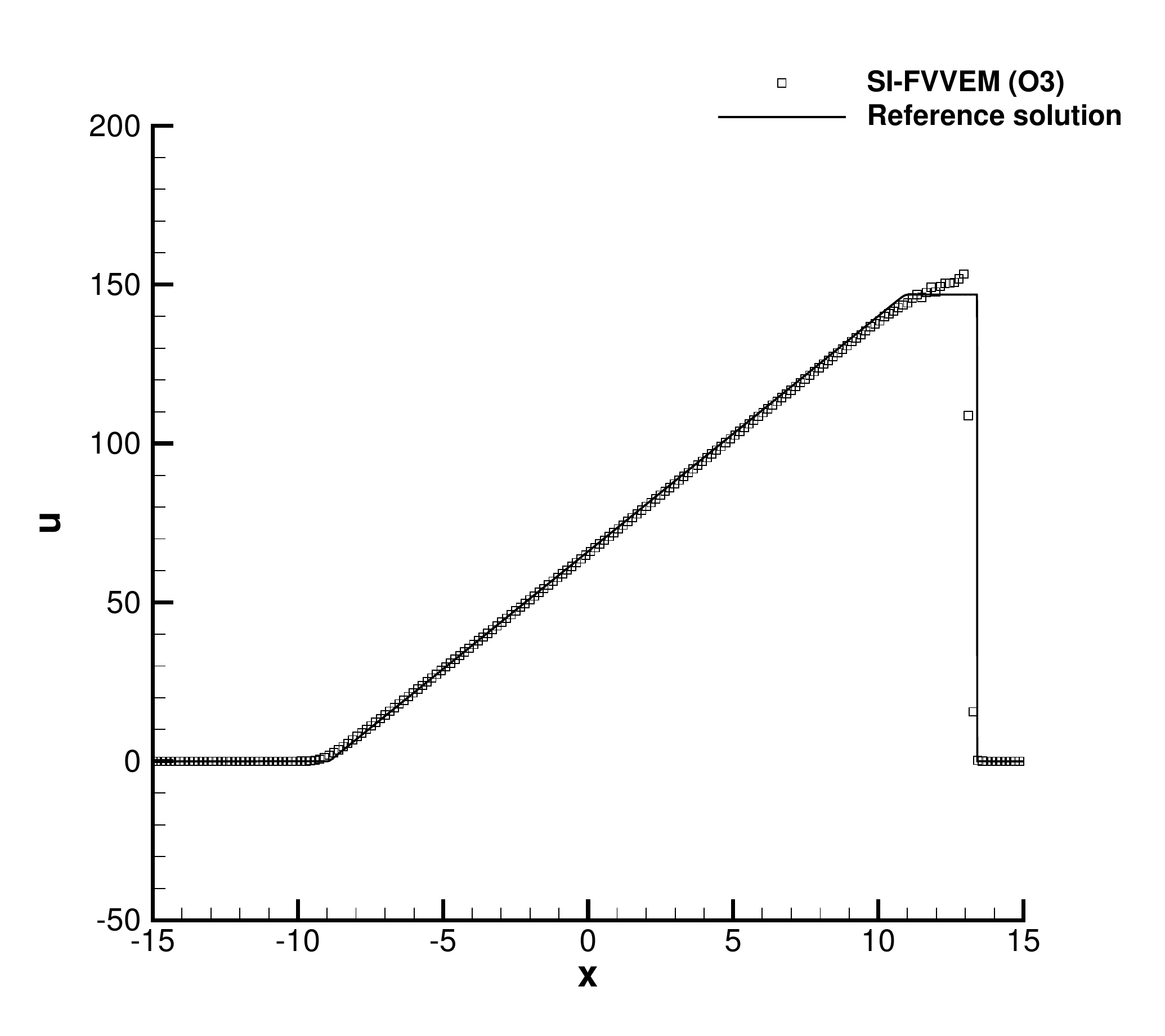} \\
			\includegraphics[width=0.33\textwidth]{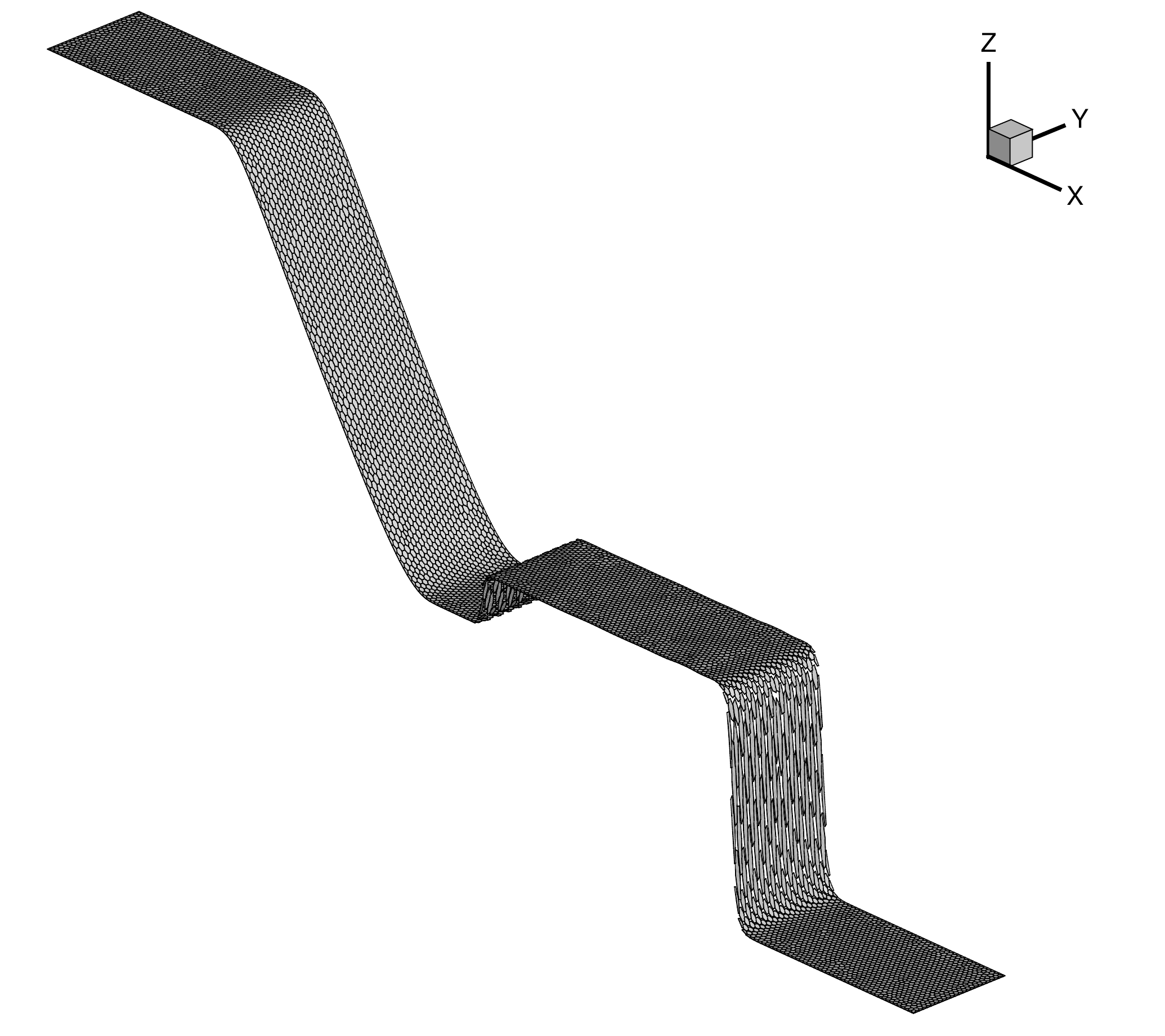} &  
			\includegraphics[width=0.33\textwidth]{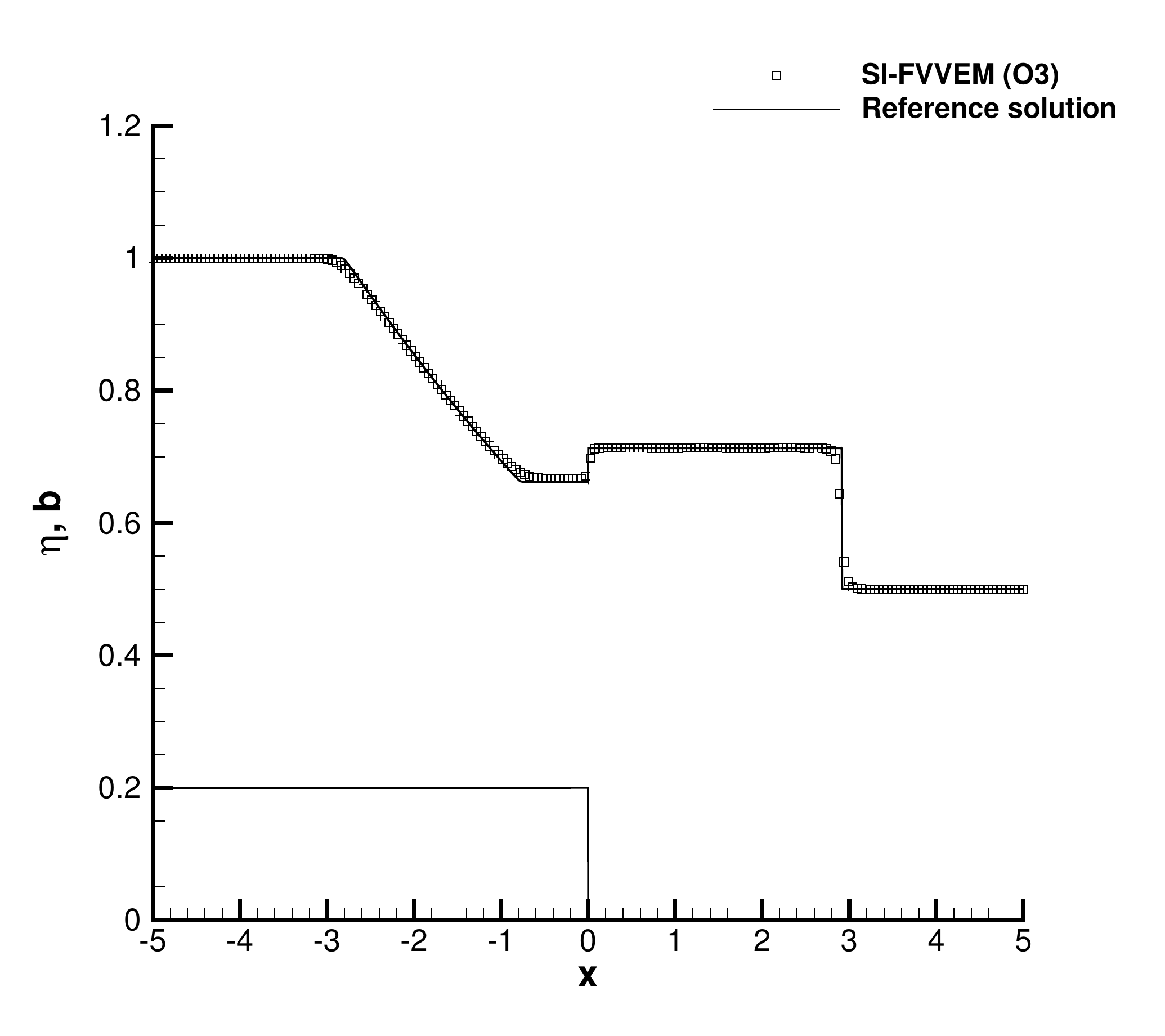} &  			\includegraphics[width=0.33\textwidth]{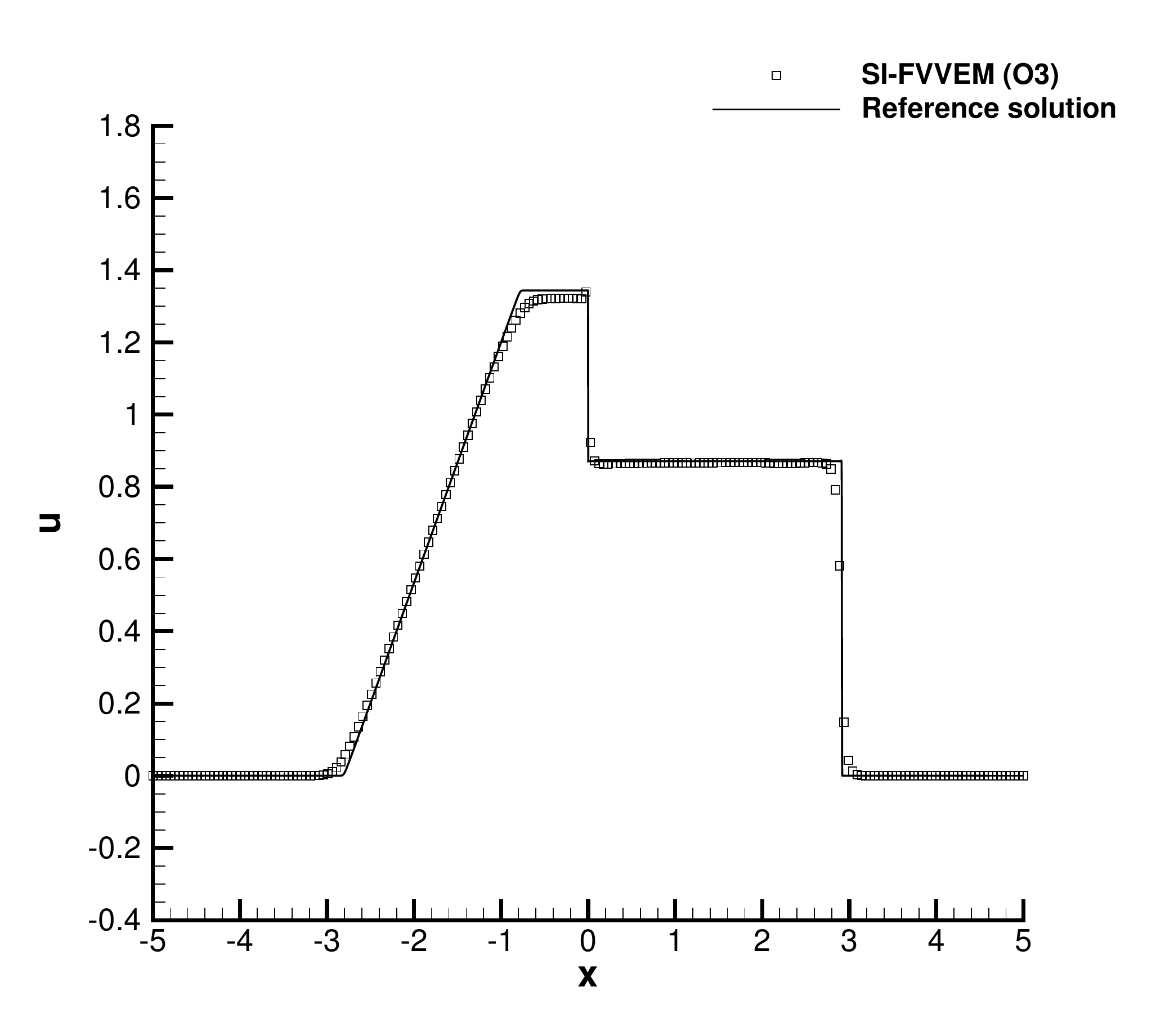} \\
			\includegraphics[width=0.33\textwidth]{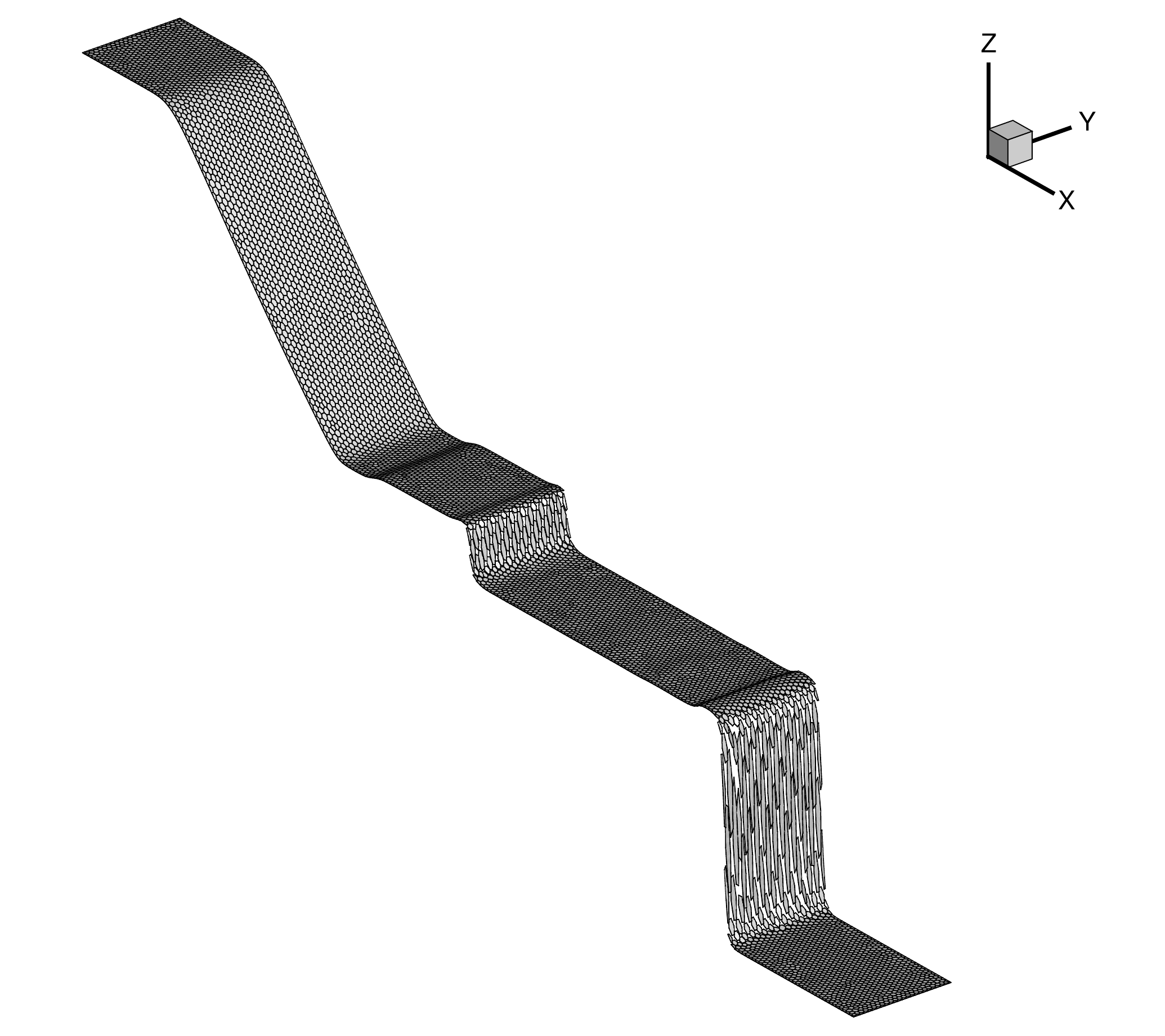} &  
			\includegraphics[width=0.33\textwidth]{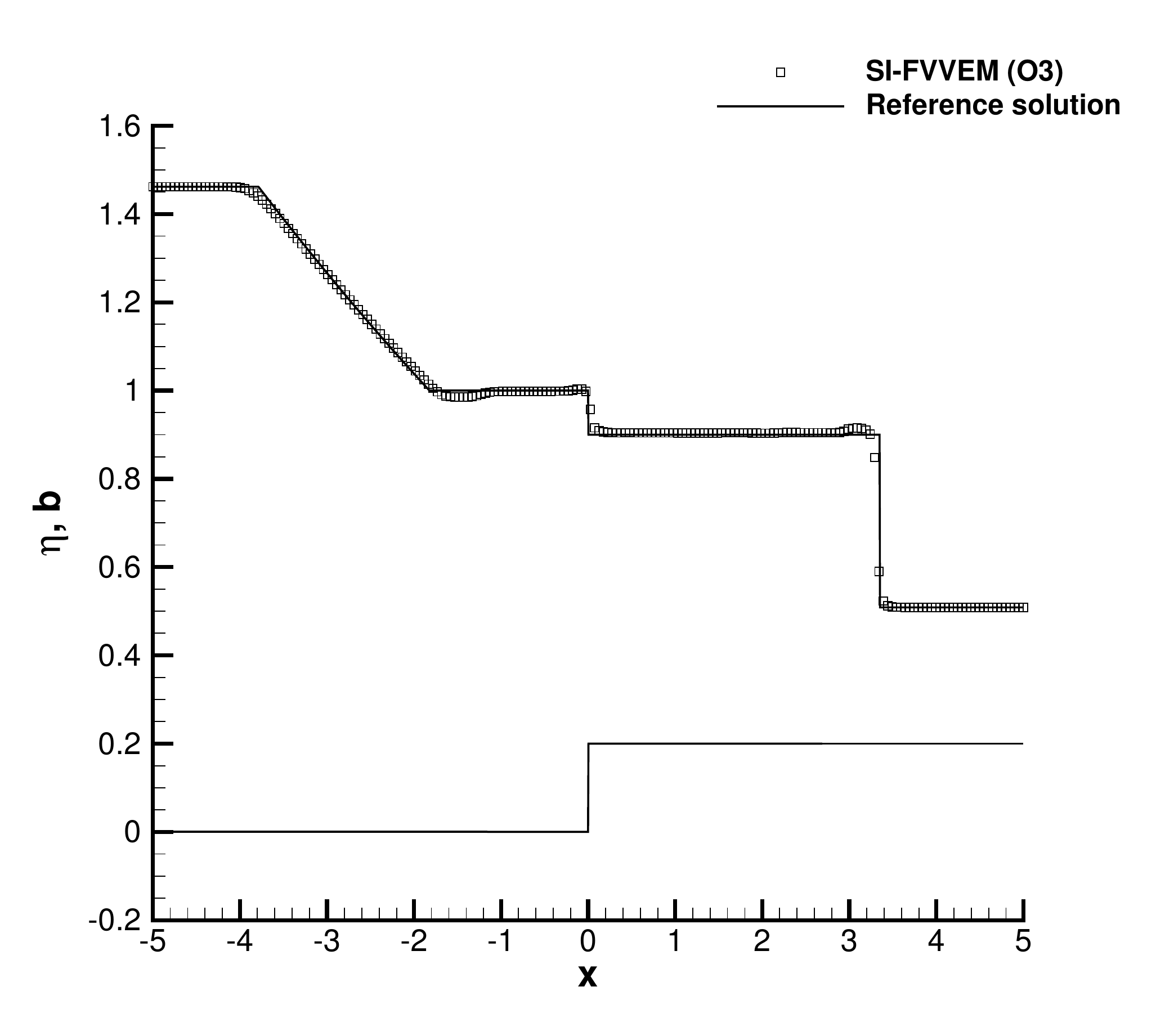} &  			\includegraphics[width=0.33\textwidth]{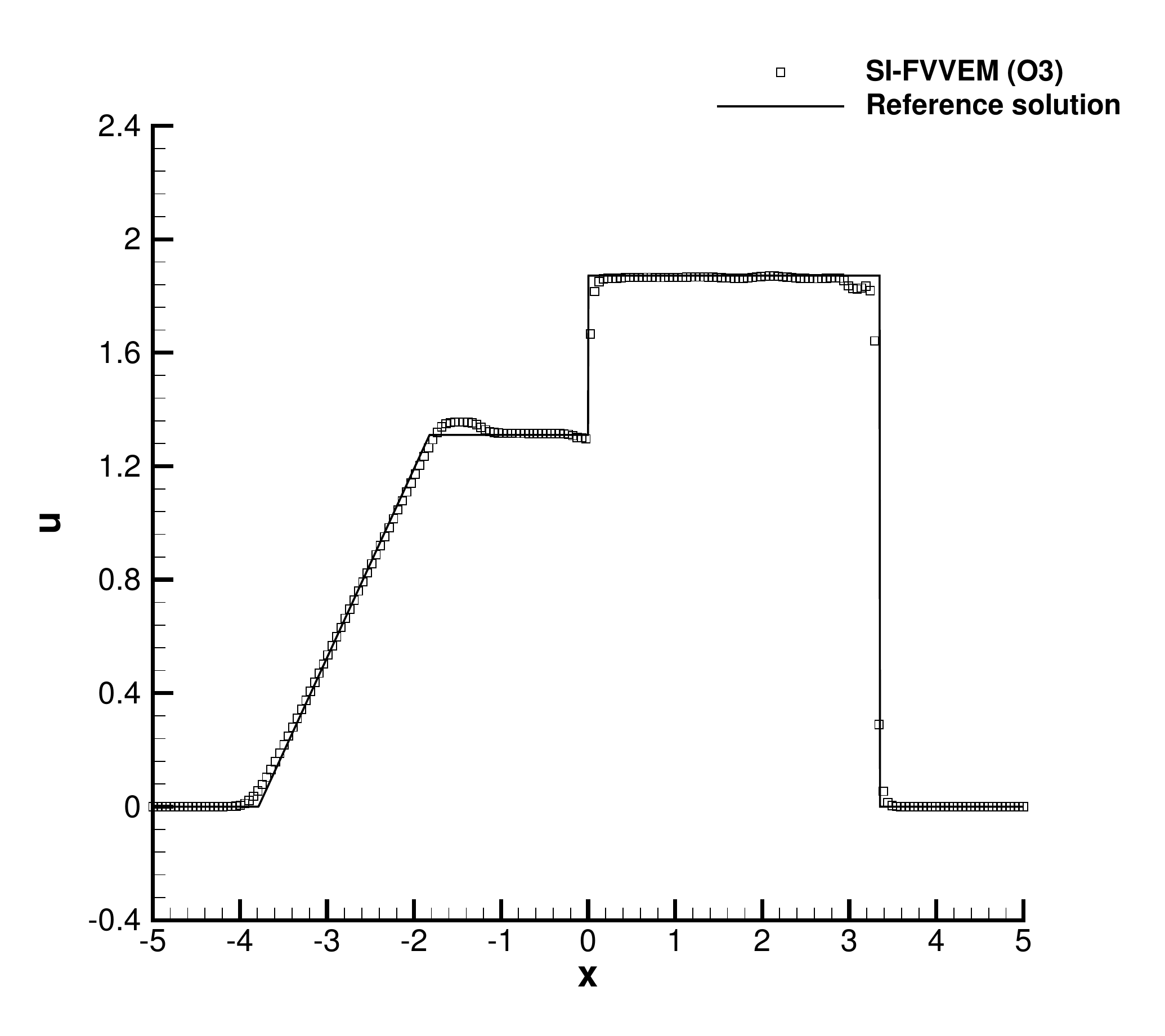} \\
		\end{tabular}
		\caption{Riemann problems RP1, RP2, RP3 and RP4 (from top to bottom row). 3D view of the free surface elevation with the Voronoi computational mesh (left) and comparison against the reference solution for the variables $\eta$ (middle) and $u$ (right).}
		\label{fig.RP}
	\end{center}
\end{figure}
%
\subsection{Smooth surface wave propagation (SWE)}

In this study, the propagation of a wave on the free surface is examined, following the configuration outlined in \cite{StagDG_Dumbser2013}. The computational domain is the square $\Omega=[-1;1]^2$, with Dirichlet boundary conditions enforced on all sides. The domain is discretized into $N_P=15717$ Voronoi cells, with a characteristic mesh size $h=1/50$. The initial condition is defined by 
\begin{equation}
	\label{eqn.smoothwave_ini}
	\eta(\xx,0) = 1 + e^{-\frac{r^2}{2\sigma^2}}, \qquad \vv(\xx,0) = \mathbf{0}, \qquad b(\xx)=0,
\end{equation}
where $\sigma=0.1$. The time step is set to $\dt=0.001$, and the simulation is run until the time $t_f=0.15$, at which the wave profile becomes stiff and a shock wave appears. Figure \ref{fig.smoothwave3D} shows a 3D view of the free surface elevation at different output times, illustrating the ability of the SI-FVVEM scheme to well preserve the symmetry of the solution even when the computational grid is not symmetric.
\begin{figure}[!htbp]
	\begin{center}
		\begin{tabular}{cc} 
			\includegraphics[width=0.40\textwidth]{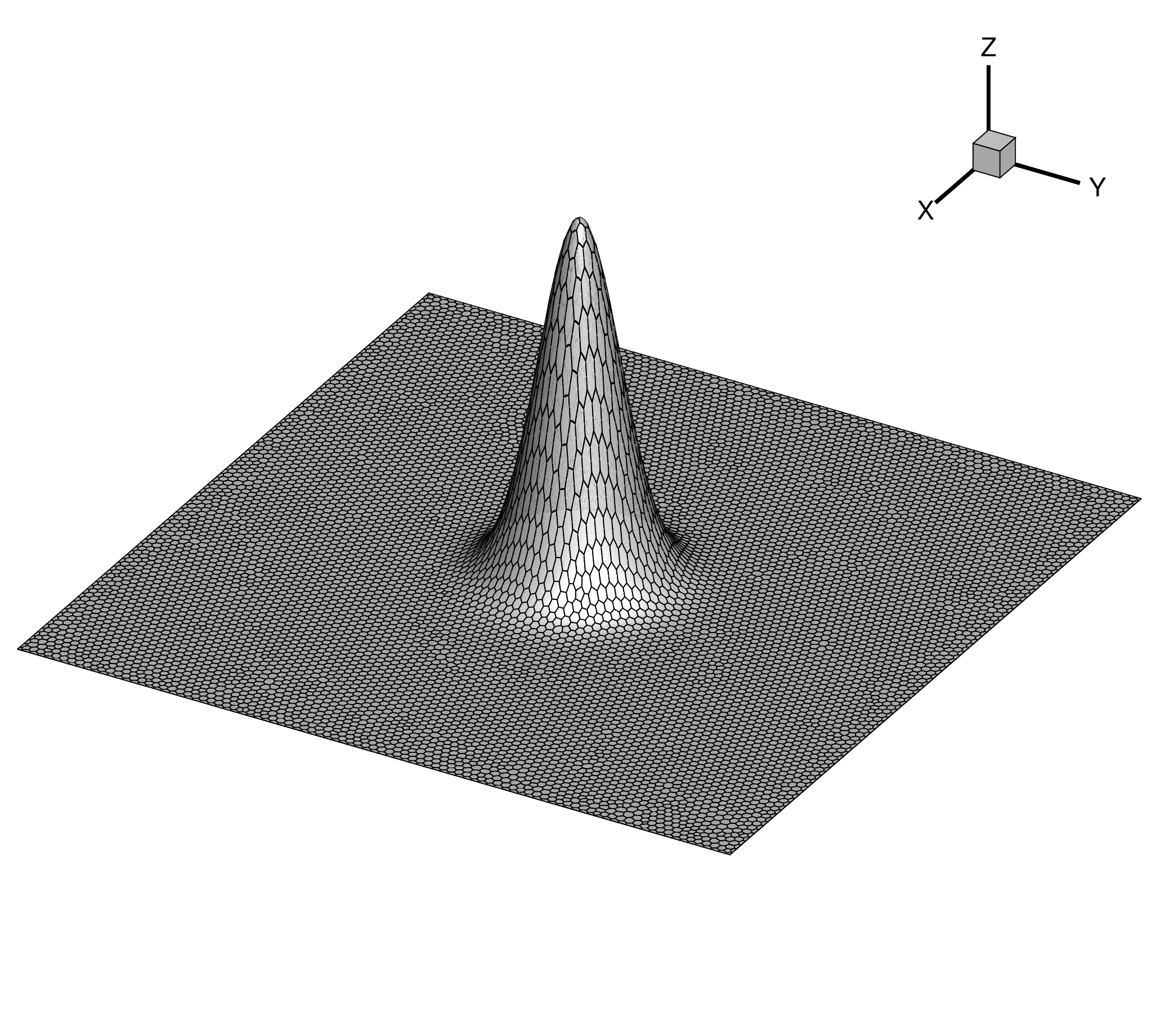} &  			\includegraphics[width=0.40\textwidth]{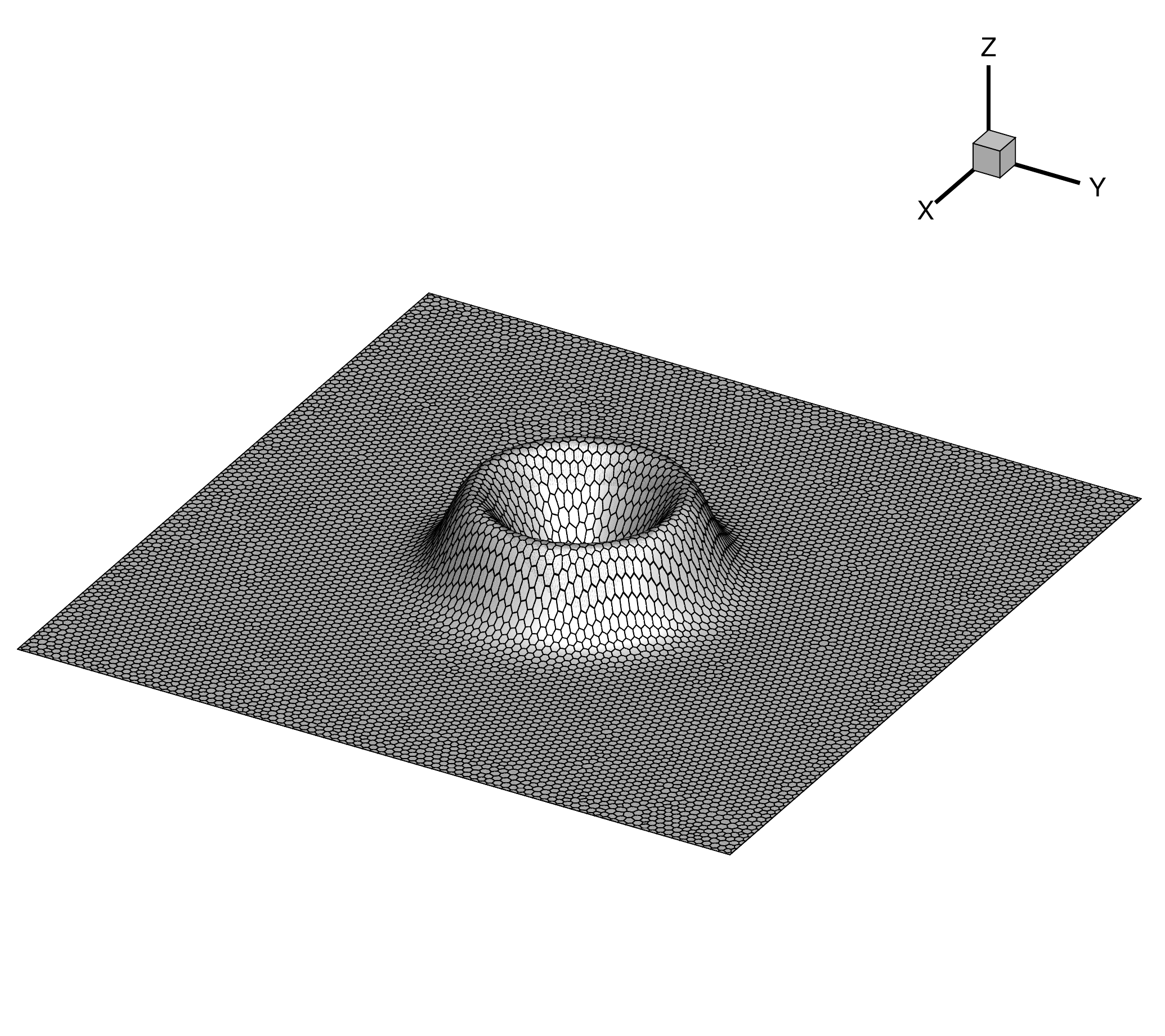} \\
			\includegraphics[width=0.40\textwidth]{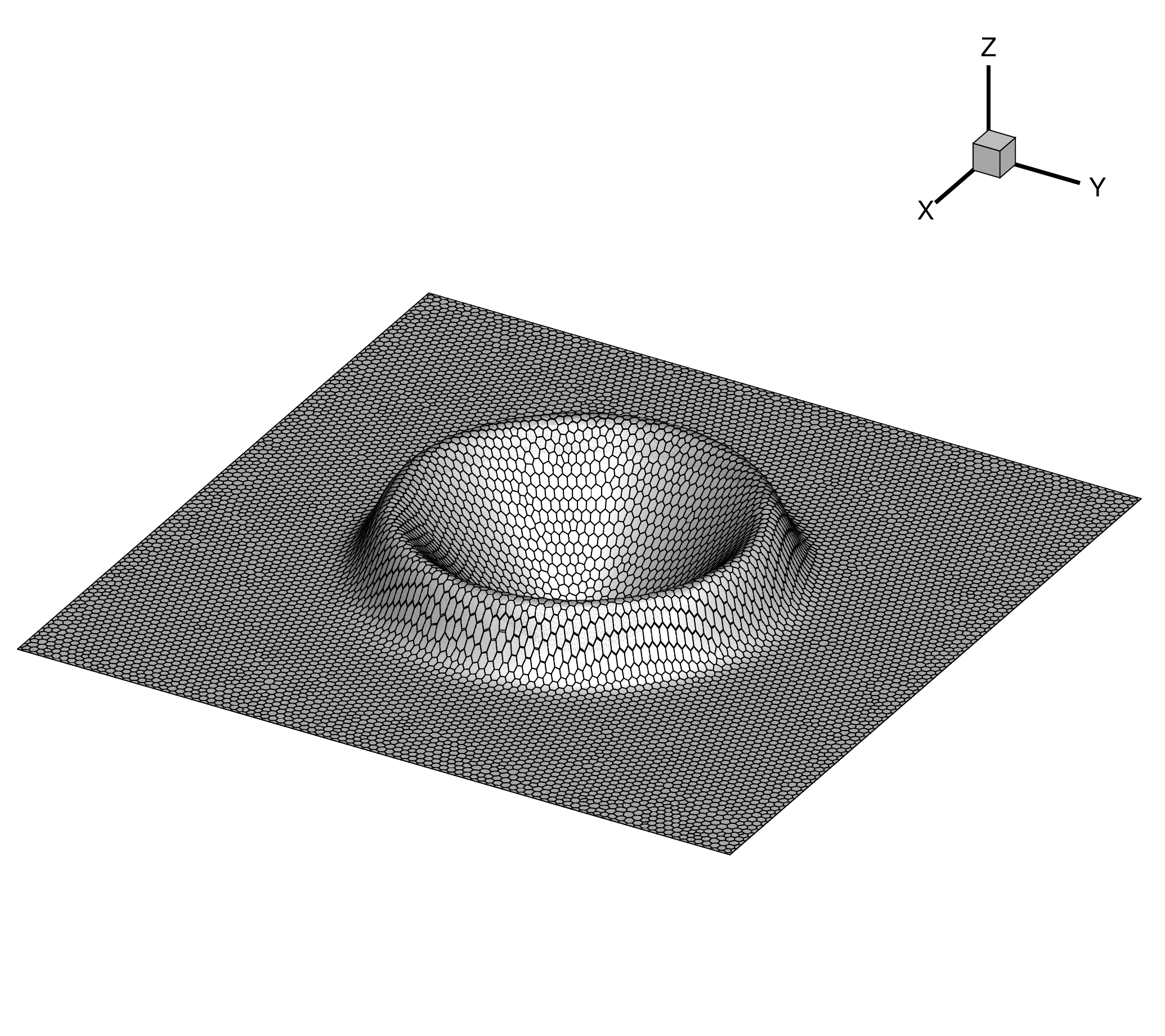} &  			\includegraphics[width=0.40\textwidth]{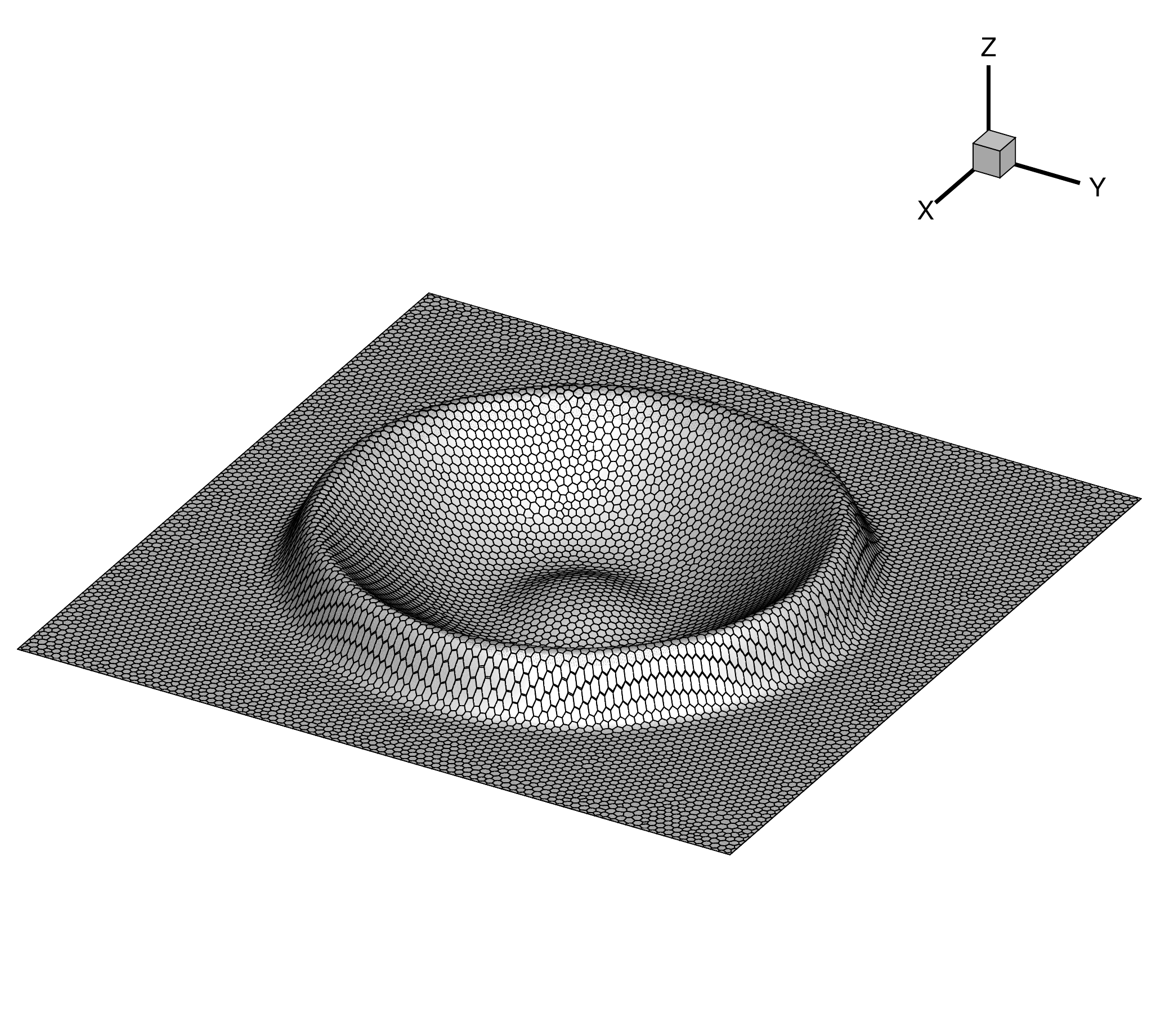} \\
		\end{tabular}
		\caption{Smooth surface wave propagation problem. 3D view of the free surface elevation and Voronoi computational mesh at output times $t=0$, $t=0.05$, $t=0.1$ and $t=0.15$ (from top left to bottom right panel).}
		\label{fig.smoothwave3D}
	\end{center}
\end{figure}
The reference solution is computed using a 1D MUSCL-TVD scheme on a very fine mesh, as done for the circular dam break problem. In Figure \ref{fig.smoothwaveXY} we plot a comparison between the numerical solution and the reference solution for the free surface elevation and horizontal velocity component. A very good agreement can be appreciated, particularly up to time $t=0.1$, when the flow is still smooth. At time $t=0.15$, the shock is smoothed by the explicit finite volume convective solver, and the SI-FVVEM scheme remains stable, without any spurious oscillations.	
\begin{figure}[!htbp]
	\begin{center}
		\begin{tabular}{cc} 
			\includegraphics[width=0.47\textwidth]{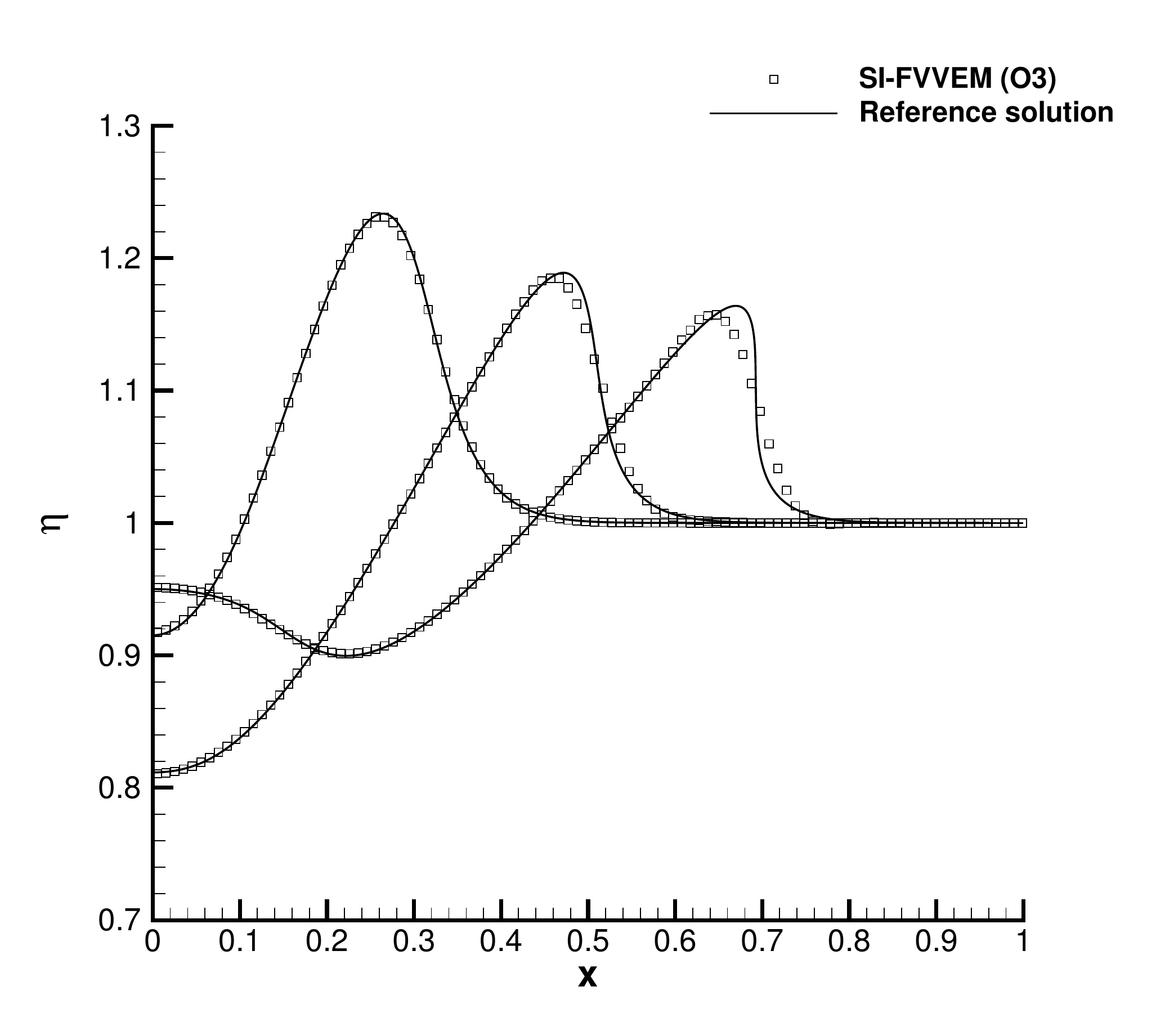} &  			\includegraphics[width=0.47\textwidth]{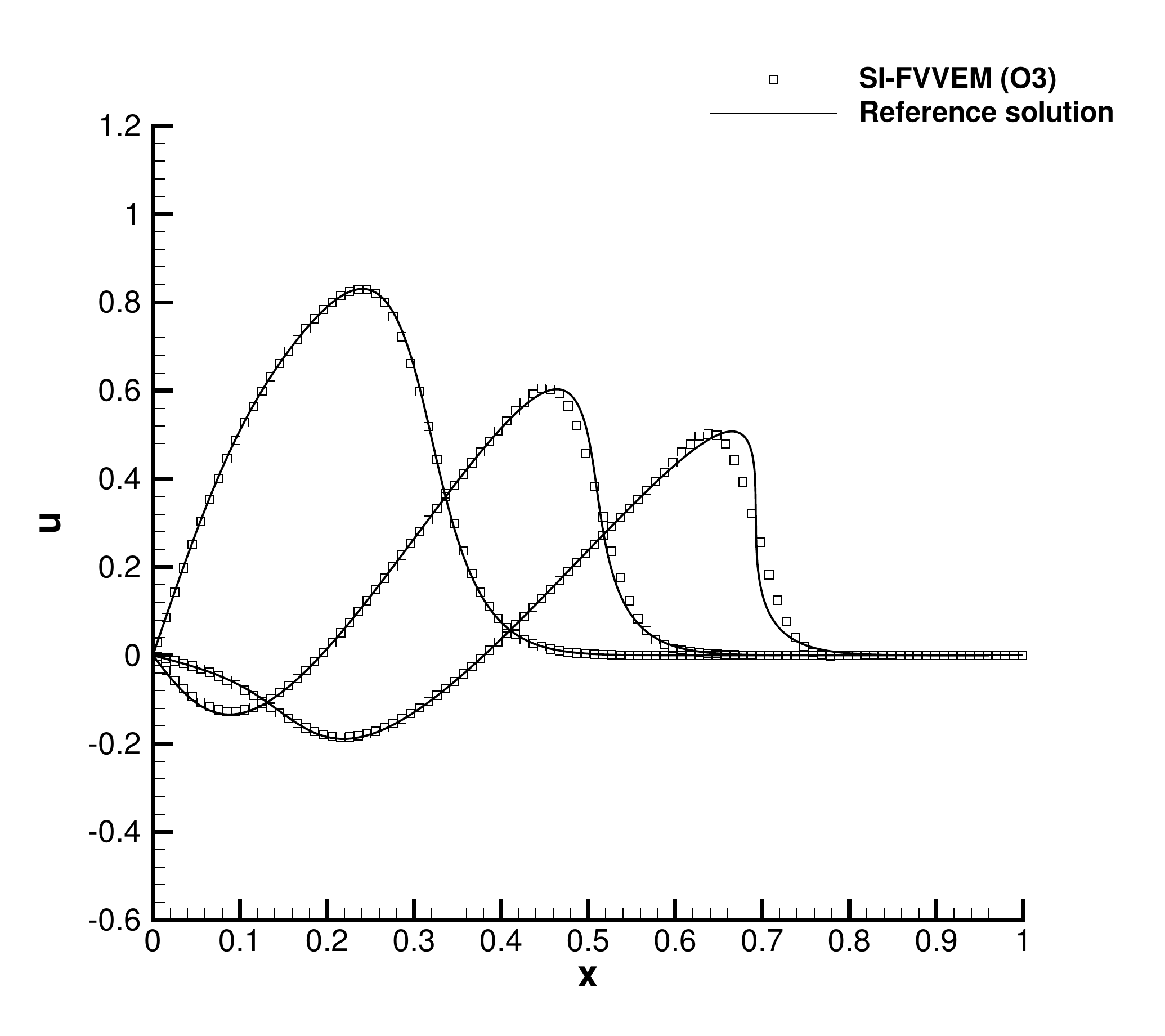} \\
		\end{tabular}
		\caption{Smooth surface wave propagation problem. Comparison between numerical and reference solution at output times $t=0.05$, $t=0.1$ and $t=0.15$ for the free surface elevation $\eta$ (left) and the horizontal velocity component $u$ (right).}
		\label{fig.smoothwaveXY}
	\end{center}
\end{figure}
%
\subsection{Low Froude number flow around a circular cylinder (SWE)}

As last test case for the SWE, we propose to simulate a low Froude ($\Fr=3.19 \cdot 10^{-3}$) flow that passes around a circular cylinder of radius $r_c=1$, as studied in \cite{Busto_SWE2022,BassiRebay97}. The computational domain is $\Omega=[-16;16]^2$, where the circular cylinder is defined as $\xx \in \R^2$ such that $r \leq r_c$, where $r=\sqrt{x^2+y^2}$ is the radial coordinate, and the bottom is assumed to be flat ($b= 0$). The mesh consists of Voronoi cells with a characteristic mesh size $h=1/20$ near the cylinder, increasing regularly in diameter until $h=1/2$ at the domain boundaries (see Figure \ref{fig.Cylinder}). This choice results necessary in order to properly approximate the internal boundary without relying on an isoparametric description of the physical geometry, as proposed in \cite{TavelliSWE2014}. The mesh contains a total of $N_P=15516$ cells.

\begin{figure}[!htbp]
	\begin{center}
		\begin{tabular}{cc} 
			\includegraphics[width=0.47\textwidth]{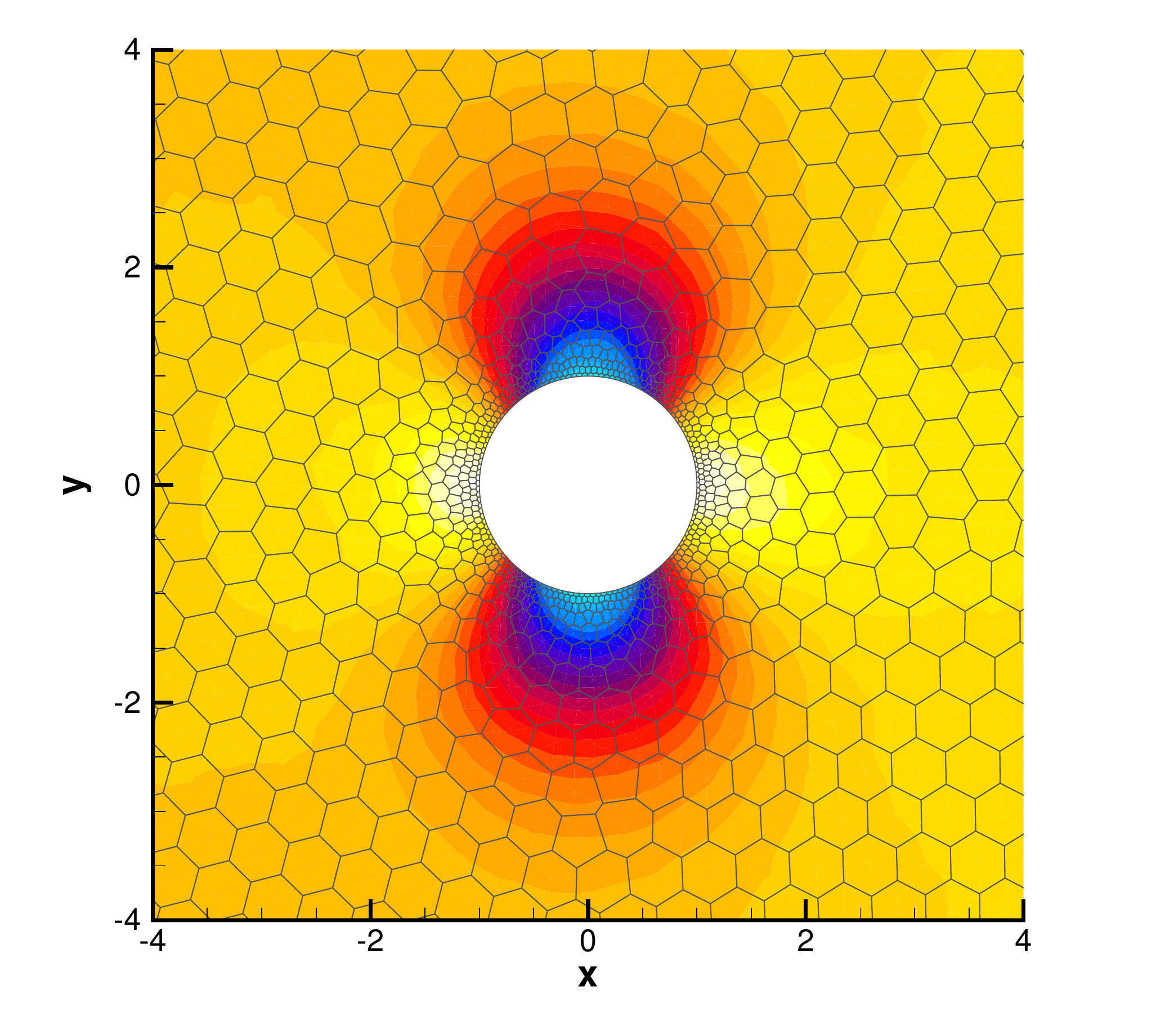} &  			
			\includegraphics[width=0.47\textwidth]{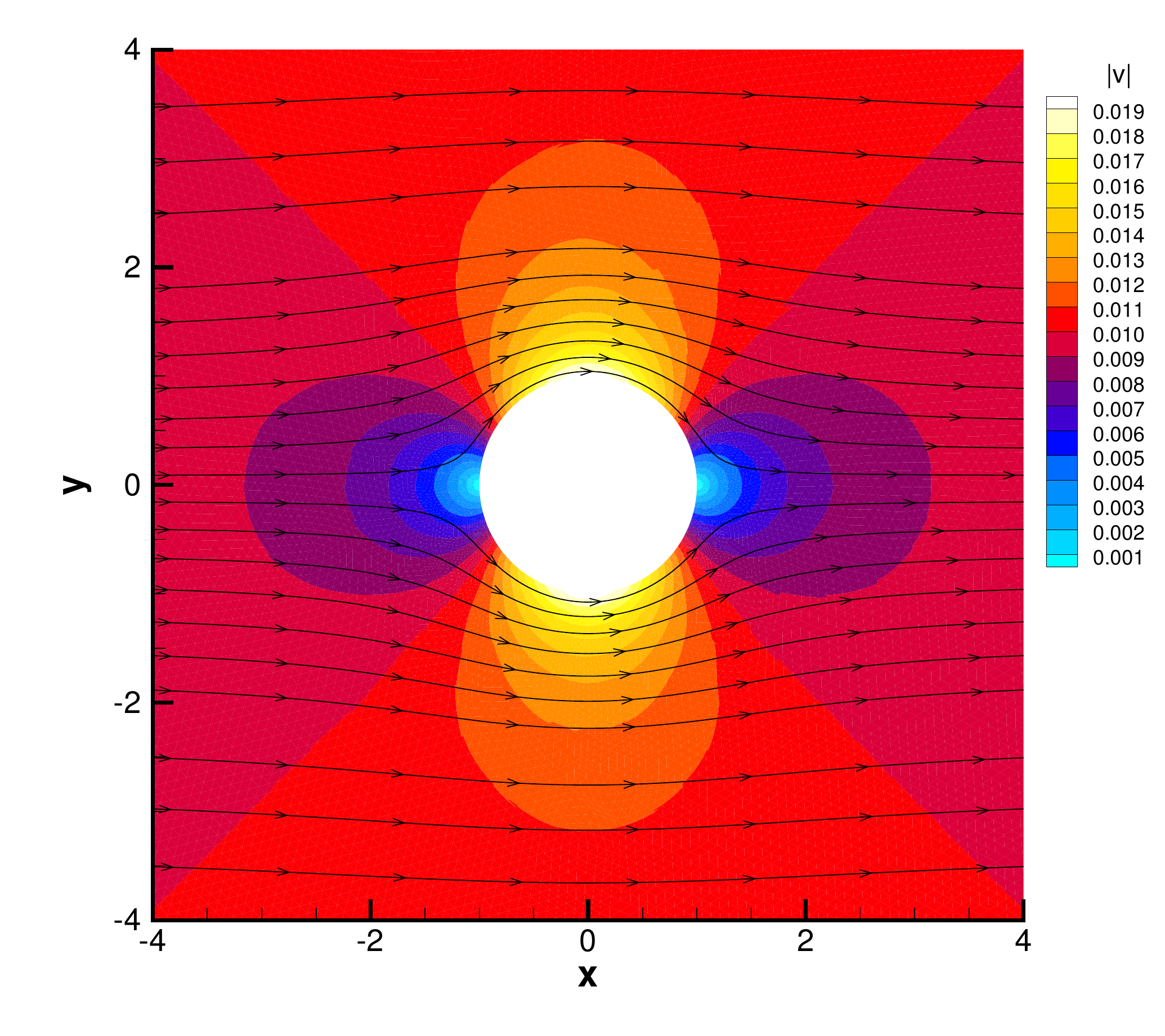} \\
			\includegraphics[width=0.47\textwidth]{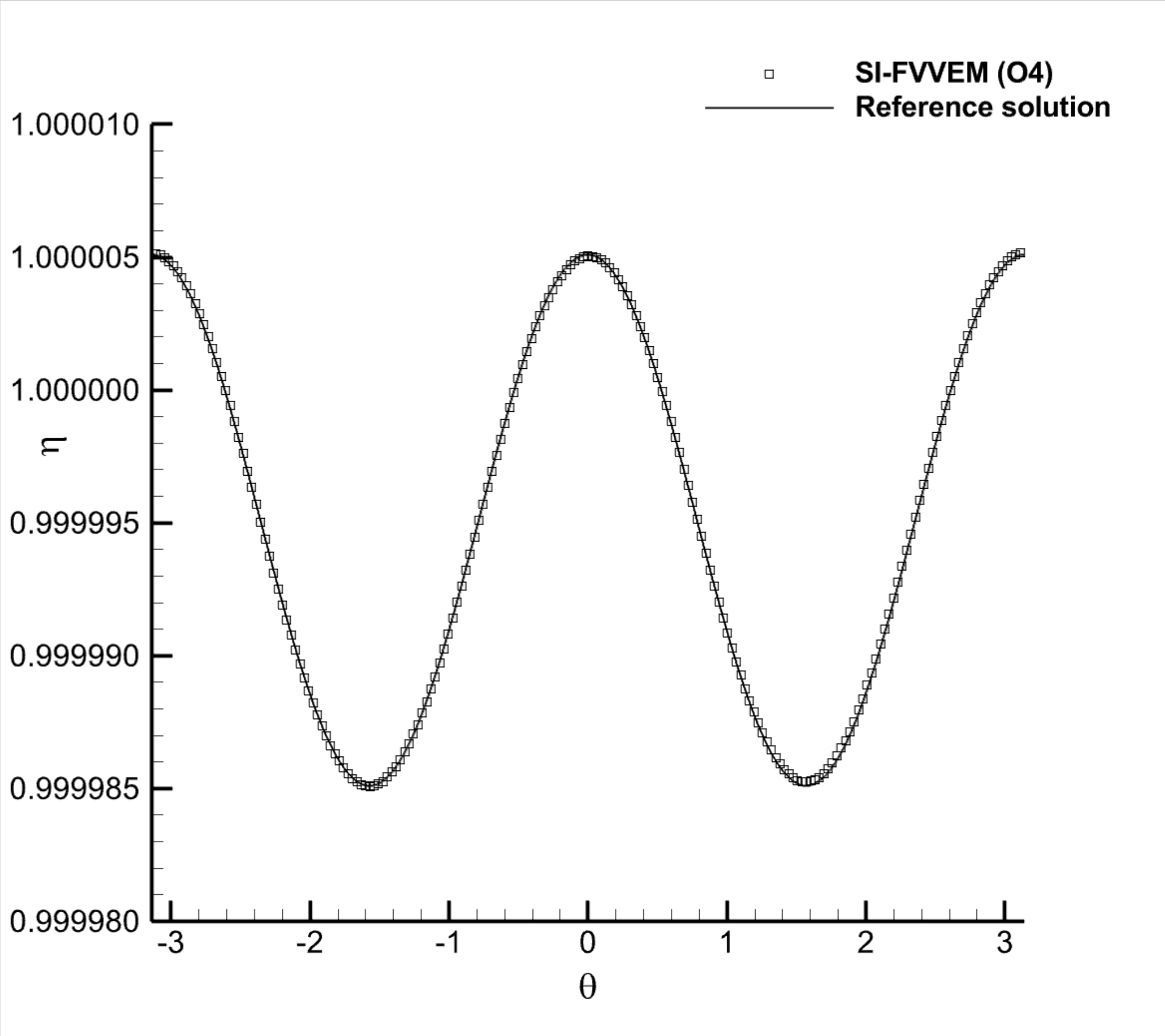} &  			
			\includegraphics[width=0.47\textwidth]{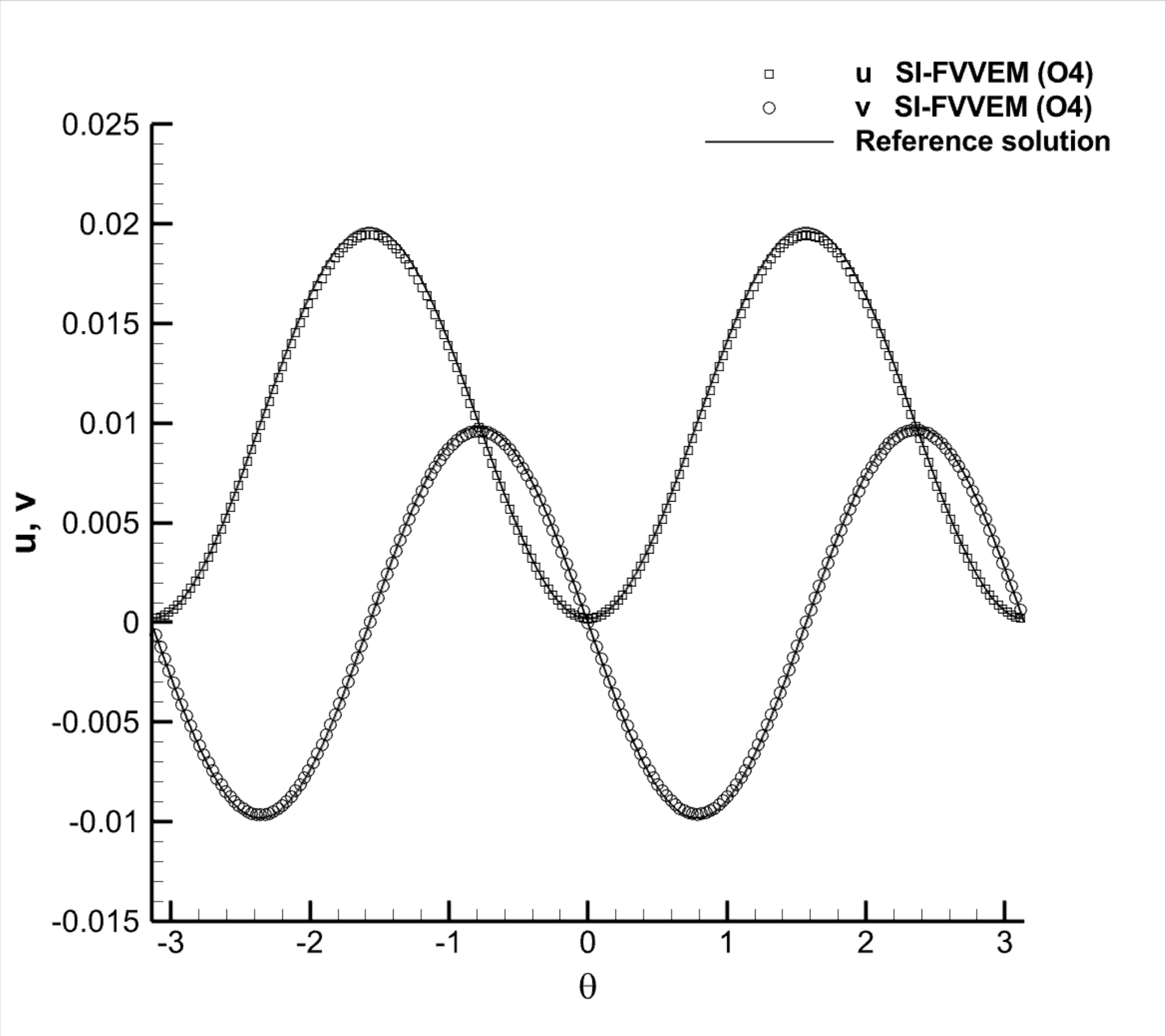} \\
		\end{tabular}
		\caption{Flow around a circular cylinder at time $t_f=10$. Top: zoom around the cylinder of the contour plot of the free surface elevation with the Voronoi computational mesh (left) and of the magnitude of the velocity field with stream-traces (right). Bottom: comparison of the numerical and exact solutions for the free surface (left) and velocity components (right) at radius $r = 1.01$.}
		\label{fig.Cylinder}
	\end{center}
\end{figure}
The exact velocity field and free surface elevation of the test problem can be expressed using polar coordinates $(r,\theta)$ as follows:
\begin{equation}
	\label{eqn.Cyl_ini}
	v_r = v_m \left( 1-\frac{r_c^2}{r^2} \right) \cos(\theta), \qquad v_{\theta} = -v_m \left( 1+\frac{r_c^2}{r^2} \right) \sin(\theta),
	\qquad \eta = \eta_0 + \frac{1}{2} v_m^2 g \left( \frac{2 r_c^2}{r^2} \cos(2\theta) -\frac{r_c^4}{r^2} \right),
\end{equation}
where $\eta_0=1$ and $v_m=10^{-2}$. The initial condition for the simulation is a flat free surface, $\eta(\mathbf{x},0)=\eta_0$, while the exact solution is imposed at all boundaries except for the rightmost side of the domain ($x=16$), where an outflow condition is set. The simulation is run until the final time $t_f=10$ to reach the stationary state. In order to exploit the high order of accuracy of the SI-FVVEM methods, we use the fourth order version in space, with a first order time discretization. The velocity field with associated streamlines at the final time are shown in Figure \ref{fig.Cylinder} (top). The comparison against the reference solution is presented in Figure \ref{fig.Cylinder} (bottom) along the circumference of radius $r = 1.01$ centred at the origin, where we observe that the fourth order in space SI-FVVEM method properly retrieves the exact profile of both the free surface elevation and velocity components.

%
\subsection{Stationary Poiseuille flow (INS)}
This first test case for the INS model concerns an incompressible viscous flow between two parallel plates. It is a particular solution of the Stokes system, i.e., the convective fluxes are neglected. The computational domain is the channel $\Omega = [x_L; x_R] \times [y_B; y_T] = [0;3] \times [0;1]$ discretized with $N_P=1949$ cells, and the stationary solution writes
\begin{equation} \label{eqn.poiseuille}
	\frac{\partial p}{\partial x} = \frac{p_R - p_L}{x_R-x_L}, \qquad 
	\vv = \left( \begin{array}{c}
		u \\ v
	\end{array} \right)
		 = \left( \begin{array}{c}
		\frac{(p_R - p_L)}{2 L \nu} y (y - y_T) \\
		0
	\end{array} \right), \qquad 
	\nu=10^{-2},
\end{equation}
where $p_L=-1$ and $p_R=-5.8$ are the constant pressures applied along $x_L$ and $x_R$, respectively. Wall boundary conditions are imposed in $y-$direction, while the Dirichlet boundaries are set along the $x-$direction. The initial condition is given by the exact solution in \eqref{eqn.poiseuille}. Table \ref{tab:pois} collects the errors in $L_2$ and $L_\infty$-norms between the recovered and exact horizontal velocity $u$ with respect to the polynomial space dimension $M$ at final time $t_f = 5$. As expected, when $M \geq 2$, i.e. the polynomial space contains parabolic functions, the numerical errors are up to machine accuracy. 

\begin{table}[!htbp]  
	\caption{Errors in $L_2$ and $L_\infty$ norms for the horizontal velocity $u$ for the stationary Poiseuille flow. The errors refer to the polynomial space dimension $M$ at final time $t_f = 5$.}  
	\begin{center} 
		\begin{small}
			\renewcommand{\arraystretch}{1.0}
			\begin{tabular}{lll}
				\hline
				$M$ & $L_2(u)$   & $L_\infty(u)$ \\ \hline
				1   & 3.0966E-02 & 8.5265E-2     \\
				2   & 4.1826E-13 & 1.6733E-13    \\
				3   & 4.2756E-13 & 3.0695E-13   \\ \hline
			\end{tabular}
		\end{small}
	\end{center}
	\label{tab:pois}
\end{table}

%
\subsection{2D Taylor-Green vortex (INS)}
The Taylor-Green vortex describes the flow of an incompressible viscous fluid in a 2D square domain $\Omega = [0;2\pi]^2$ with periodic boundary conditions. The exact solution for this problem is
\begin{equation} \label{eqn.TGV}
	p = -\frac{e^{-4 \nu t}}{4} \left( \cos( 2x ) + \cos( 2y ) \right) , \qquad 
	\vv = \left( \begin{array}{c}
		u \\ v
	\end{array} \right) = \left( \begin{array}{c}
		\phantom{-}\sin( x ) \cos(y)  \\
		- \cos( x ) \sin( y) 
	\end{array} \right) e^{-2 \nu t}, \qquad 
	\nu=10^{-2}.
\end{equation} 
The initial condition is defined by the exact solution \eqref{eqn.TGV} at time $t = 0$. A second order SI-FVVEM method is run on a mesh composed of $N_P=5705$ polygonal cells with characteristic size $h = 1/20$. Figure \ref{fig.TGV_mu1e-2} depicts the numerical solution obtained at the final time of the simulation, $t_f = 0.2$. The velocity magnitude and the pressure distribution with the stream-traces are shown and a comparison of 1D cuts along the $x-$ and the $y-$axis of the numerical solution against the exact one is proposed for both pressure and velocity, retrieving an excellent matching. 
%

\begin{figure}[!htbp]
	\begin{center}
		\begin{tabular}{cc} 
			\includegraphics[width=0.47\textwidth]{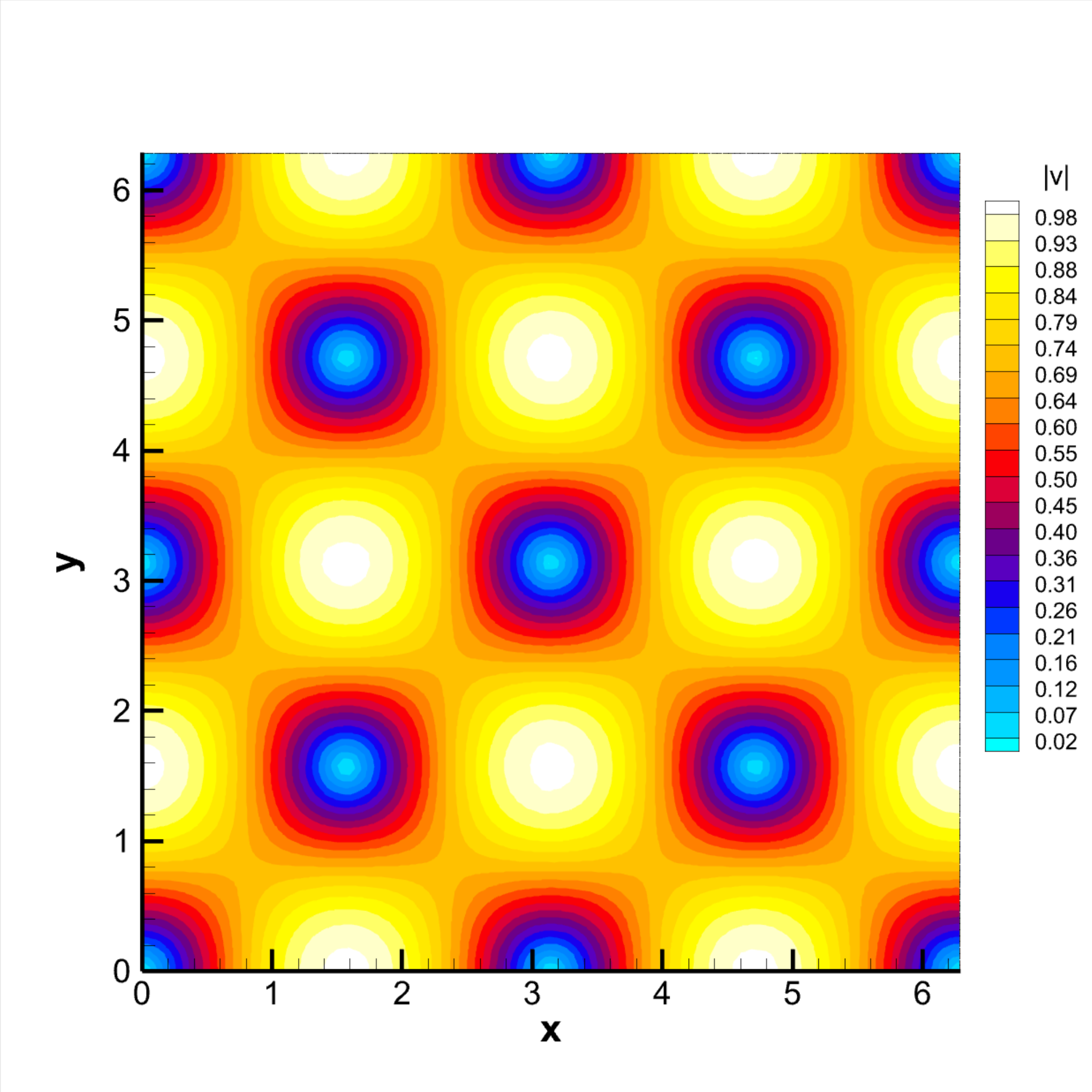} &  			
			\includegraphics[width=0.47\textwidth]{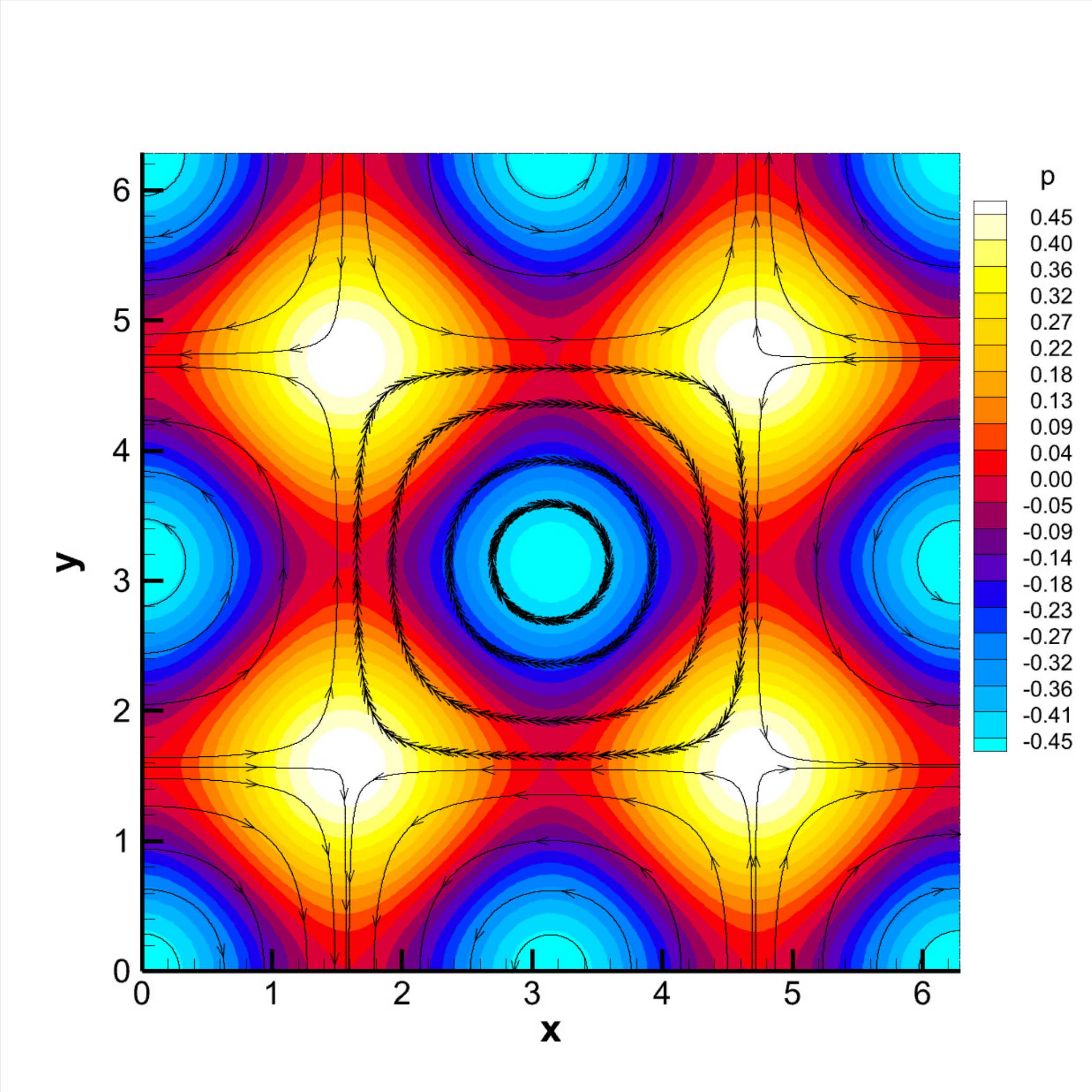} \\
			\includegraphics[width=0.47\textwidth]{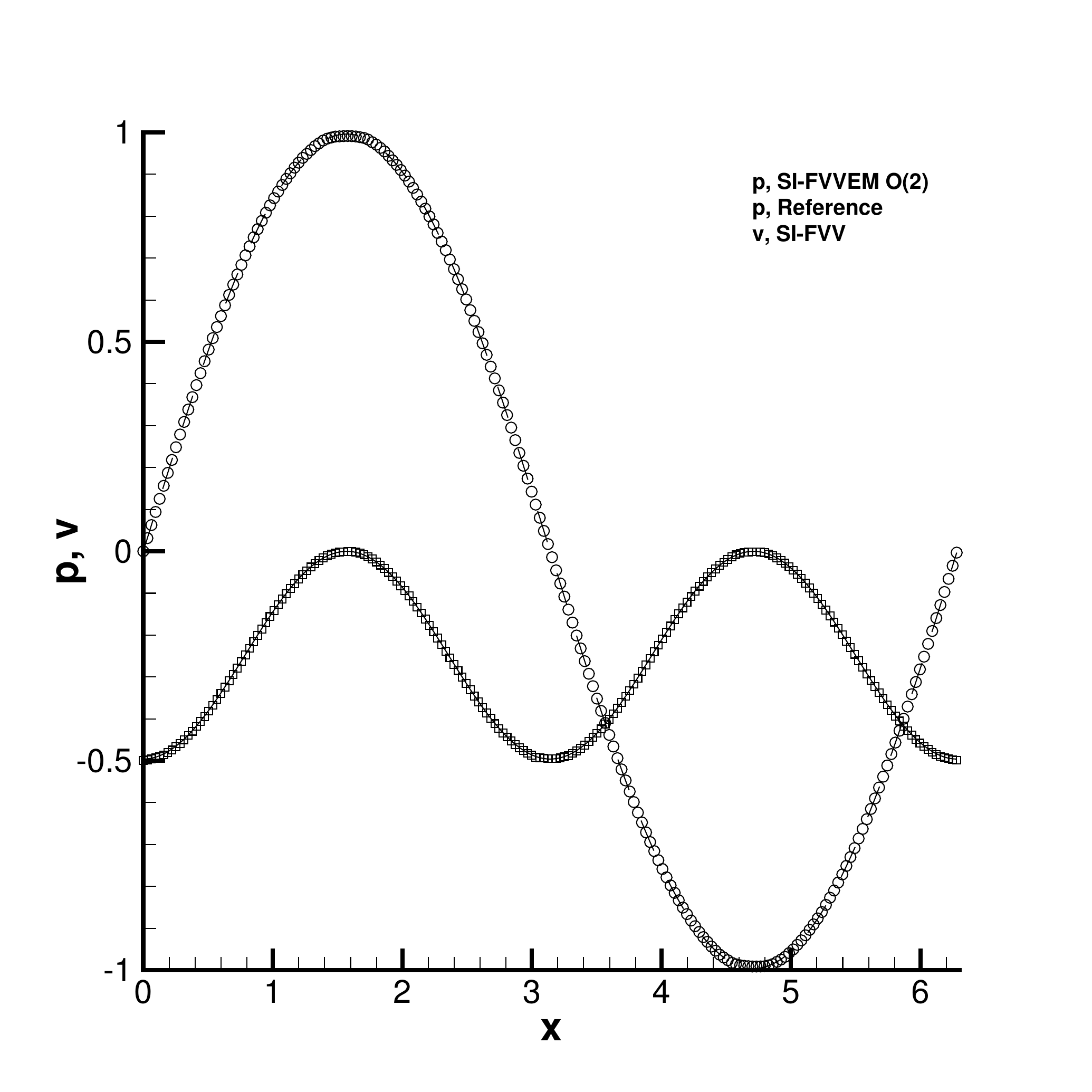} &  			
			\includegraphics[width=0.47\textwidth]{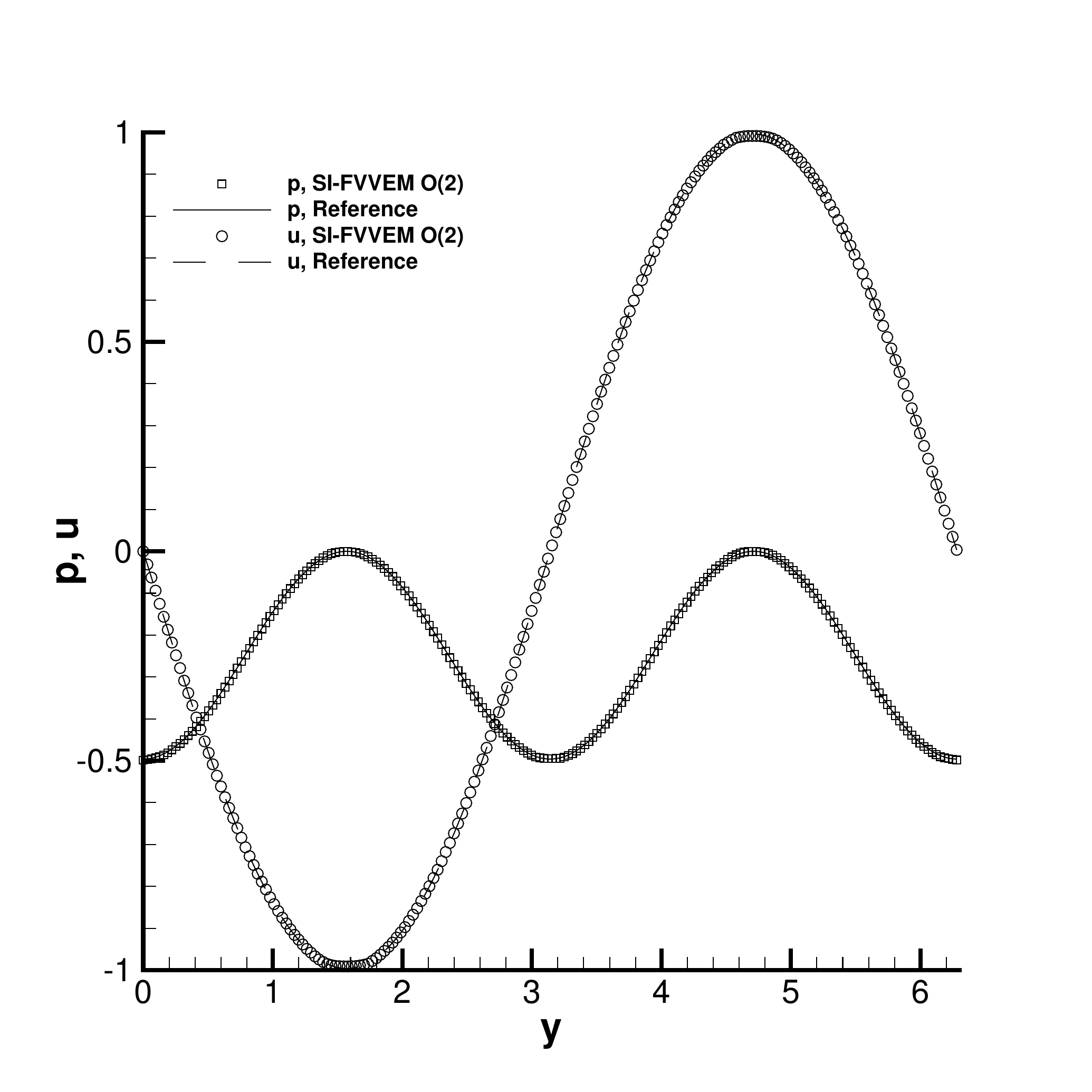} \\
		\end{tabular}
		\caption{Taylor-Green vortex at time $t = 0.2$ with viscosity $\nu = 10^{-2}$. Top: velocity magnitude (left) and pressure with the stream-traces (right). Bottom: 1D cuts with 200 equidistant points along the $x-$axis (left) and the $y-$axis (right) of the numerical solution against the exact solution.}
		\label{fig.TGV_mu1e-2}
	\end{center}
\end{figure}

This test case is also used to asses the convergence property of the proposed SI-FVVEM scheme. We consider a sequence of successively refined meshes, and we compute the errors in $L_2$ norm with respect to the analytical solution \eqref{eqn.TGV}. The convergence rates are reported in Table \ref{tab.conv_rate_tgv} for different values of the Reynolds number, in order to evaluate the asymptotic behaviour of the scheme. As demonstrated by Theorem \ref{th_ins_2}, the schemes satisfy the AP property and the errors are independent of the stiffness parameter $\Rey$, achieving the formal second order of accuracy in space and time.

\begin{table}[!htp]
	\begin{center}
		\caption{Numerical convergence results of the SI-FVVEM scheme with second order of accuracy in space and time using the Taylor-Green vortex on general polygonal meshes. The errors are measured in $L_2$ norm and refer to the pressure $p$ and velocity component $u$ at time $t_f=0.2$. The asymptotic preserving (AP) property of the scheme is studied by considering different Reynolds numbers $\Rey=\{ 10^{2}, 10^{3}, 10^{4}, 10^{5} \}$.}
		\begin{small}
			\renewcommand{\arraystretch}{1.1}	
			\begin{tabular}{ccccc}
				$h(\Omega)$ & ${L_2}(p)$ & $\mathcal{O}(p)$ & ${L_2}(u)$ & $\mathcal{O}(u)$ \\
				\hline
				& \multicolumn{4}{c}{$\Rey=10^2$} \\
				\hline
				3.1483E-01 & 1.8321E-02 & -    & 4.6596E-02 & -    \\
				1.7729E-01 & 6.2191E-03 & 1.88 & 9.2174E-03 & 2.82 \\
				1.2069E-01 & 1.7224E-03 & 3.34 & 3.7950E-03 & 2.31 \\
				8.9945E-02 & 9.2102E-04 & 2.13 & 2.1919E-03 & 1.87 \\
				\hline
				& \multicolumn{4}{c}{$\Rey=10^3$} \\
				\hline
				3.1483E-01 & 1.8501E-02 & -    & 4.6894E-02 & -    \\
				1.7729E-01 & 6.0593E-03 & 1.94 & 9.3017E-03 & 2.82 \\
				1.2069E-01 & 1.7371E-03 & 3.25 & 3.8958E-03 & 2.26 \\
				8.9945E-02 & 9.2967E-04 & 2.13 & 2.2858E-03 & 1.81 \\ 
				\hline
				& \multicolumn{4}{c}{$\Rey=10^4$} \\
				\hline
				3.1483E-01 & 1.8520E-02 & -    & 4.6924E-02 & -    \\
				1.7729E-01 & 6.0464E-03 & 1.95 & 9.3105E-03 & 2.82 \\
				1.2069E-01 & 1.7387E-03 & 3.24 & 3.9069E-03 & 2.26 \\
				8.9945E-02 & 9.3060E-04 & 2.13 & 2.2964E-03 & 1.81 \\ 
				\hline
				& \multicolumn{4}{c}{$\Rey=10^5$} \\
				\hline
				3.1483E-01 & 1.8522E-02 & -    & 4.6927E-02 & -    \\
				1.7729E-01 & 6.0451E-03 & 1.95 & 9.3114E-03 & 2.82 \\
				1.2069E-01 & 1.7389E-03 & 3.24 & 3.9080E-03 & 2.26 \\
				8.9945E-02 & 9.3069E-04 & 2.13 & 2.2975E-03 & 1.81 \\
				\hline
			\end{tabular}
		\end{small}
		\label{tab.conv_rate_tgv}
	\end{center}
\end{table}

%
\subsection{First problem of Stokes (INS)}
In this test case, we consider a problem only dominated by viscous effects, known as the first problem of Stokes, for which an exact solution of the unsteady Navier-Stokes equations has been derived \cite{Schlichting}. An infinite incompressible shear layer is modelled starting from the following velocity field:
\begin{equation}\label{eqn.stokes1ex}
	u = 0, \qquad v = v_0 \, \erf \left( \frac{1}{2} \frac{x}{\sqrt{\nu t}} \right), \qquad v_0 = 0.1.
\end{equation}
The initial condition then writes
\begin{equation} \label{eqn.stokes1ic}
	p = 1, \qquad u = 0, \qquad v = \begin{cases}
		-v_0, & x \leq 0 \\
		\phantom{-}v_0, & x > 0
	\end{cases}.
\end{equation}
The computational domain is $\Omega=[-0.5;0.5]\times[-0.1;0.1]$ and it is paved with a grid of characteristic mesh size $h=1/100$ along the $x-$direction, while $h=1/10$ is used along the $y-$direction. At the boundaries, the initial condition \eqref{eqn.stokes1ic} is imposed along the $y$-direction while periodic conditions close the problem on the remaining sides. The first problem of Stokes is solved using three values of viscosity, namely $\nu = 10^{-3}$, $\nu = 10^{-4}$ and $\nu = 10^{-5}$. Figure \ref{fig.stokes1} depicts the recovered solutions for the vertical velocity $v$ at $y = 0$ over the exact solution at the final time $t_f = 1$. The figures show a perfect overlapping between numerical and analytical solutions. We remark that the time step is independent of the value of the viscosity coefficient, thanks to the implicit discretization of the viscous terms. Therefore, the three simulations ran for the same amount of computational time.

\begin{figure}[!htbp]
	\begin{center}
		\begin{tabular}{ccc} 
			\includegraphics[width=0.33\textwidth]{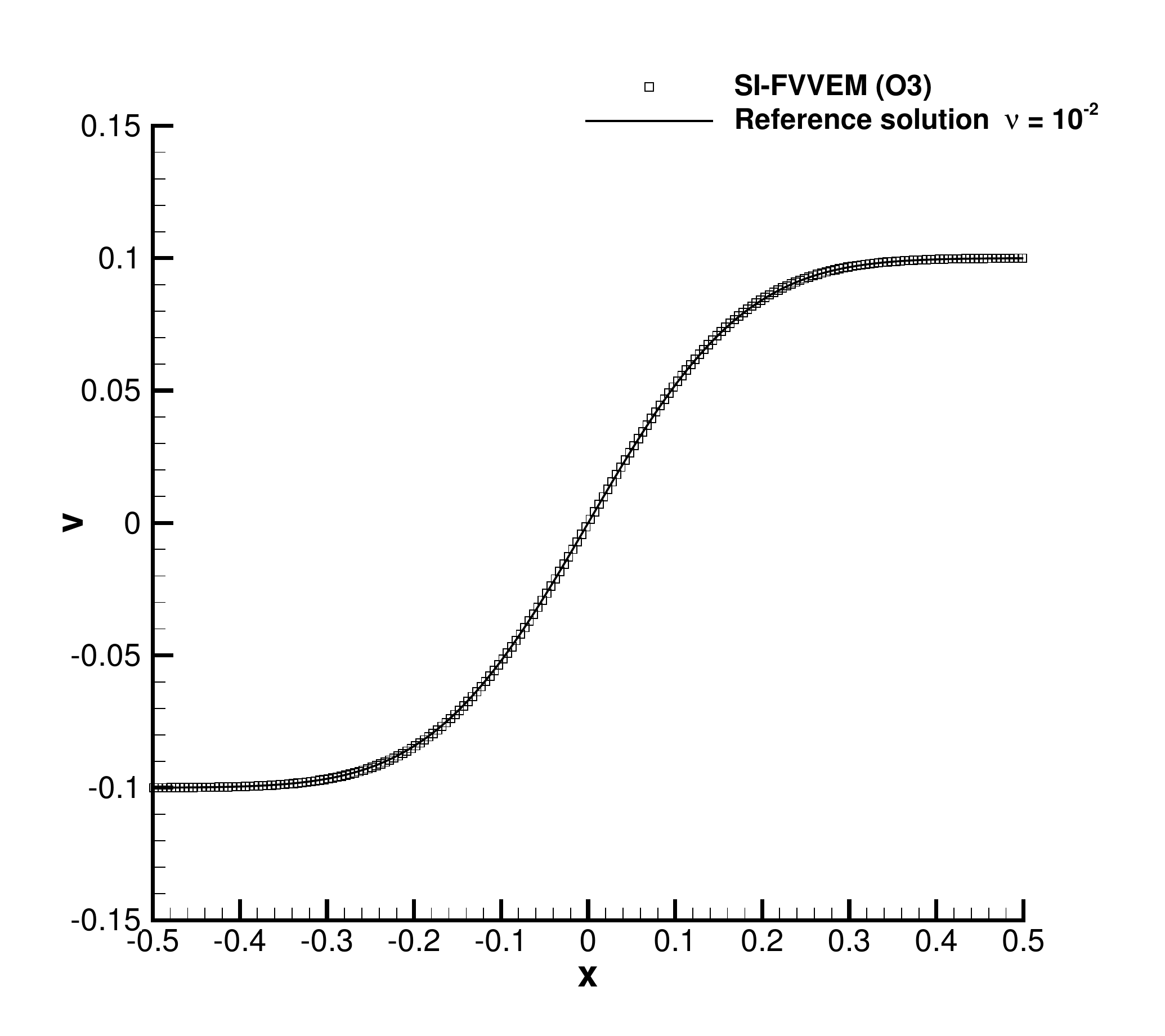} & 
			\includegraphics[width=0.33\textwidth]{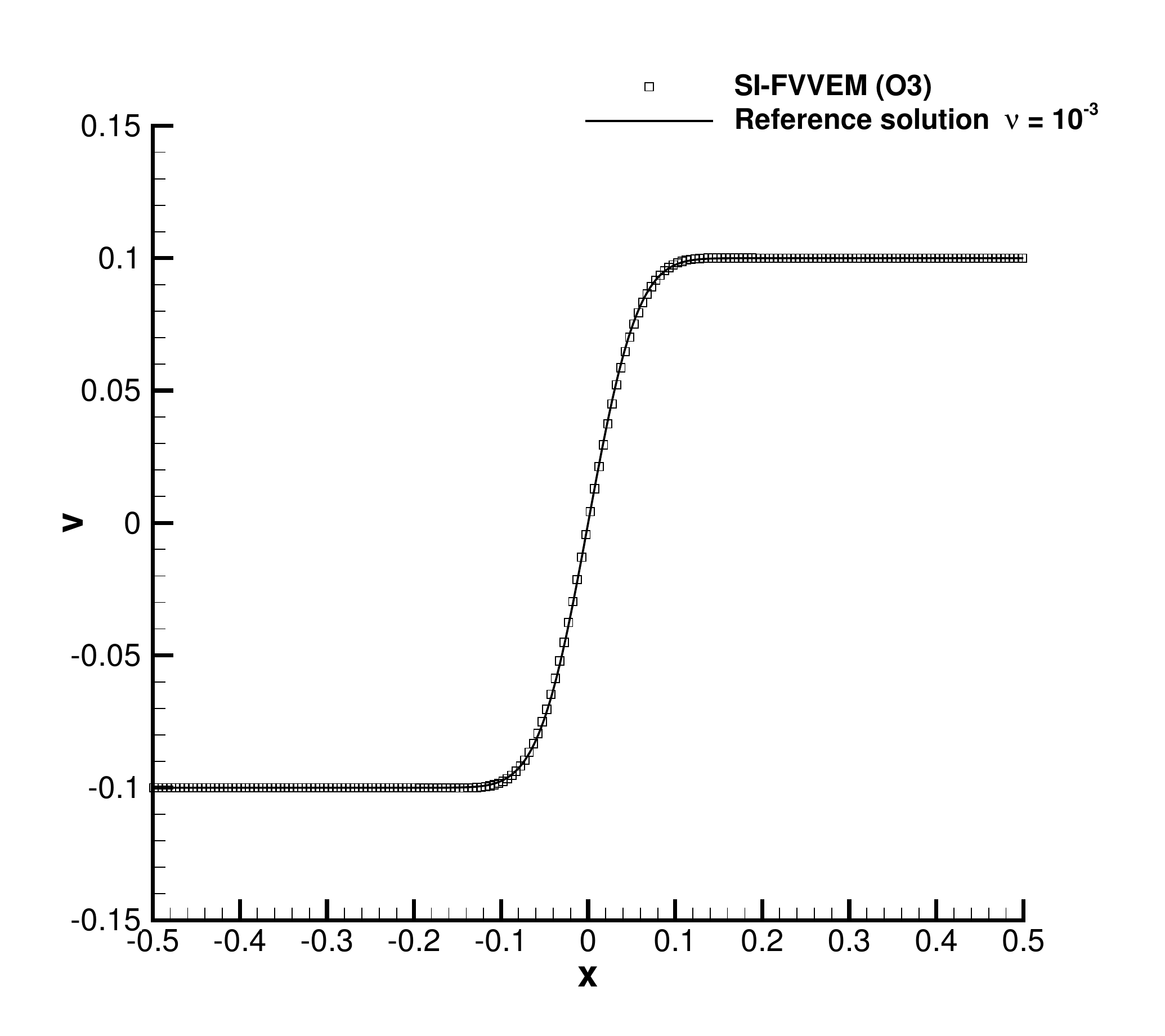} &  			\includegraphics[width=0.33\textwidth]{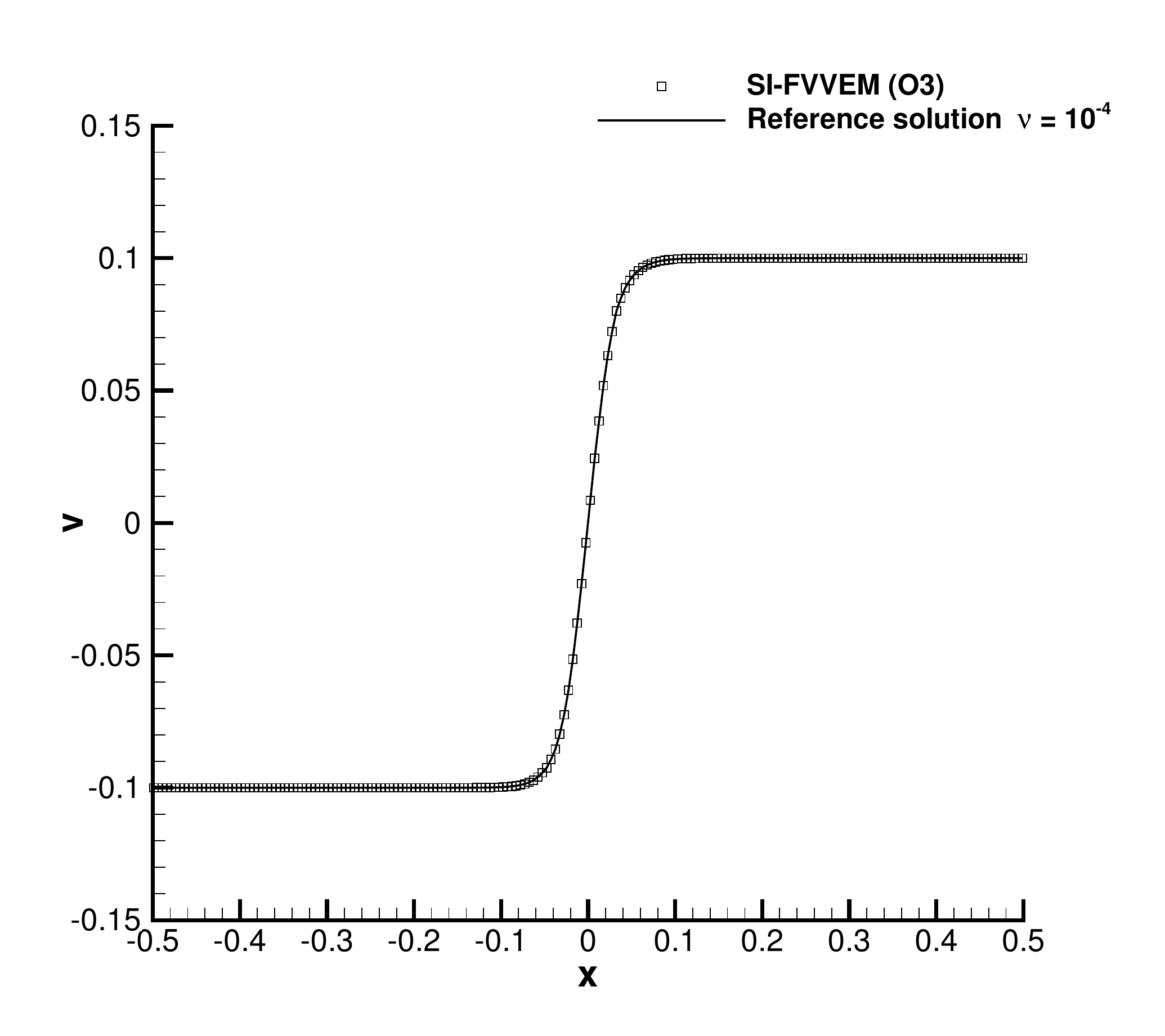} \\
		\end{tabular}
		\caption{First problem of Stokes at time $t_f=1$. Comparison against the analytical solution with viscosity $\nu=10^{-2}$ (left), $\nu=10^{-3}$ (middle) and $\nu=10^{-4}$ (right). The numerical solution is interpolated with 200 equidistant points at $y=0$.}
		\label{fig.stokes1}
	\end{center}
\end{figure}

%
\subsection{Oscillatory viscous flow between two flat plates (INS)}
Here, we consider a viscous incompressible flow between two flat plates. The solution solves the Stokes equations and was derived by Womersley in \cite{womersley1955method}. The computational domain is the square $\Omega=[x_L;x_R]^2=[-0.5;0.5]^2$ with wall boundaries in the $y-$direction and periodic sides elsewhere. The employed mesh has a total number of $N_P = 14368$ polygons with a characteristic size $h = 1/80$. The non-linear convective terms are neglected according to \cite{Loudon1998}, hence obtaining the analytical solution for the fluid velocity and pressure gradient as a function of time and the distance from the plate:
\begin{equation}
	u = \frac{A}{i \, \omega} \left[ 1- \frac{\cosh \left(\alpha_W \sqrt{i} \, y/R\right)}{\cosh \left(\alpha_W \sqrt{i}\right)} \right] \, e^{i \omega t},  \qquad
	\frac{\partial p}{\partial x} = \frac{p(x_R)-p(x_L)}{L} = A \, e^{i \omega t}.
	\label{eqn.Womersley_exact}
\end{equation}
Here, $\alpha_W=R\sqrt{\omega/\nu}$ is the Womersley number, with $R$ being the half distance between the two plates and $\omega=2\pi$ denoting the frequency of the oscillations. Moreover, the amplitude of the sinusoidal pressure gradient is set to $A=1$, the imaginary unit is addressed as usual with $i=\sqrt{-1}$, and $L=x_R-x_L$ is the total length of the computational domain along the $x-$direction. The pressure gradient is imposed at the aid of a source term that is discretized explicitly in time as already done in \cite{boscheri2021efficient}. The viscosity coefficient is $\nu=2 \cdot 10^{-2}$ and the final time of the simulation is $t_f = 1$. To properly follow the oscillatory dynamics, a time step $\Delta t = 0.01$ is imposed. Figure \ref{fig.womersley} shows a 3D view of the obtained numerical horizontal velocity $u$ for time instances $t \in \{ 0.2, 0.4, 0.6, 0.8 \}$. We also compare the numerical horizontal velocity profile along the cut at $x = 0$ against the exact solution \eqref{eqn.Womersley_exact}, achieving and excellent agreement for all output times. 

\begin{figure}[!htbp]
	\begin{center}
		\begin{tabular}{cc} 
			\includegraphics[width=0.47\textwidth]{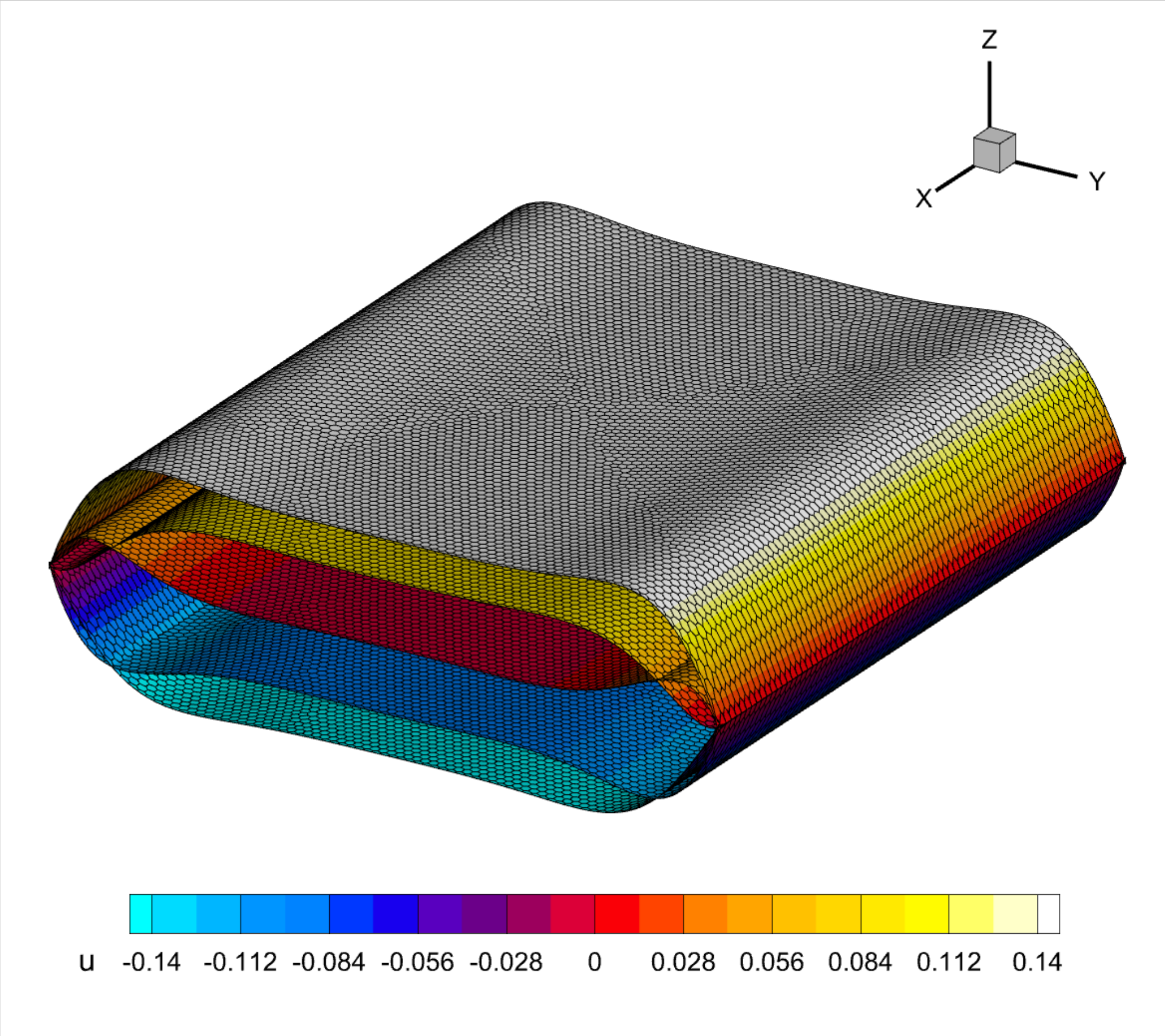} &  			\includegraphics[width=0.47\textwidth]{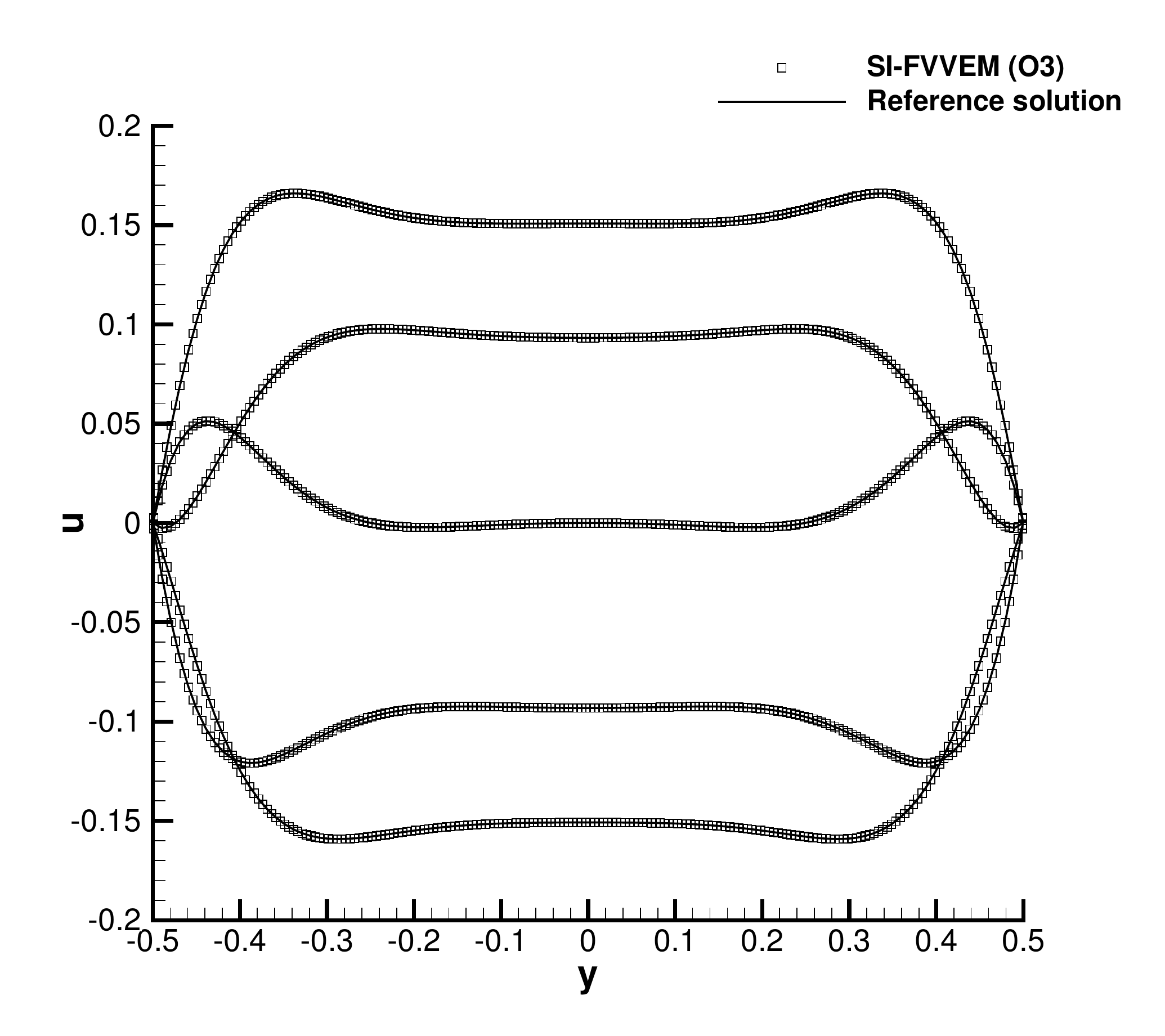} \\
		\end{tabular}
		\caption{Womersley flow with $\nu = 2 \cdot 10^{-2}$. Left: 3D view of the horizontal velocity $u$ over the entire computational domain at different output times. Right: numerical solution interpolated along 200 equidistant spatial points at $y=0$ compared with the exact solution at output times $t = 0.2$, $t = 0.4$, $t = 1.0$, $t = 0.6$, $t = 0.8$ (from the bottom to the top).}
		\label{fig.womersley}
	\end{center}
\end{figure}

%
\subsection{Double shear layer (INS)}

This test, originally introduced in \cite{bell1989second}, is concerned with the incompressible flow of a double shear layer solving the Navier-Stokes equations. The physical dynamics consists in a horizontal flow jet of a fluid given by a small vertical pressure gradient balancing the convective forces. As a consequence, several vortexes develop during the simulation. The computational domain is the square $\Omega = [0;1]^2$ with periodic boundaries. At the initial time the following pressure and velocity field are prescribed:
\begin{equation} \label{eqn.double_shear}
	p = 1, \qquad u = \begin{cases}
		\tanh( \theta (y - 1/4) ), & y \leq 1/2 \\
		\tanh( \theta (3/4 - y) ), & y > 1/2
	\end{cases}, \qquad v = \delta \sin( 2 \pi x ),	
\end{equation}
with parameters $\theta = 30$ and $\delta = 0.05$. A Reynolds number $\Rey=5000$ is chosen, and the final time of the simulation is fixed at $t_f = 1.8$. The computational grid counts $N_P=15684$ polygonal cells with $h=1/100$. Figure \ref{fig.DSL} depicts the vorticity magnitude at different time instances. Even in this case, numerical results are qualitatively in very good agreement with reference solutions available in the literature \cite{dumbser2016high, boscheri2021efficient, boscheri2021structure}. 

\begin{figure}[!htbp]
	\begin{center}
		\begin{tabular}{cc} 
			\includegraphics[width=0.47\textwidth]{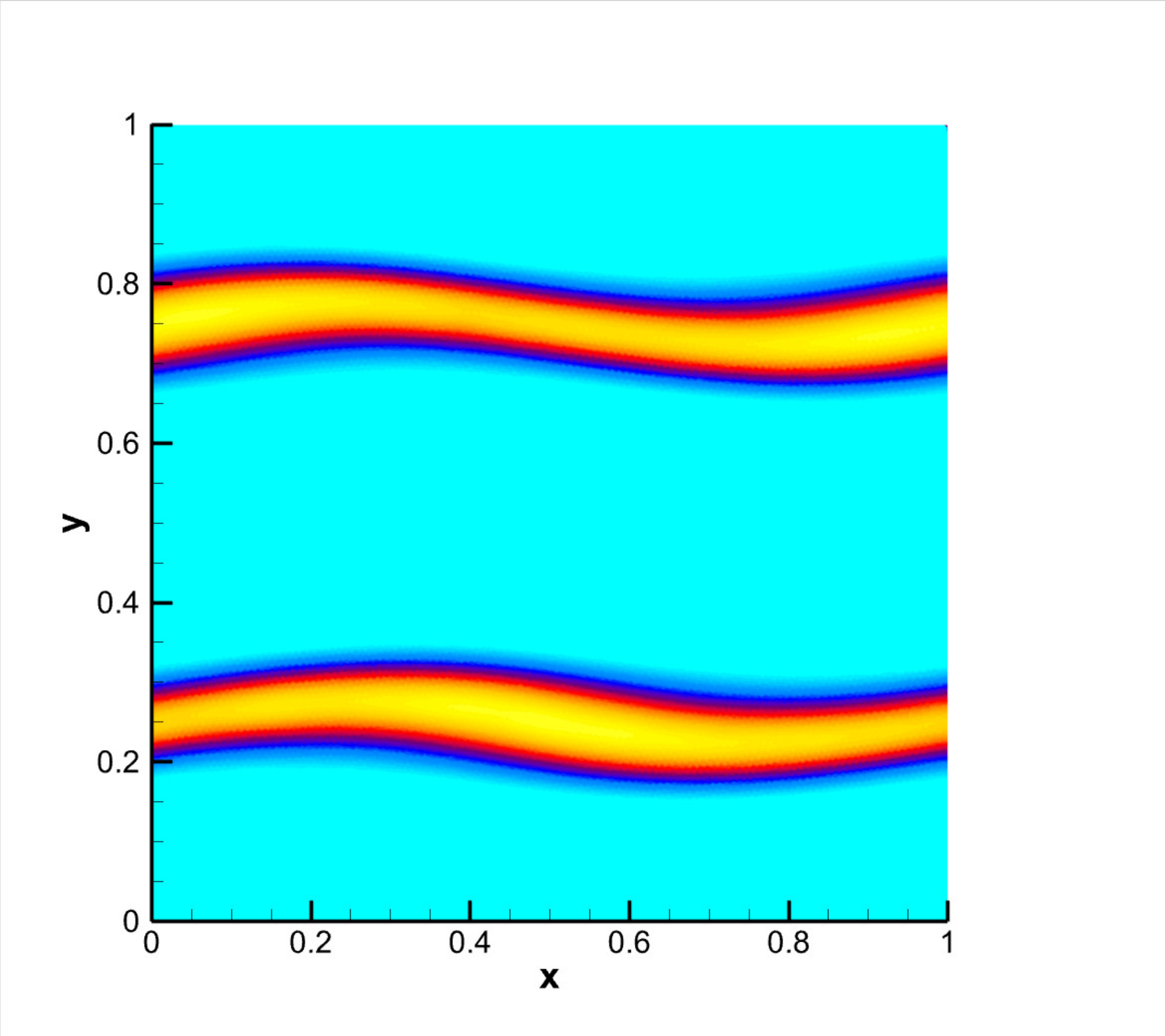} &  			\includegraphics[width=0.47\textwidth]{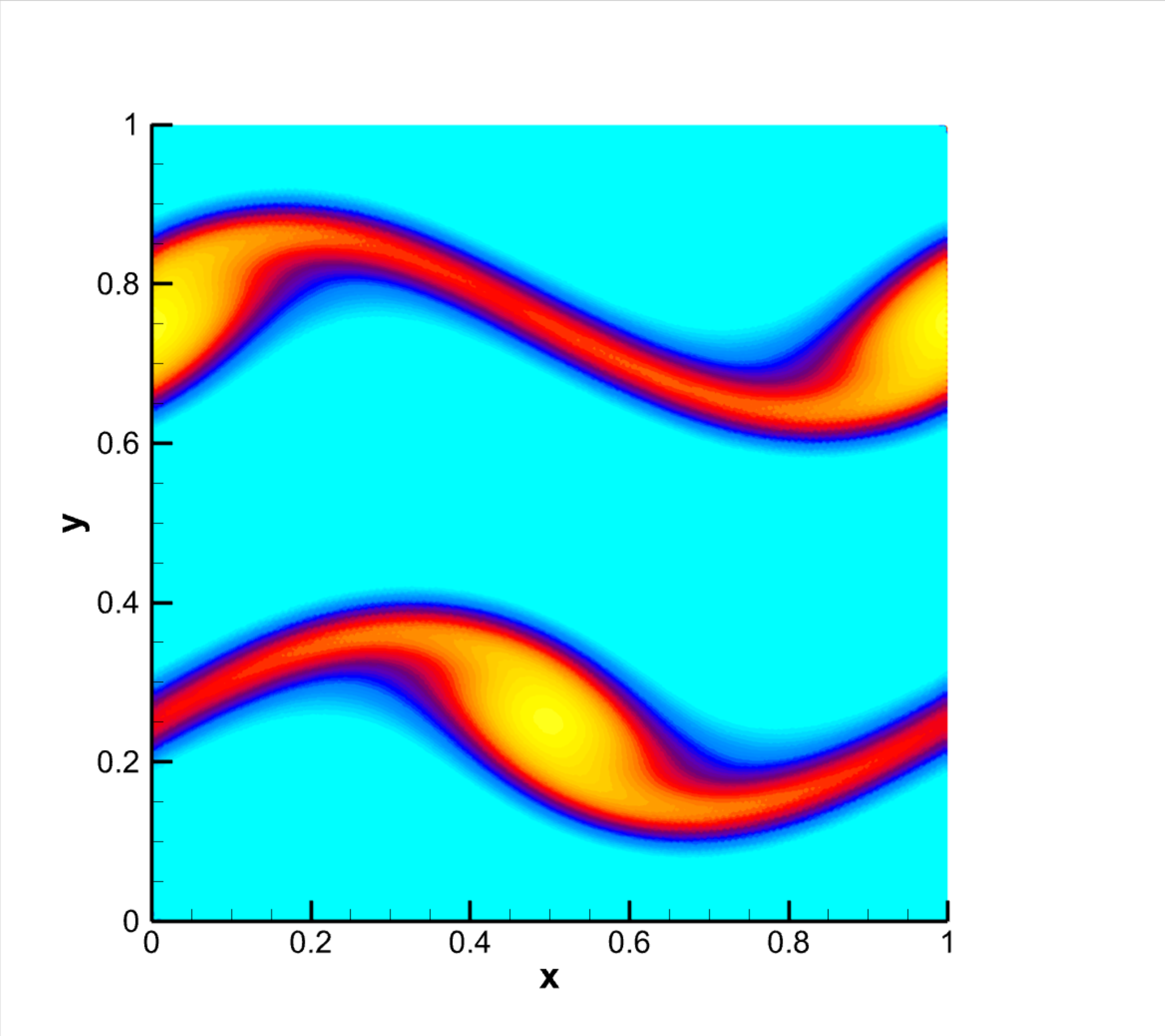} \\
			\includegraphics[width=0.47\textwidth]{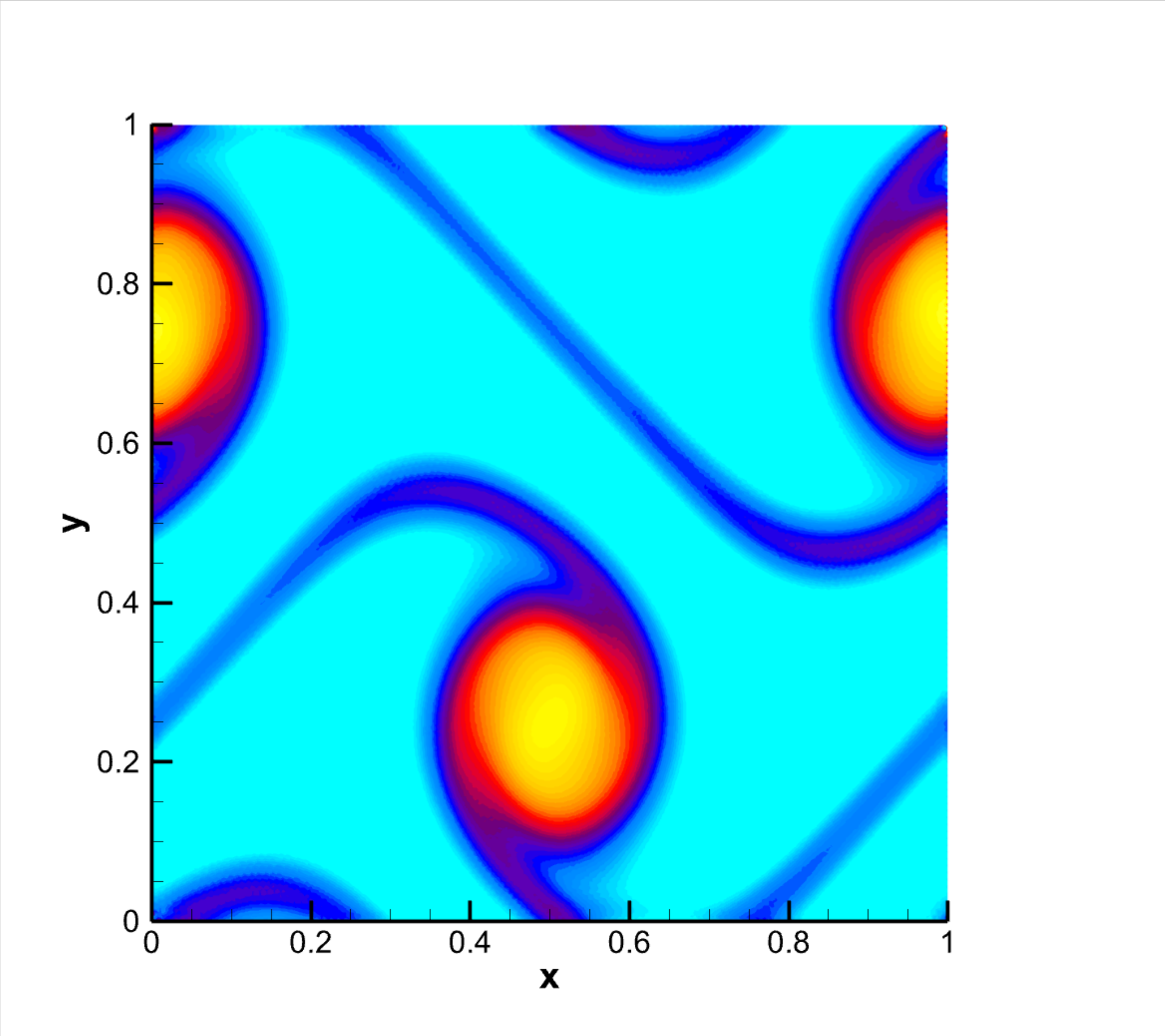} &  			\includegraphics[width=0.47\textwidth]{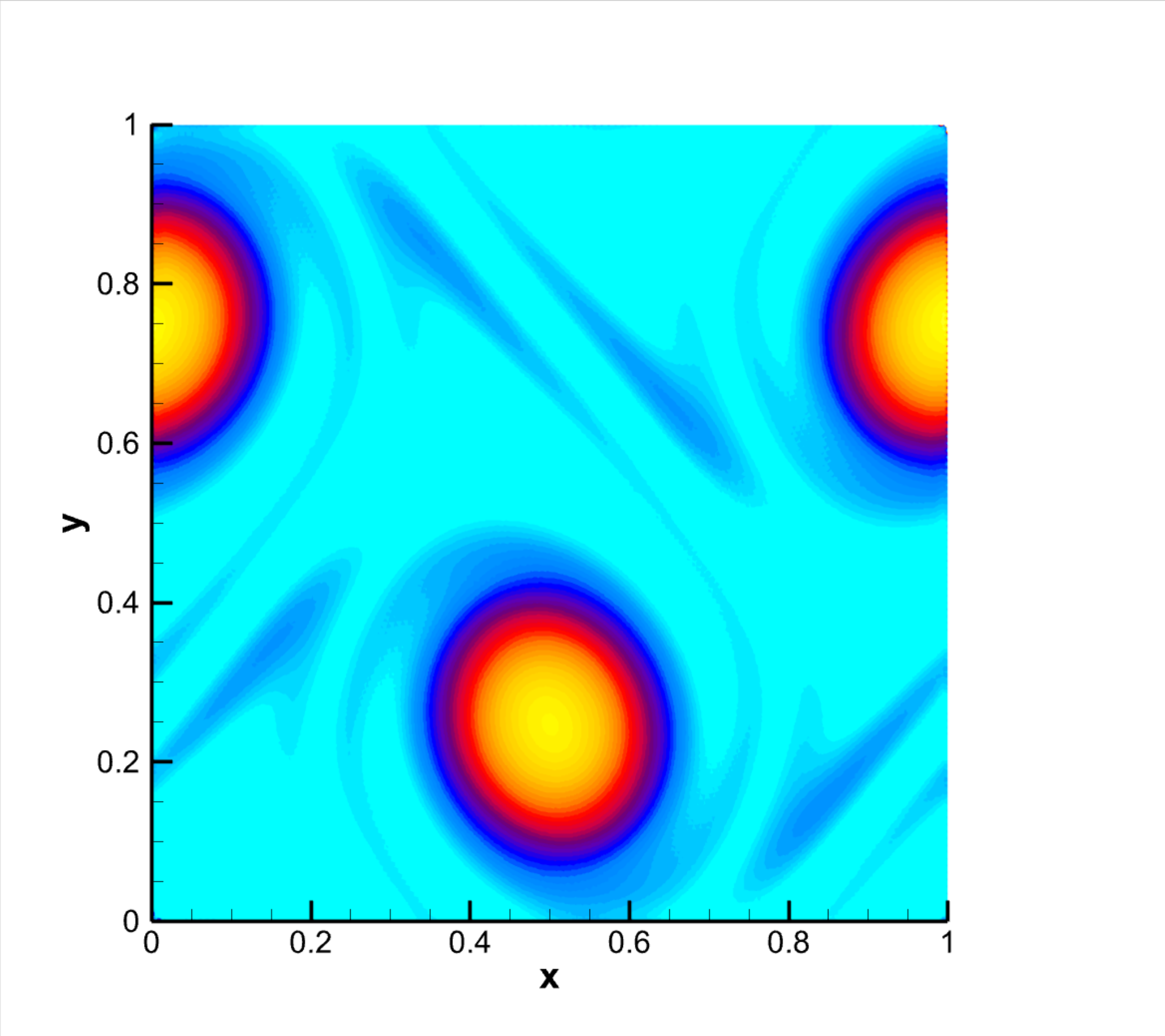} \\
		\end{tabular}
		\caption{Double shear layer. Vorticity magnitude at output times $t=0.4$ (top left), $t=0.8$ (top right), $t=1.2$ (bottom left) and $t=1.8$ (bottom right). The results are plotted with 41 contours in the interval $[0;26]$.}
		\label{fig.DSL}
	\end{center}
\end{figure}

%
\subsection{Lid-driven cavity flow (INS)}
We further test the novel SI-FVVEM scheme on the classical benchmark of the lid driven cavity flow in the computational domain $\Omega = [-0.5, 0.5]^2$. In this case, the target is to find the steady hydrodynamics state for a fluid initially at rest, i.e., $p = 0$ and $\vv = \mathbf{0}$. Wall boundary conditions are defined on the vertical sides ($x=\pm0.5$) and at the bottom ($y=-0.5$) of the domain, while on the top side ($y=0.5$) the velocity field $\vv = (1, 0)$ is imposed. The final time of the simulation is $t_f = 25$. Two different values of the Reynolds number are considered, namely $\Rey = 100$ and $\Rey=400$. All the simulations are run using the third order version of the scheme in space, while relying on a cheaper first order time discretization, because the goal is to catch the steady state. Figure \ref{fig.Cavity} shows the distribution of the horizontal velocity at the final time as well as the stream-traces of the velocity field, which permit to notice the generation of small vortical flows with a counter orientation with respect to the main vortex led by the cavity. We also show the numerical velocity components $u$ and $v$ along the cuts $x = 0$ and $y = 0$, respectively, comparing them with the results found in \cite{Ghia1982}. From the plots, it is possible to appreciate a good matching between our numerical results and data available in the literature for both Reynolds numbers.    

\begin{figure}[!htbp]
	\begin{center}
		\begin{tabular}{cc} 
			\includegraphics[width=0.47\textwidth]{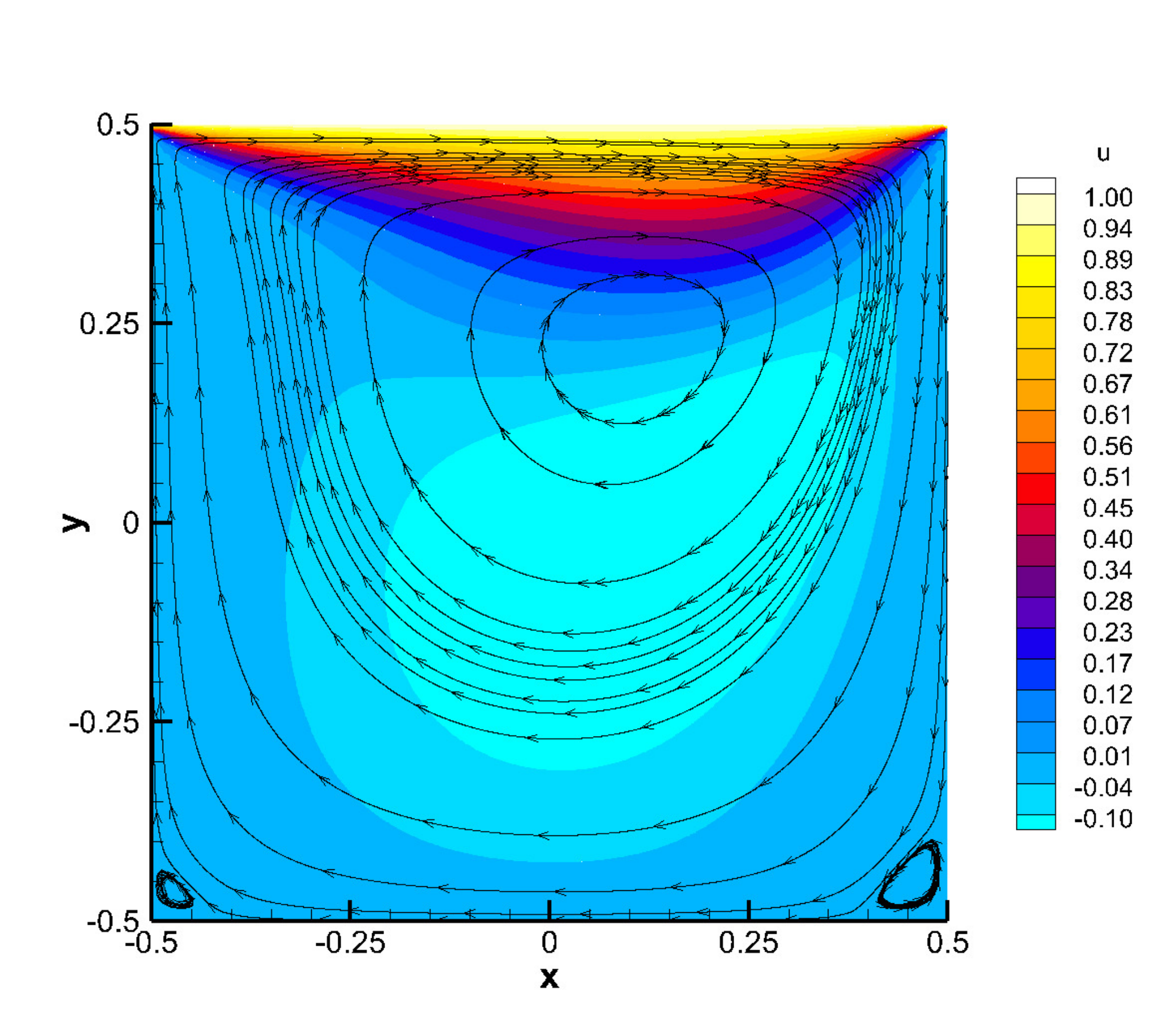} &  			\includegraphics[width=0.47\textwidth]{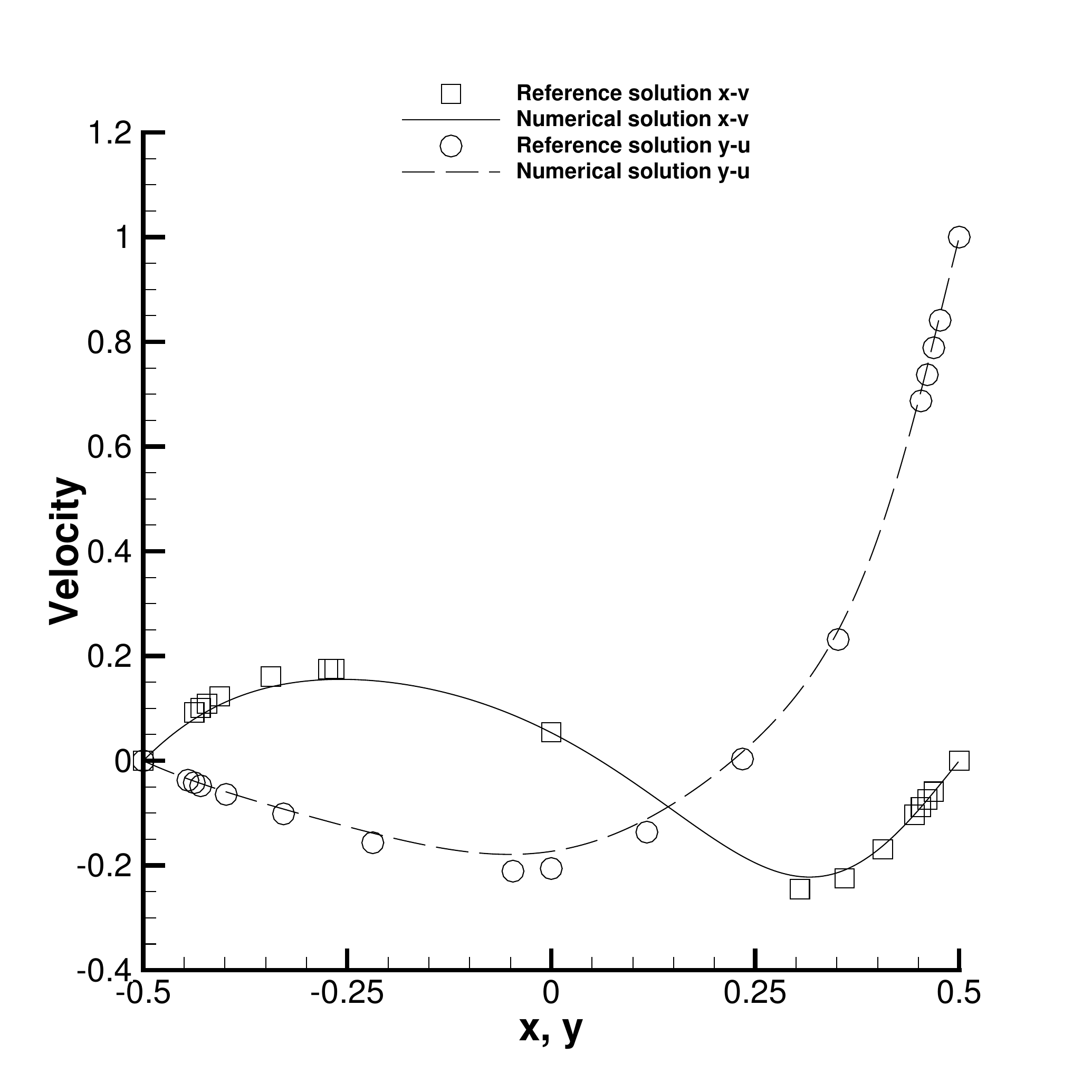} \\
			\includegraphics[width=0.47\textwidth]{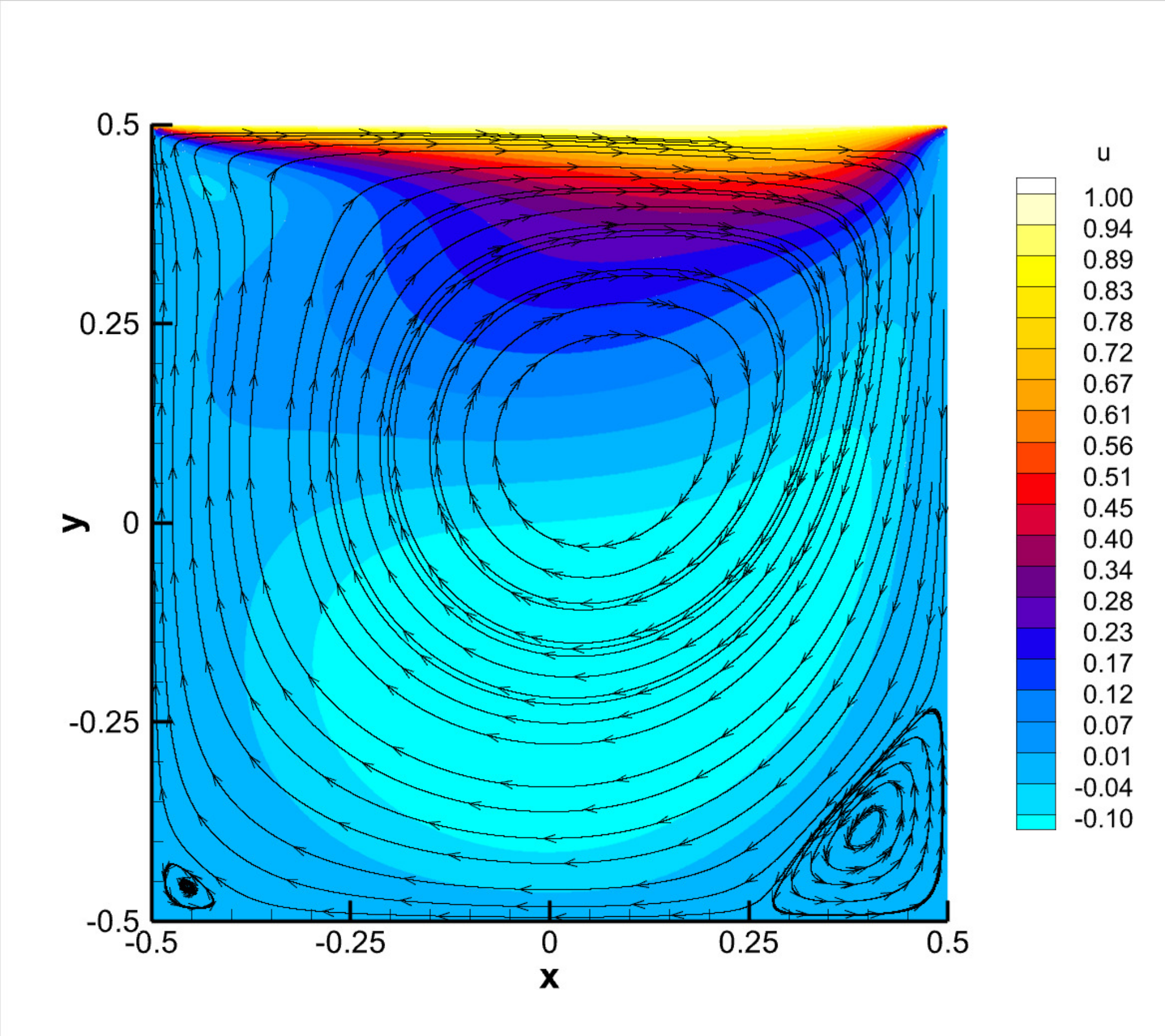} &  			\includegraphics[width=0.47\textwidth]{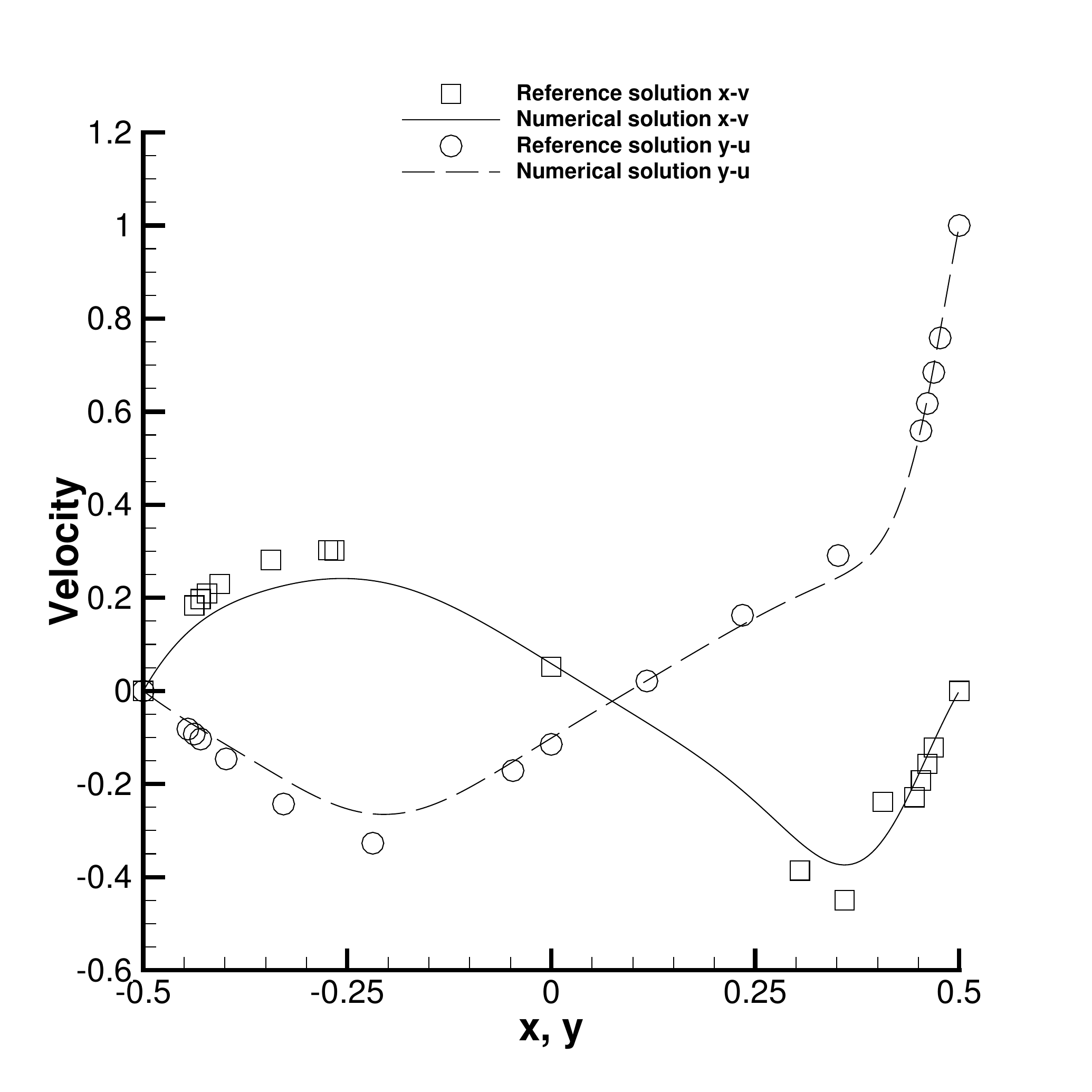} \\
		\end{tabular}
		\caption{Lid-driven cavity flow at time $t=25$ with a Reynolds number $\text{Re}=100$ (top row) and $\text{Re}=400$ (bottom row). Left: horizontal velocity contours and streamlines. Right: comparison with the reference solution of \cite{Ghia1982} for the velocity components $u$ and $v$ along the lines $y=0$ and $x=0$. }
		\label{fig.Cavity}
	\end{center}
\end{figure}

%
\subsection{Laminar flow over a cylinder (INS)}

We finally solve the problem related to a laminar flow over a circular cylinder \cite{TAVELLI2015235}, which generates a von Karman vortex that is shed behind the obstacle. The computational domain is defined as $\Omega=[-10;40]\times[-7;7]$, in which a cylinder of radius $r_c=1$ is located with centre at $\xx_c=(0,0)$. The computational mesh is made of $N_P=25390$ Voronoi cells with characteristic mesh size $h=0.25$. On the left side of the domain we prescribe an inflow velocity of $\vv=(0.5,0)$, while the remaining sides are assigned with transmissive boundary conditions. On the obstacle, no-slip conditions are used. A constant pressure $p=1$ is imposed at the initial time, and the flow is characterized by a Reynolds number of $\Rey=100$. A plot of the resulting stream-traces is reported in Figure \ref{fig.CylinderINS_stream} at different output times, showing the generation of vortical flow patterns past the cylinder. Figure \ref{fig.CylinderINS} depicts the vorticity and the horizontal velocity distribution at the final time $t_f=300$, clearly identifying two main vortexes departing from the obstacle that further generates the Von Karman street.

\begin{figure}[!htbp]
	\begin{center}
		\begin{tabular}{cc} 
			\includegraphics[width=0.47\textwidth]{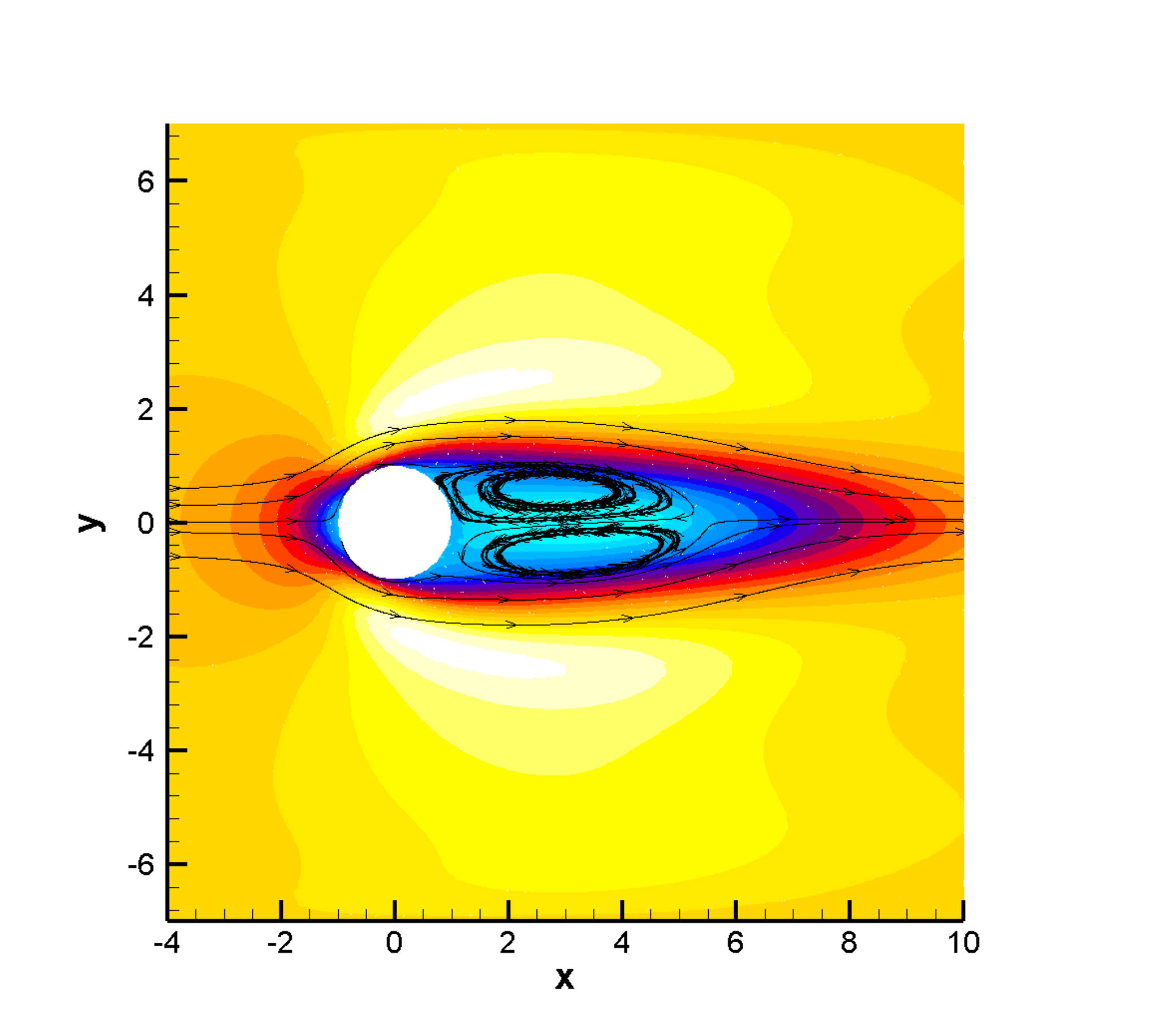} &
			\includegraphics[width=0.47\textwidth]{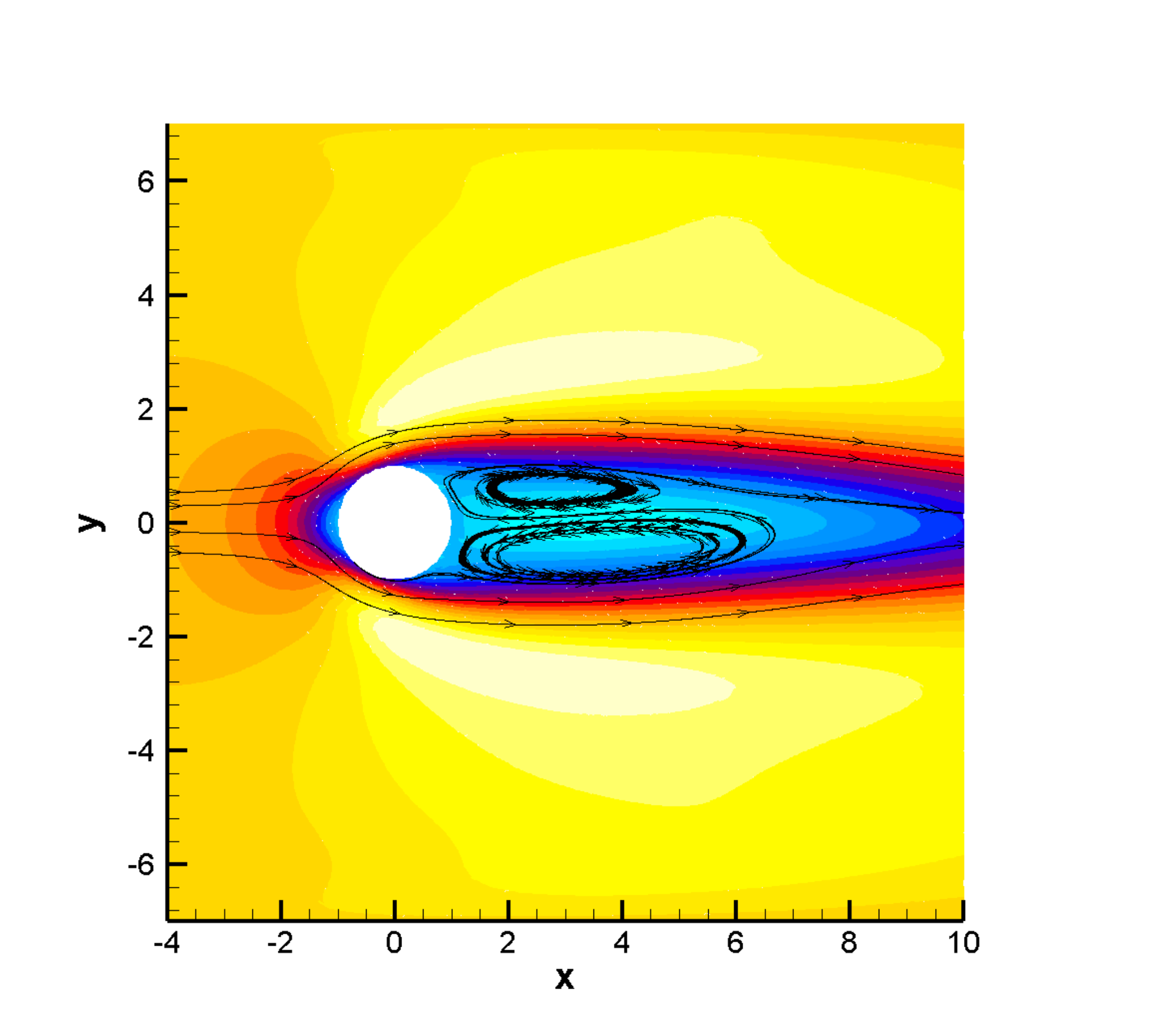} \\
			\includegraphics[width=0.47\textwidth]{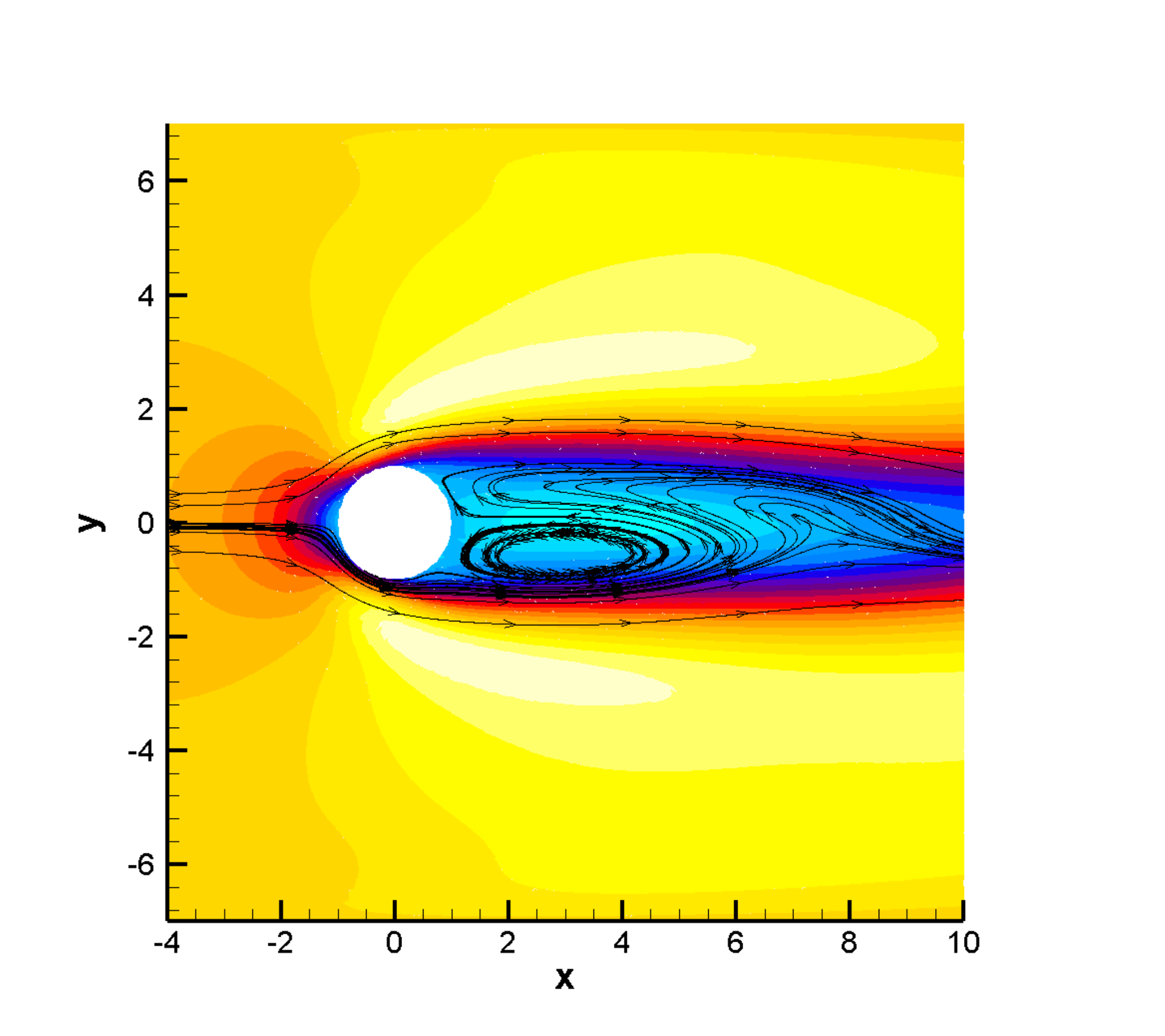} &
			\includegraphics[width=0.47\textwidth]{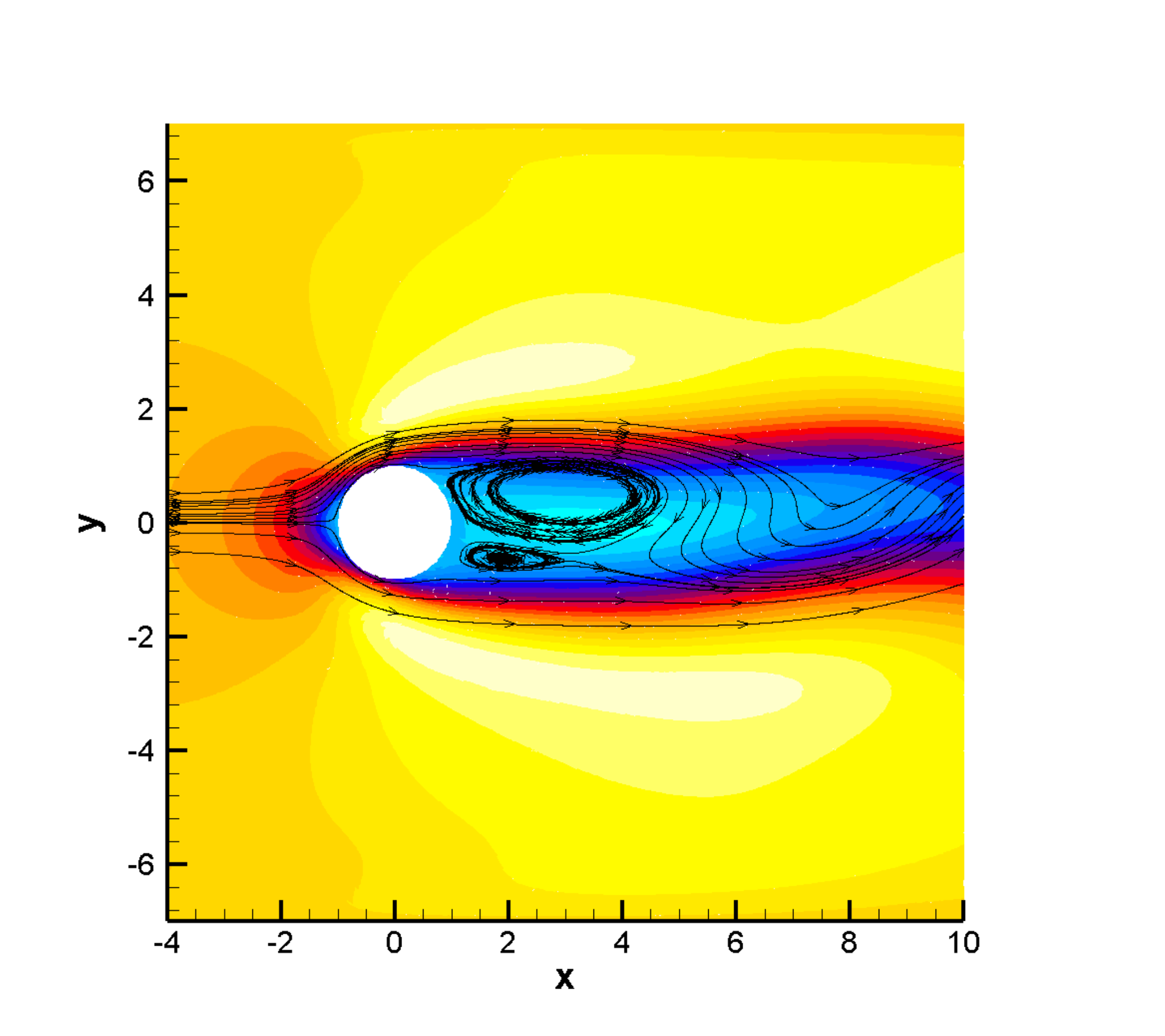} \\
		\end{tabular}
		\caption{Laminar flow over a cylinder. Stream-traces along the circular cylinder at times, $t=25$, $t=50$, $t=100$ and $t=200$ (from top left to bottom right). Color map of the horizontal velocity $u$ with 21 contour levels in the range $[-0.05;0.65]$.}
		\label{fig.CylinderINS_stream}
	\end{center}
\end{figure}

\begin{figure}[!htbp]
	\begin{center}
		\begin{tabular}{c} 
			\includegraphics[width=0.9\textwidth]{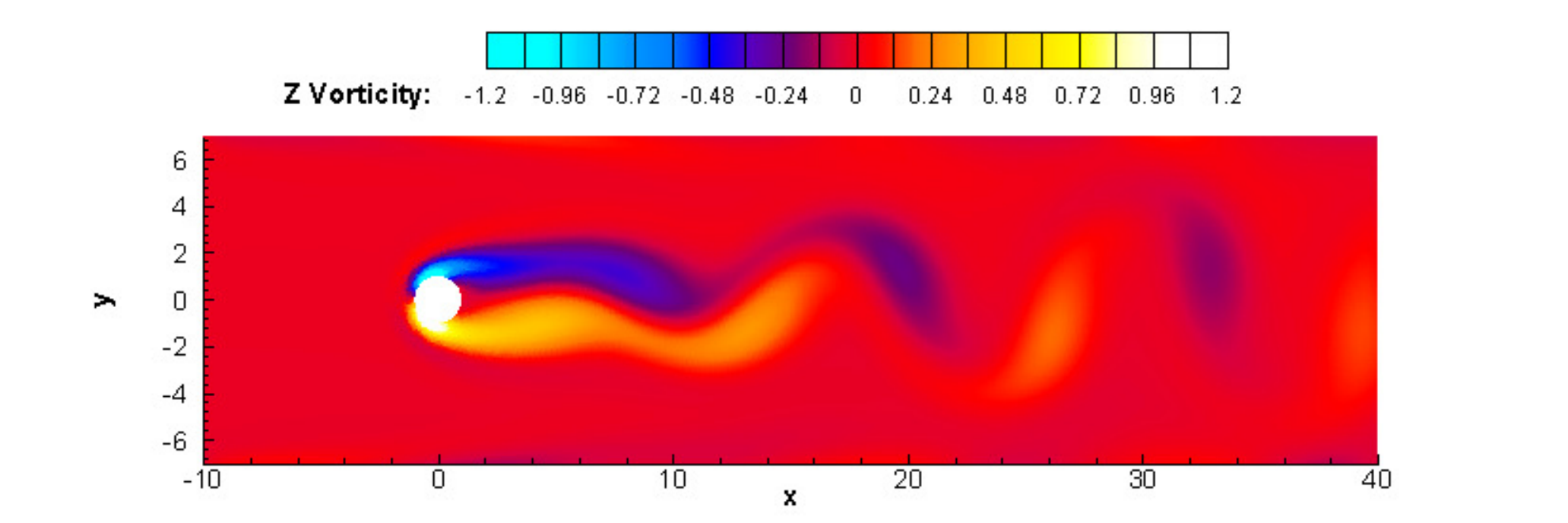} \\
			\includegraphics[width=0.9\textwidth]{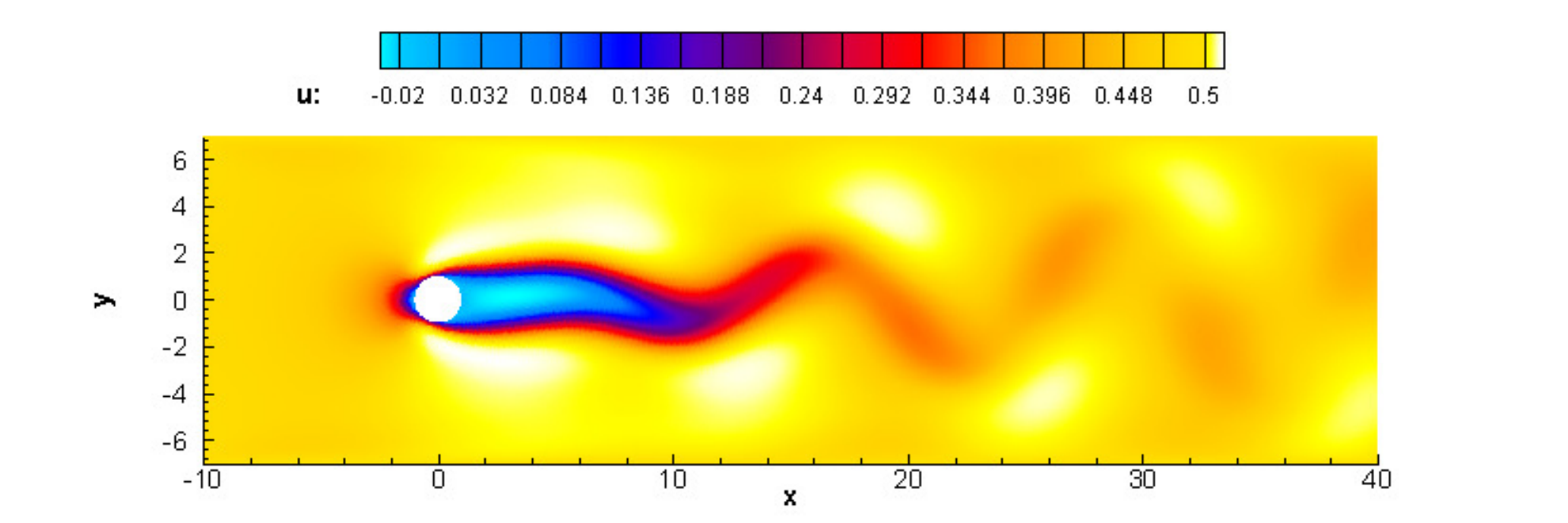} 
		\end{tabular}
		\caption{Laminar flow over a cylinder at time $t_f=300$. Map of the vorticity (top) and the horizontal velocity (bottom). }
		\label{fig.CylinderINS}
	\end{center}
\end{figure}

To give a quantitative interpretation, the same test case is also run using an inflow velocity of $\vv=(1,0)$, while keeping the same Reynolds number. The time evolution of the pressure is measured at the point $\xx=(15,0)$, from which we extract the associated frequency spectrum $f$. The results are depicted in Figure \ref{fig.CylinderPickPoint}, showing that we obtain a Strouhal number $\textit{St}=f r_c/u$ of $\textit{St}=0.1401$ for $u=0.5$ and $\textit{St}=0.1680$ for $u=1$ which is in reasonable good agreement with the value of $\textit{St}=0.1649$ reported in \cite{Qu2013}.

\begin{figure}[!htbp]
	\begin{center}
		\begin{tabular}{cc} 
			\includegraphics[width=0.49\textwidth]{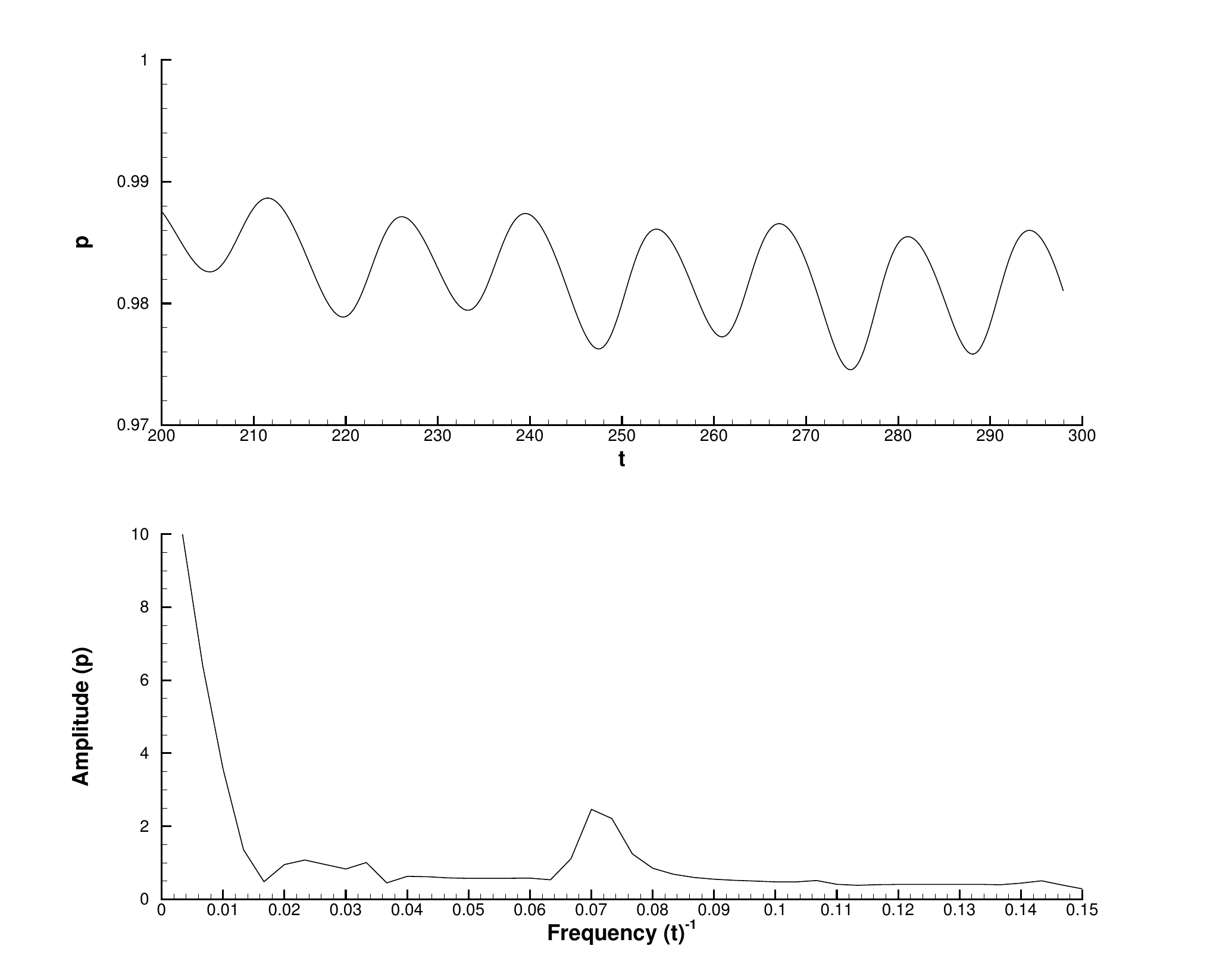} &
			\includegraphics[width=0.49\textwidth]{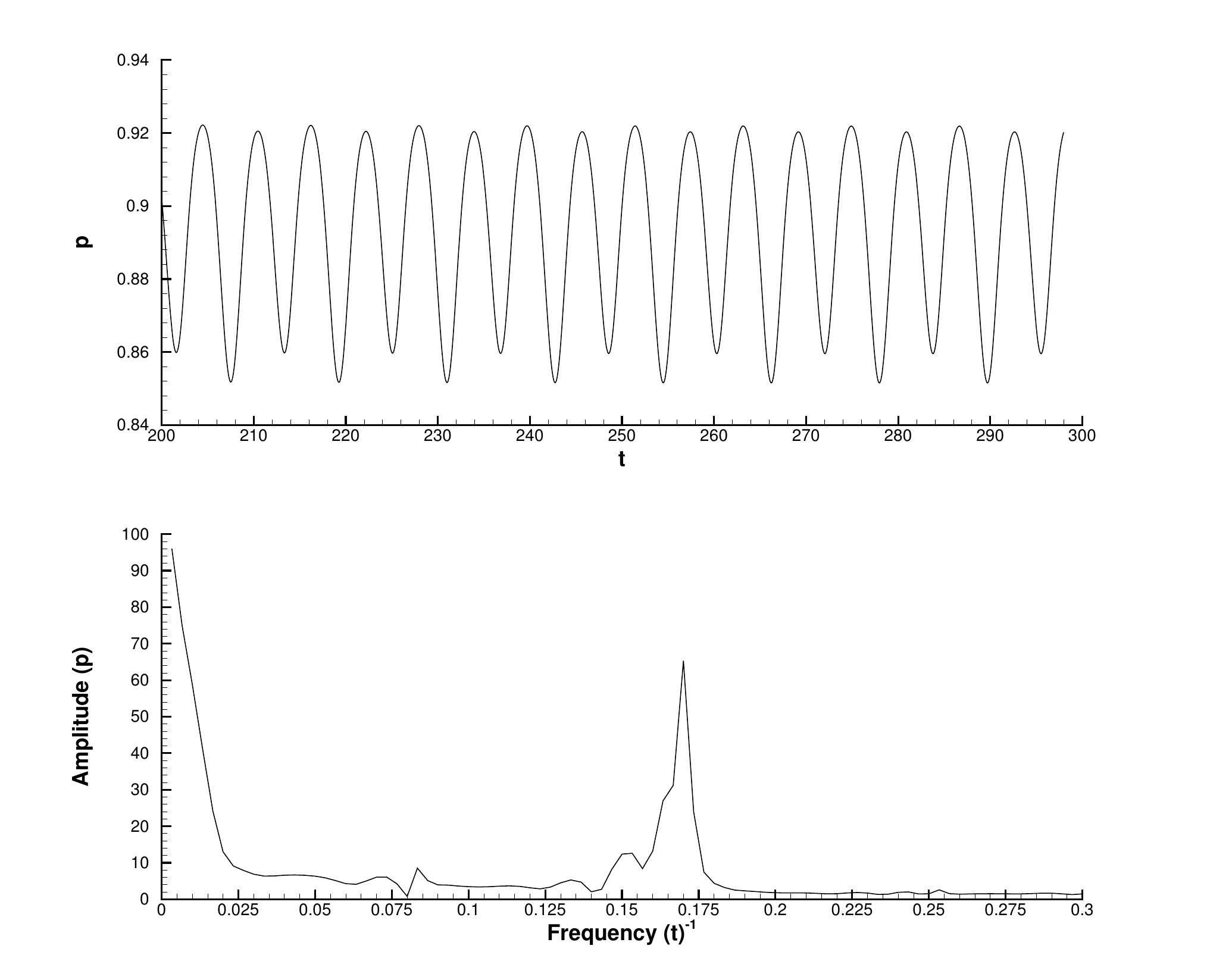} \\
		\end{tabular}
		\caption{Laminar flow over a cylinder at time $t_f=300$. Pressure and related frequency modes measured at $\xx=(15,0)$ for a flow with $\Rey=100$ and inflow velocity of $u=0.5$ (left) and $u=1$ (right).}
		\label{fig.CylinderPickPoint}
	\end{center}
\end{figure}

\section{Conclusions} \label{sec.concl}
The purpose of this work was to devise a new family of semi-implicit schemes on general polygonal meshes for the solution of multi-scale PDE. To this aim, the fast and the slow scales of the mathematical models are separated, so that a conservative shock capturing finite volume scheme is used to deal with the non-linear explicit convective terms, while the virtual element method is employed for discretizing the implicit sub-system. \textit{Ad hoc} projector operators have been designed to transfer the numerical solution from one approximation space to another. A high order numerical solution in space and time is obtained by a CWENO spatial reconstruction and an IMEX Runge-Kutta time integrator. Moreover, high order of accuracy is granted in the implicit solver by the VEM paradigm. Well-balancing and asymptotic-preservation properties of the first order semi-discrete scheme have been demonstrated. As prototype examples, we have considered two non-linear systems of hyperbolic PDE that govern the motion of incompressible fluids, namely the shallow water equations and the incompressible Navier-Stokes model. The novel numerical scheme have been thoroughly tested against a wide range of academic benchmarks, showing accuracy and robustness capabilities. The VEM approach naturally allows to achieve high order space discretizations on arbitrary shaped meshes, that is here combined with a robust finite volume solver for treating shock waves and strong discontinuities. The resulting numerical method is stable, accurate and robust, being suitable for the treatment of complex geometries as well.

In the future we plan to extend the approach presented in this work to compressible flows, along the lines of \cite{SICNS22}. Discontinuous Galerkin schemes in the framework of the virtual element space will also be investigated in order to devise a novel class of nonconforming virtual element methods. Finally, the development of purely VEM schemes in the context of semi-implicit time marching techniques is also foreseen, where even the convective sub-system is solved at the aid of the VEM paradigm, supplemented with suitable limiters to ensure stability and conservation at the same time.

\section*{Acknowledgments}

WB and MGC received financial support by Fondazione Cariplo and Fondazione CDP (Italy) under the project No. 2022-1895. WB and GB are members of the GNCS-INdAM (\textit{Istituto Nazionale di Alta Matematica}) group. AC and GB have been partially funded by the call ``Bando Giovani anno 2022 per progetti di ricerca finanziati con il contributo 5x1000 anno 2020" of the University of Ferrara. This work was partially carried out at the Institute des Mathematiques de Bordeaux (IMB, Bordeaux-France) during the visiting program of WB and MGC.

\appendix

\section{IMEX Runge-Kutta schemes} \label{app.IMEX}
The Butcher tableau for the IMEX Runge-Kutta schemes used in this work are reported hereafter. They have been derived in \cite{PR_IMEX,PR_IMEXHO,Boscarino22} and each IMEX scheme is described with a triplet $(s,\tilde{s},p)$ which characterizes the number $s$ of stages of the implicit method, the number $\tilde{s}$ of stages of the explicit method and the order $p$ of the resulting scheme. The acronym SA stands for Stiffly Accurate, LS implies L-Stability, while DIRK refers to Diagonally Implicit Runge-Kutta schemes.

\begin{itemize}
	\item SP(1,1,1)
	
	\begin{equation}
		\begin{array}{c|c}
			0 & 0 \\ \hline & 1
		\end{array} \qquad
		\begin{array}{c|c}
			1 & 1 \\ \hline & 1
		\end{array}
		\label{eqn.IMEX1}
	\end{equation}
	
	\item LSDIRK2(2,2,2) \hspace{0.2cm} $\gamma=1-1/\sqrt{2}$, \hspace{0.05cm} $\beta=1/(2\gamma)$
	
	\begin{equation}
		\begin{array}{c|cc}
			0 & 0 & 0 \\ \beta & \beta & 0 \\ \hline & 1-\gamma & \gamma
		\end{array} \qquad
		\begin{array}{c|cc}
			\gamma & \gamma & 0 \\ 1 & 1-\gamma & \gamma \\ \hline & 1-\gamma & \gamma
		\end{array}
		\label{eqn.IMEX2}
	\end{equation}
	
	\item SA DIRK (3,4,3) \hspace{0.2cm} $\gamma=0.435866$
	
	\begin{equation}
		\begin{array}{c|cccc}
			0 & 0 & 0 & 0 & 0 \\ \gamma & \gamma & 0 & 0 & 0 \\ 0.717933 & 1.437745 & -0.719812 & 0 & 0 \\ 1 & 0.916993 & 1/2 & -0.416993 & 0 \\ \hline  & 0 & 1.208496 & -0.644363 & \gamma
		\end{array} \qquad
		\begin{array}{c|cccc}
			\gamma & \gamma & 0 & 0 & 0 \\ \gamma & 0 & \gamma & 0 & 0 \\ 0.717933 & 0 & 0.282066 & \gamma & 0  \\ 1 & 0 & 1.208496 & -0.644363 & \gamma \\ \hline  & 0 & 1.208496 & -0.644363 & \gamma
		\end{array}
		\label{eqn.IMEX3}
	\end{equation}

\end{itemize} 
Notice that the first order scheme \eqref{eqn.IMEX1} corresponds to the implicit Euler method.

%
\bibliographystyle{elsarticle-num}
\bibliography{PAPER_FVVEM}

\end{document}